\documentclass[11pt, reqno]{amsart}

\usepackage{amssymb,amsmath,amsthm,amsfonts,amscd,bbm,stmaryrd}
\usepackage{enumitem}
\usepackage{pdflscape}
\usepackage{caption}
\usepackage{bm}
\usepackage{comment}

\usepackage{ifpdf}
\ifpdf
\usepackage[pdftex]{graphicx}
\else
\usepackage[dvips]{graphicx}
\fi
\usepackage[all]{xy}
\usepackage{hyperref}
\usepackage{xcolor}
\usepackage{graphicx}
\usepackage[outdir=./]{epstopdf}
\usepackage{enumitem}
\usepackage{tensor}
\usepackage{tikz-cd}
\usepackage{multicol}
\usepackage{mathtools}
\usepackage{tocvsec2}
\usepackage{bbm}
\usepackage{longtable}
\usepackage{fullpage}

\newcommand{\doi}[1]{\href{https://dx.doi.org/#1}{{\small DOI:#1}}}

\newcommand{\googlebooks}[1]{(preview at \href{https://books.google.com/books?id=#1}{google books})}

\newcommand{\numdam}[1]{}

\usepackage{scalerel,stackengine}
\stackMath
\newcommand\Widehat[1]{%
\savestack{\tmpbox}{\stretchto{%
  \scaleto{%
    \scalerel*[\widthof{\ensuremath{#1}}]{\kern-.6pt\bigwedge\kern-.6pt}%
    {\rule[-\textheight/2]{1ex}{\textheight}}
  }{\textheight}%
}{0.5ex}}%
\stackon[1pt]{#1}{\tmpbox}%
}

\usepackage{mathrsfs}
\DeclareMathAlphabet{\mathpzc}{OT1}{pzc}{m}{it}

\def\semicolon{;}
\def\applytolist#1{
    \expandafter\def\csname multi#1\endcsname##1{
        \def\multiack{##1}\ifx\multiack\semicolon
            \def\next{\relax}
        \else
            \csname #1\endcsname{##1}
            \def\next{\csname multi#1\endcsname}
        \fi
        \next}
    \csname multi#1\endcsname}

\def\calc#1{\expandafter\def\csname c#1\endcsname{{\mathcal #1}}}
\applytolist{calc}QWERTYUIOPLKJHGFDSAZXCVBNM;
\def\bbc#1{\expandafter\def\csname bb#1\endcsname{{\mathbb #1}}}
\applytolist{bbc}QWERTYUIOPLKJHGFDSAZXCVBNM;
\def\bfc#1{\expandafter\def\csname bf#1\endcsname{{\mathbf #1}}}
\applytolist{bfc}QWERTYUIOPLKJHGFDSAZXCVBNM;
\def\sfc#1{\expandafter\def\csname s#1\endcsname{{\sf #1}}}
\applytolist{sfc}QWERTYUIOPLKJHGFDSAZXCVBNM;
\def\fc#1{\expandafter\def\csname f#1\endcsname{{\mathfrak #1}}}
\applytolist{fc}QWERTYUIOPLKJHGFDSAZXCVBNM;

\usepackage{tikz}
\usepackage{tikz-cd}

\makeatletter
\def\fixtikzforbreqn#1#2{%
  \protected\edef#1{\noexpand\ifmmode\mathchar\the\mathcode`#2 \noexpand\else#2\noexpand\fi}%
}
\fixtikzforbreqn\tikz@nonactivesemicolon;
\fixtikzforbreqn\tikz@nonactivecolon:
\fixtikzforbreqn\tikz@nonactivebar|
\fixtikzforbreqn\tikz@nonactiveexlmark!
\makeatother

\usepackage{breqn}   

\usetikzlibrary{arrows,backgrounds,patterns.meta}
\usetikzlibrary{positioning,shadings,cd}
\usetikzlibrary{shapes}
\usetikzlibrary{backgrounds}
\usetikzlibrary{decorations,decorations.pathreplacing,decorations.markings,decorations.pathmorphing}
\usetikzlibrary{fit,calc,through}
\usetikzlibrary{external}
\usetikzlibrary{arrows}
\tikzset{vertex/.style = {shape=circle,draw,fill=black,inner sep=0pt,minimum size=5pt}}
\tikzset{edge/.style = {->,> = latex', bend right}}
\tikzset{
	super thick/.style={line width=3pt}
}
\tikzset{
    quadruple/.style args={[#1] in [#2] in [#3] in [#4]}{
        #1,preaction={preaction={preaction={draw,#4},draw,#3}, draw,#2}
    }
}
\tikzstyle{shaded}=[fill=red!10!blue!20!gray!30!white]
\tikzstyle{unshaded}=[fill=white]
\tikzstyle{empty box}=[circle, draw, thick, fill=white, opaque, inner sep=2mm]
\tikzstyle{annular}=[scale=.7, inner sep=1mm, baseline]
\tikzstyle{rectangular}=[scale=.75, inner sep=1mm, baseline=-.1cm]
\tikzstyle{mid>}=[decoration={markings, mark=at position 0.5 with {\arrow{>}}}, postaction={decorate}]
\tikzstyle{mid<}=[decoration={markings, mark=at position 0.5 with {\arrow{<}}}, postaction={decorate}]
\tikzstyle{over}=[double, draw=white, super thick, double=]
\tikzstyle{snake}=[decorate, decoration={snake, segment length=1mm, amplitude=.3mm}]
\tikzstyle{saw}=[decorate, decoration={saw, segment length=.7mm, amplitude=.25mm}]
\tikzstyle{knot}=[preaction={super thick, white, draw}]
\tikzstyle{knotthin}=[preaction={ultra thick, white, draw}] 
\tikzstyle{knotrColor}=[preaction={super thick, \rColor, draw}]

\tikzstyle{coupon}=[draw, very thick, rectangle, rounded corners=5pt]
\tikzset{Rightarrow/.style={double equal sign distance,>={Implies},->},
triplecd/.style={-,preaction={draw,Rightarrow}},
quadruplecd/.style={preaction={draw,Rightarrow,
shorten >=0pt
},
shorten >=1pt,
-,double,double
distance=0.2pt}}
\tikzset{
    tripleline/.style args={[#1] in [#2] in [#3]}{
        #1,preaction={preaction={draw,#3},draw,#2}
    }
}
\tikzstyle{triple}=[tripleline={[line width=.15mm,black] in
      [line width=.7mm,white] in
      [line width=1mm,black]}] 
\tikzset{
    quadrupleline/.style args={[#1] in [#2] in [#3] in [#4]}{
        #1,preaction={preaction={preaction={draw,#4},draw,#3}, draw,#2}
    }
}
\tikzstyle{quadruple}=[quadrupleline={[line width=.3mm,white] in
      [line width=.6mm,black] in
      [line width=1.2mm,white] in
      [line width=1.5mm,black]}]
      
\newcommand{\roundNbox}[6]{
	\draw[rounded corners=5pt, very thick, #1] ($#2+(-#3,-#3)+(-#4,0)$) rectangle ($#2+(#3,#3)+(#5,0)$);
	\coordinate (ZZa) at ($#2+(-#4,0)$);
	\coordinate (ZZb) at ($#2+(#5,0)$);
	\node at ($1/2*(ZZa)+1/2*(ZZb)$) {#6};
}

\newcommand{\tikzmath}[2][]
     {\vcenter{\hbox{\begin{tikzpicture}[#1]#2
                     \end{tikzpicture}}}
     }

\newcommand*{\Scale}[2][4]{\scalebox{#1}{$#2$}}%

\theoremstyle{plain}
\newtheorem{thm}{Theorem}[section]
\newtheorem*{thm*}{Theorem}
\newtheorem{thmalpha}{Theorem}

\newtheorem{cor}[thm]{Corollary}
\newtheorem{coralpha}[thmalpha]{Corollary}
\newtheorem*{cor*}{Corollary}

\newtheorem*{conj*}{Conjecture}
\newtheorem{lem}[thm]{Lemma}
\newtheorem*{lem*}{Lemma}

\newtheorem{prop}[thm]{Proposition}

\newtheorem*{quest*}{Question}
\newtheorem*{claim*}{Claim}

\theoremstyle{definition}
\newtheorem{defn}[thm]{Definition}
\newtheorem{fact}[thm]{Fact}
\newtheorem{facts}[thm]{Facts}
\newtheorem{construction}[thm]{Construction}

\newtheorem{nota}[thm]{Notation}

\newtheorem{ex}[thm]{Example}
\newtheorem{sub-ex}[thm]{Sub-Example}
\newtheorem{counter-ex}[thm]{Counter-Example}
\newtheorem{rem}[thm]{Remark}
\newtheorem*{rem*}{Remark}

\usepackage{xcolor}
\definecolor{dark-red}{rgb}{0.7,0.25,0.25}
\definecolor{dark-blue}{rgb}{0.15,0.15,0.55}
\definecolor{medium-blue}{rgb}{0,0,.8}
\definecolor{DarkGreen}{RGB}{0,150,0}
\definecolor{rho}{named}{red}
\hypersetup{
   colorlinks, linkcolor={purple},
   citecolor={medium-blue}, urlcolor={medium-blue}
}

\newcommand{\XColor}{red} 
\newcommand{\YColor}{blue}

\newcommand{\QsColor}{black}

\newcommand{\rColor}{gray!20} 

\newcommand{\id}{\operatorname{id}}
\newcommand{\Int}{\operatorname{Int}}
\newcommand{\Inv}{\operatorname{Inv}}
\newcommand{\Inn}{\operatorname{Int}}

\newcommand{\im}{\operatorname{im}}

\newcommand{\tr}{\operatorname{tr}}

\newcommand{\Hom}{\operatorname{Hom}}

\newcommand{\End}{\operatorname{End}}
\newcommand{\Out}{\operatorname{Out}}
\newcommand{\loc}{\operatorname{loc}}
\newcommand{\ai}{\operatorname{ai}}
\newcommand{\ct}{\operatorname{ct}}

\newcommand{\Aut}{\operatorname{Aut}}
\newcommand{\Ct}{\operatorname{Ct}}
\newcommand{\Ai}{{\overline{\operatorname{Int}}}}
\newcommand{\Ad}{\operatorname{Ad}}

\newcommand{\uAo}{\underline{\operatorname{Aut}}_{\otimes}}

\newcommand{\Rep}{\mathsf{Rep}}

\newcommand{\Chi}{\tilde{\chi}}

\newcommand{\fgpMod}{\mathsf{Mod_{fgp}}}
\newcommand{\Bim}{\mathsf{Bim}}

\newcommand{\Hilb}{\mathsf{Hilb}}

\newcommand{\Fun}{\mathsf{Fun}}

\newcommand{\fgpBim}{\mathsf{Bim_{fgp}}}

\newcommand{\GInn}{\overline{\mathrm{GInt}}}
\newcommand{\GCt}{\mathrm{GCt}}
\newcommand{\ootimes}{\overline{\otimes}}
\newcommand{\bigootimes}{\overline{\bigotimes}}
\newcommand{\alge}{\mathsf{alg}}

\newcommand{\ctBim}{{\Bim}_{\ct}}
\newcommand{\aiBim}{{\Bim}_{\ai}}

\def\altdb{\vadjust{\vbox to 0pt{\vss\hbox{\kern \hsize
\quad{\dbend}}\kern\baselineskip\kern-10pt}}}

\newcommand{\noshow}[1]{}

\title{Gauging the Categorical Connes' $\tilde{\chi}(M)$}
\author{Quan Chen}

\begin{document}

\begin{abstract}
We prove that if a finite group $G$ acts outerly on a McDuff $\rm II_1$ factor $M$, then $\mathsf{Rep}(G/KL)$ is a braided monoidal full subcategory of the categorical Connes' $\tilde{\chi}(M\rtimes G)$ defined in arXiv:2111.06378, where $K$ and $L$ are the centrally trivial and approximately inner parts in $G$ respectively. 

When $L$ is trivial, we give an explicit formula for the $G/K$-gauging procedure on $\tilde{\chi}(M\rtimes G)$. This is the categorical generalization of Connes' short exact sequence on $\chi(M\rtimes G)$. 
Using this machinery, for any finite group $G$, we construct a McDuff $\rm II_1$ factor $M$, whose $\tilde{\chi}(M)$ is braided equivalent to $\mathsf{Rep}(G)$.
This is the first example of a braided fusion category which is not modular as $\tilde\chi$.
\end{abstract}

\maketitle
\tableofcontents

\section{Introduction}

Tensor categories provide a fundamental algebraic language for describing symmetries in a wide range of mathematical and physical contexts, from topological quantum field theory to bimodules of operator algebras. 
A pivotal example arises from Jones' subfactor theory, where the standard invariant of a finite index $\rm II_1$-subfactor \cite{MR1278111,MR4374438} is captured by a unitary tensor category, together with a chosen Q-system in this category \cite{MR1966524,MR1027496,MR2909758,MR3948170,MR4419534}. 

In Algebraic Quantum Field Theory (AQFT), the principles of locality and covariance are encoded in a local net of operator algebras, and the Doplicher-Haag-Roberts (DHR) program describes superselection sectors as a braided tensor category \cite{MR0297259,MR1016869,MR1405610,MR1838752,MR4362722,MR4814692}. A parallel and powerful theme is the gauging of a finite global symmetry, which studies the transition of a theory enriched by symmetry defects into a new dynamical phase. 
Taking the fixed-point net under an action of a finite group $G$  yields a gauged theory, whose representation category is the $G$-equivariantization of the original theory's category\cite{MR2183964}.
Similarly, in the context of topological phases, gauging a global symmetry group results in a new phase whose excitations are described by the $G$-equivariantization of the original category \cite{MR2200691}. These constructions highlight a universal mechanism: passing from a symmetric system to its orbifold quotient is categorically realized by equivariantization.

This paper situates these ideas within the realm of $\rm II_1$ factors. 
In the recent work, we introduce the categorical Connes' invariant $\Chi(M)$ \cite{MR4753059}, a unitary braided tensor category constructed from bimodules over a $\rm II_1$ factor $M$ that are both approximately inner and centrally trivial. 
This invariant generalizes Connes' original group-theoretic invariant $\chi(M)$ \cite{MR377534} and Jones' quadratic form $\kappa$ \cite{MR585235}.
Our goal is to develop a gauging framework for this new invariant.

\begin{thmalpha} 
Let $\psi:G\to\Aut(M)$ be an outer action of a finite group $G$ on a McDuff $\rm II_1$ factor $M$. 
Let $K=\psi^{-1}(\Ct(M))\cap G$ and $L=\psi^{-1}(\Ai(M))\cap G$, then
$\Rep(G/KL)$ is a unitary braided full subcategory of $\Chi(M\rtimes G)$, where $\Ct(M)$ is the group of centrally trivial automorphisms, and $\Ai(M)$ is the group of approximately inner automorphisms.
\end{thmalpha}

The proof of this theorem relies on two key technical observations, which are of independent interest. 
This shows how the representation theory factors through the two components of Connes' invariant:
Let $\psi:G\to\Aut(M)$ be an outer action of a finite group $G$ on a McDuff $\rm II_1$ factor $M$.
From the natural embedding $\Rep(G)\hookrightarrow \Bim(M\rtimes G)$,
\begin{align*}
    \Rep(G/G\cap \psi^{-1}(\Ct(M))) &\hookrightarrow \Bim(M\rtimes G) \text{ is approximately inner;}\\
    \Rep(G/G\cap \psi^{-1}(\Ai(M))) &\hookrightarrow \Bim(M\rtimes G) \text{ is centrally trivial.}
\end{align*}
As a full tensor subcategory, $\Rep(G/KL)\subseteq \Rep(G/K)\cap \Rep(G/L)\subseteq \aiBim(M\rtimes G)\cap \ctBim(M\rtimes G)=\Chi(M\rtimes G)$ is also a braided subcategory.

Gauging is a well-known procedure of group equivariantization and deequivariantization \cite{MR2183964,MR2183279,MR2609644}. 
\[
\begin{tikzpicture}[node distance=5cm, every node/.style={align=center}]
\node (A) {$\left\{ \parbox{4.5cm}{\rm \centering Braided tensor categories \\ containing $\Rep(G)$ \\ as braided subcategory} \right\}$};
\node (B) [right=of A] {$\left\{ \parbox{3.5cm}{\centering $G$-crossed braided \\ tensor categories} \right\}$};
\draw[->] ([yshift=3pt] A.east) -- ([yshift=3pt] B.west) 
    node[midway, above] {deequivariantization};
\draw[<-] ([yshift=-3pt] A.east) -- ([yshift=-3pt] B.west) 
    node[midway, below] {$G$-equivariantization};
\end{tikzpicture}
\]

From the categorical point of view, a $G$-crossed braided tensor category is a 3-category with one object, and the set of 1-morphisms is equivalent to the group $G$ \cite{MR4535015}. The $G$-crossed braiding arises as the interchanger of 3-categories. 
The deequivariantization of $\Chi(M\rtimes G)$ as a $G/KL$-crossed braided tensor category suggests a stronger evidence that von Neumann algebras should be objects in a yet to be discovered 3-category.

\begin{thmalpha}
Suppose $G$ is a finite group, $M$ is a McDuff $\rm II_1$ factor. 
Suppose $\psi:G\to \Aut(M)$ is an outer action and $\psi$ is not approximately inner. 
Then the $G/K$-deequivariantization of $\Chi(M\rtimes G)$
$$\Chi(M\rtimes G)_{G/K}\cong \bigoplus_{g\in G/K} ({}_{\psi_g}L^2(M\rtimes K)\boxtimes_{M\rtimes K} \ctBim(M\rtimes K))\cap \aiBim(M\rtimes K)$$
as unitary $G/K$-crossed braided tensor categories, where $K=\psi^{-1}(\Ct(M))\cap G$.
\end{thmalpha}

Jones discusses it in his notes on $\chi(M)$ \cite{chiNotes} about the Connes' short exact sequence:
\begin{thm*}\cite{chiNotes}
Suppose the outer action $\psi:G\to \Aut(M)$ is not approximately inner. 
There exist group homomorphisms $\delta:\widehat{G/K}\to \chi(M\rtimes G)$ and $\Pi:\chi(M\rtimes G)\to (\GCt(M)\cap \GInn(M))/\Inn(M)$ such that 
$$0\to \widehat{G/K}\xrightarrow{\delta} \chi(M\rtimes G)\xrightarrow{\Pi}(\GCt(M)\cap \GInn(M))/\Inn(M)\to 0$$
is a short exact sequence, where $\GInn(M)$ is
the group of $G$-approximately inner automorphisms, consisting of $\alpha\in \Aut(M)$ such that there exists a sequence of unitaries $(u_n)\subseteq M^G$ with $\|\alpha(x)-u_nxu_n^*\|_2\to 0$ for all $x\in M$;
$\GCt(M)$ is the group of $G$-centrally trivial automorphisms, consisting of $\alpha\in \Aut(M)$ that there exists $g\in G$ with $\|\alpha(a_n)-\psi_g(a_n)\|_2\to 0$ for all central sequence $(a_n)\in M_\omega=M'\cap M^\omega$. 
\end{thm*}

This short exact sequence is one of the motivations for applying the \textit{gauging paradigm} to $\Chi(M\rtimes G)$. 
The group characters can be viewed as $1$-dimensional representations. Therefore, Theorem A is the categorical generalization of $\delta:\widehat{G/K}\to \chi(M\rtimes G)$.

In general, suppose $G$ is a finite group, $\cB$ is a braided $\rm C^*$-tensor category and $\Rep(G)\hookrightarrow \cB$ is a unitary braided fully faithful tensor subcategory. 
We prove that 
$$0\to \widehat{G}\to \Inv(\cB)\to \Inv(\cB_G)_{\rm{lift}}\to 0$$ is a short exact sequence, where 
$\Inv(-)$ is the group of isomorphism classes of invertible objects and 
$$\Inv(\cB_G)_{\rm{lift}}=\{[y]\in \Inv(\cB_G)\mid y\ \mathrm{admits\ a\ unitary\ } G\mathrm{-equivariant\ structure}\}$$

We show that the invertible liftable part of the $G/K$-deequivariantization on $\Chi(M\rtimes G)$ is equivalent to $(\GCt(M)\cap\GInn(M))/\Inn(M)$. 
This entails that the gauging procedure is a complete categorical generalization of Connes' short exact sequence.

\begin{thmalpha}
Suppose the outer action $\psi:G\to \Aut(M)$ is not approximately inner and $K=\psi^{-1}(\Ct(M))\cap G$.
The group of invertible $G/K$-liftable bimodules in the $G/K$-deequivariantization $\Chi(M\rtimes G)_{G/K}$ is equivalent to the group of the intersection of $G$-centrally trivial automorphisms and $G$-approximately inner automorphisms $\GCt(M)\cap \GInn(M)$ modulo inner automorphisms, i.e.,
$$\Scale[0.91]{\Inv\left(\bigoplus\limits_{g\in G/K} ({}_{\psi_g}L^2(M\rtimes K)\boxtimes_{M\rtimes K}\ctBim(M\rtimes K))\cap \aiBim(M\rtimes K)\right)_{\mathsf{lift}}\cong (\GCt(M)\cap\GInn(M))/\Inn(M).}$$   
\end{thmalpha}

In \cite{MR2661553}, Popa studies Connes' $\chi(M)$ of a countable infinite spatial tensor product of non-Gamma $\rm II_1$ factors $M=\bigootimes_{i=1}^\infty N_i$. 
He shows $\chi(M)$ is trivial, and $M$ is McDuff but not s-McDuff ($M$ is s-McDuff if $M$ is isomorphic to $R\ootimes N$ for some non-Gamma $N$), which gives a negative answer to Connes' question \cite{MR377534}. 
Building off techniques of Popa and \cite{MR4753059}, we prove that the categorical Connes' $\Chi(M)\cong \Hilb$ is trivial. 
This also gives a negative answer to \cite{MR4753059}'s generalized version of Connes' question:
whether it is true that for a McDuff $\rm II_1$ factor $M$, $M$ is s-McDuff if and only if $\Chi(M)\cong \Hilb$. 

\begin{thmalpha}
Suppose $G$ is a finite group, $N$ is a non-Gamma $\rm II_1$ factor and $\alpha:G\to \Aut(N)$ is an outer action. Let $M=\bigootimes_{i=1}^\infty N$ be the countable infinite spatial tensor product of $N$ and $\psi:G\to \Aut(M)$ is given by $\psi_g:=\bigotimes_{i=1}^\infty \alpha_g$. 
Then $\Chi(M\rtimes_{\psi} G)\cong \Rep(G)$ as unitary braided tensor categories.
\end{thmalpha}

To prove this result, we have to use the gauging construction established before.
We first show that the group action $\psi:G\to\Aut(M)$ is not approximately inner nor centrally trivial, then by Theorem B,
the $G$-deequivariantization on $\Chi(M\rtimes G)$,
$$\Chi(M\rtimes G)_G\cong \bigoplus_{g\in G} ({}_{\psi_g}L^2(M)\boxtimes_M \ctBim(M))\cap \aiBim(M).$$
The $e$-grading component is $\Chi(M)\cong \Hilb$. 
We also prove that ${}_{\psi_g}L^2(M)\boxtimes_M X$ cannot be approximately inner, when $X$ is centrally trivial. 
In other words, the $g$-grading component is empty when $g\ne e$.
Therefore,
$$\Chi(M\rtimes G)_G\cong \Hilb.$$
By the equivariantization/deequivariantization correspondence,
$$\Chi(M\rtimes G)\cong (\Chi(M\rtimes G)_G)^G\cong \Hilb^G\cong \Rep(G).$$

\begin{coralpha}
For any finite group $G$, its unitary representation category $\Rep(G)$ can be realized as $\Chi(M)$ for some McDuff $\rm II_1$ factor $M$.
\end{coralpha}
Since $\Rep(G)$ is a symmetric braided tensor category, it is not modular. 
Thus, we obtain an example of a McDuff $\rm II_1$ factor whose $\Chi(M)$ is not modular.

\begin{coralpha}
There exist McDuff $\rm II_1$ factors $M$ whose $\Chi(M)$ is not modular.
\end{coralpha}

\hspace{.2cm}

\noindent\textbf{Outline.} 
In \S\ref{section:background}, we recall the categorical Connes' $\Chi(M)$ and Connes' short exact sequence.

In \S\ref{section:Rep(G)subcat}, we first prove $\Rep(G/K)\subseteq \Bim(M\rtimes G)$ is approximately inner and $\Rep(G/L)\subseteq \Bim(M\rtimes G)$ is centrally trivial. 
Then we prove that $\Rep(G/KL)\subseteq \Chi(M\rtimes G)$ is braided.

In \S\ref{section:ConnesSES}, we unpack the gauging procedure, which is necessary for computing the deequivariantization of $\Chi(M\rtimes G)$. 
We prove that the general gauging construction induces a short exact sequence of Connes' type.
Then we prove that the gauging construction indeed generalizes Connes' short exact sequence.

In \S\ref{section:Example}, for $M=\bigootimes_{i=1}^\infty N_i$ with non-Gamma $\rm II_1$ factor $N_i$, we show that $\Chi(M)\cong \Hilb$ is trivial.  
For any finite group $G$, we consider the entry-wise action on $M$ for $N_i=N$.
We show that it is neither approximately inner nor centrally trivial. Then by the gauging construction, we obtain that $\Chi(M\rtimes G)\cong \Rep(G)$.

\hspace{.2cm}

\noindent\textbf{Acknowledgment.} I am grateful to Corey Jones for bringing the gauging construction and Vaughan Jones' unpublished notes on $\chi(M)$ to my attention.
I also would like to thank Dietmar Bisch, Yasu Kawahigashi, Corey Jones, David Penneys, Jinsong Wu, and Feng Xu for helpful conversations and encouragement. 
Quan Chen is supported as a Postdoctoral Scholar
by the Department of Mathematics at Vanderbilt University.

\section{Background} \label{section:background}
\subsection{The central sequence algebra}
Let $M$ be a von Neumann algebra and $\omega$ a free ultrafilter on $\bbN$.  
We denote by $M^\omega=\{(x_n)\mid x_n\in M,\ \sup\|x_n\|<\infty\}/\{(x_n) \mid\lim_{n\to\omega} \|x_n\|_2=0\}$ the ultrapower algebra, and we let $M_\omega:=M'\cap M^\omega$ be the central sequence algebra. 

We say $M$ is a \textit{McDuff} $\rm II_1$ factor if $M\cong M\ootimes R$, where $R$ is the hyperfinite $\rm II_1$ factor.
McDuff shows that $M_\omega$ is either abelian or a $\rm II_1$ von Neumann algebra, and it is abelian if and only if $M$ is not McDuff \cite{MR281018}.
Furthermore, Connes profoundly utilizes central sequences in his groundbreaking proof of injectivity implies hyperfiniteness for $\rm II_1$ factors \cite{MR454659}.

Every automorphism $\alpha\in \Aut(M)$ determines an automorphism $\alpha_\omega$ of $M_\omega$ via $\alpha_\omega((x_n))=(\alpha(x_n))$.

\subsection{Connes' original $\chi(M)$}
For a given $\rm II_1$ factor $M$ and a free ultrafilter $\omega$ on $\bbN$, we call an automorphism $\alpha\in\Aut(M)$ approximately inner if there exists a sequence of unitaries $\{u_n\}_{n\in\omega}$ such that 
$$\lim_{n\to\omega}\|\alpha(x) - u_n x u_n^*\|_2=0$$
for all $x\in M$. 
It is clear that the composition of approximately inner automorphisms remains approximately inner. 
We denote the group of approximately inner automorphisms by $\overline{\Inn}(M)$.

We call an automorphism $\beta\in \Aut(M)$ centrally trivial if for any central sequence $(x_n)_{n\in\omega} \in M_\omega$, 
$$\lim_{n\to\omega}\|\beta(x_n) - x_n\|_2=0.$$
In other words, $\beta$ is centrally trivial if and only if the induced automorphism $\beta_\omega$ on $M_\omega$ is the identity.
Similarly, it is clear that the composition of centrally trivial automorphisms remains centrally trivial. 
We denote the group of centrally trivial automorphisms by $\Ct(M)$.

In \cite{MR377534}, Connes introduced an invariant $\chi(M)$ of a $\rm II_1$-factor $M$ by
$$\chi(M):= (\Ct(M)\cap \Ai(M))/\Inn(M),$$
a subgroup of $\Out(M)$ consisting of the image of approximately inner and centrally trivial automorphisms. 
Connes showed that $\alpha\beta\alpha^{-1}\beta^{-1}\in\Inn(M)$, for any $\alpha\in \Ai(M)$, $\beta\in \Ct(M)$, which implies $\chi(M)$ is an abelian group.  
In \cite{MR585235}, Jones defined the $\kappa$ invariant, which is a quadratic form on $\chi(M)$. 
By using $(\chi(M),\kappa)$, Jones constructed a $\rm II_1$ factor $M$ which is antiisomorphic to itself but has no involutory antiautomorphic.
Therefore, these $\rm II_1$ factors are not group von Neumann algebras.

From the tensor category perspective, Eilenberg and MacLane showed that an abelian group together with a quadratic form defines a braided 2-group \cite{MR65163},
which gives a braided tensor category under linearization. This indicates the existence of a bimodule version of $\chi$ as a braided tensor category generalizing both Connes' $\chi(M)$ and Jones' $\kappa$ invariant.

\subsection{Categorical Connes' $\Chi(M)$}

In this section, we recall the construction of categorical Connes' $\chi(M)$, including approximately inner and centrally trivial bimodules, and properties from 
\cite{MR4753059}.

Let $M$ be a $\rm II_1$ factor with faithful normal trace $\tr_M$. 
We denote by $\Bim(M)$ the category of right $\rm W^*$-correspondences as in \cite[Def.~2.15]{MR4419534} and by $\fgpBim(M)$ a subcategory with finite generated projective bimodules.
In this paper, we only consider finite generated projective (or bifinite) bimodules, and we will denote them as $\Bim(M)$ below.

For $X\in \Bim(M)$, we define the norm of $x\in X$ as
$$\|x\|_2^2: = \| \langle x| x\rangle_M^X \|_2$$
where $\langle \cdot| \cdot\rangle_M^X$ is the $M$-valued inner product on $X$ and $\|\cdot\|_2$ the 2-norm on $M$ defined by $\|a\|_2:=\tr(a^*a)^{\frac{1}{2}}$ for $a\in M$. 

\begin{defn}
\label{defn:ApproximatePPBasis}
For $X\in\fgpMod(M)$, an \emph{approximate $X_M$-basis} is a sequence $\{x_i^{(n)}\}_{i=1}^{K}\subseteq X$,
$\sup\limits_{i,n} \|x_i^{(n)}\| <\infty$,
and
$$
\lim_{n\to\omega} \left\|x-\sum_{i=1}^K x_i^{(n)}\langle x_i^{(n)}|x\rangle_M^X \right\|_2=0
\qquad\qquad
\forall\,x\in X.
$$
For $X\in \Bim(X)$, an \emph{approximately inner} $X_M$-basis is an approximate $X_M$-basis such that
\begin{equation}
\label{eq:AICommutativity}
\lim_{n\to\omega}\left\|ax_i^{(n)}-x_i^{(n)}a\right\|_2= 0
\qquad\qquad
\forall\,a\in M.
\end{equation}
We call a bimodule $X$ \textit{approximately inner} over $M$ if and only if there exists an approximately inner $X_M$-basis.
\end{defn}

\begin{defn}
We call a bimodule $Y\in \Bim(M)$ \textit{centrally trivial} if for all central sequences $(a_n)\in M_\omega$ and for all $y\in Y$, 
$$\lim_{n\to\omega}\|a_ny-ya_n\|_2= 0.$$
\end{defn}

Note that all the approximately inner bimodules form a tensor subcategory of $\Bim(M)$, called $\aiBim(M)$ and all the centrally trivial bimodules form a tensor subcategory called $\ctBim(M)$.

There is a centralizing structure $u$ between $\aiBim(M)$ and $\ctBim(M)$:
For $X,Y \in \Bim(M)$,
choosing approximately inner $X_M$-bases $\{x_i^{(n)}\}_i$
and ordinary $Y_M$-bases $\{y_j\}_j$ respectively,
the unitary braiding $u_{X,Y}:X\boxtimes_M Y\xrightarrow{\sim} Y\boxtimes_M X$ can be expressed as the following $\|\cdot\|_2$-limits of $M$-finite rank operators:
$$u_{X,Y}=\lim_{n\to \omega} \sum_{i,j}|y_j\boxtimes x_i^{(n)}\rangle\langle x_i^{(n)}\boxtimes y_j |.$$

Note that $u_{X,Y}$ is independent of the choice of approximately inner $X_M$-basis and $Y_M$-basis.
When we take $\Chi(M): = \aiBim(M)\cap \ctBim(M)\subseteq \Bim(M)$ be the full subcategory of bimodules that are both approximately inner and centrally trivial, $\Chi(M)$ is a unitary braided tensor category.

There are several properties and results for $\Chi(M)$ proved in \cite{MR4753059}: 
\begin{facts} \label{Facts:categoricalChi}
\mbox{}
\begin{enumerate}
\item Let $N \subseteq M$ be a finite index $\rm II_1$ subfactor such that $Q := {}_NL^2(M)_N$ is a commutative Q-system in $\Chi(N)$. Then $\Chi(M)\cong \Chi(N)^\loc_Q$ as braided unitary tensor categories.
\item Suppose $N\subseteq M$ is a finite depth non-Gamma $\rm II_1$ subfactor and $M_\infty$ is the inductive limit of the basic construction. Then $\Chi(M_\infty)$ is equivalent to the Drinfeld center of the standard invariant $\cC(N\subseteq M)$ as unitary braided tensor categories. 
\item Suppose $\rm II_1$ factors $M$ and $N$ are Morita equivalent (or called stably isomorphic for $\rm II_1$ factors), then $\Chi(M)\cong \Chi(N)$ as unitary braided tensor categories.
\item $\Chi(R)\cong\Chi(N)\cong \Chi(R\ootimes N)\cong \Hilb$, where $R$ is the hyperfinite $\rm II_1$ factor and $N$ is non-Gamma $\rm II_1$ factor.
\end{enumerate}   
\end{facts}

\subsection{Connes' short exact sequence}
Suppose $G$ is a finite group, $M$ is a McDuff $\rm II_1$-factor, $\psi:G\to \Aut(M)$ is an outer action.
Suppose $\psi$ is not approximately inner and $K=G\cap \psi^{-1}(\Ct(M))$.

\begin{defn}
Let $\alpha\in \Aut(M)$.
We say $\alpha$ is \textit{$G$-approximately inner}, denoted $\alpha\in\GInn(M)$, if there exists a sequence of unitaries $(u_n)\subseteq M^G$ such that $\|\alpha(x)-u_nxu_n^*\|_2\to 0$ for all $x\in M$.
We say $\alpha$ is \textit{$G$-centrally trivial}, denoted $\alpha\in \GCt(M)$, if there exists $g\in G$, such that $\|\alpha(a_n)-\psi_g(a_n)\|_2\to 0$ for all central sequences $(a_n)\in M_\omega$. 
In other words, the induced automorphism $\alpha_\omega=(\psi_g)_\omega$ on $M_\omega$ for some $g\in G$.
\end{defn}

\begin{thm}\cite{chiNotes}
There exist group homomorphisms $\delta:\widehat{G/K}\to \chi(M\rtimes G)$ and $\Pi:\chi(M\rtimes G)\to (\GCt(M)\cap \GInn(M))/\Inn(M)$ such that 
$$0\to \widehat{G/K}\xrightarrow{\delta} \chi(M\rtimes G)\xrightarrow{\Pi}(\GCt(M)\cap \GInn(M))/\Inn(M)\to 0$$
is a short exact sequence.
\end{thm}
This short exact sequence provides a tool for calculating $\chi(M\rtimes G)$ from the data of $M$ and the action $\psi$ on $M$.

\section{Representation subcategories in $\Chi(M\rtimes G)$}\label{section:Rep(G)subcat}

Inspired by the orbifold construction on $\chi(M)$ for Connes' short exact sequence  \cite{chiNotes} (see also \cite[\S15]{MR1642584}), orbifold algebraic quantum field theory \cite{MR2183964} and the observation in \cite[Thm.~B]{MR4753059}, we give a categorical gauging construction on $\Chi(M)$ in this section.

Suppose $\psi:G\to \Aut(M)$ is an outer action of a finite group $G$ on a $\rm II_1$ factor $M$.
Recall that in the crossed product $M\rtimes_\psi G$, every element can be uniquely expressed as $\sum\limits_{g\in G} a_gu_g$, $a_g\in M$, where 
\begin{align*}
u_gu_h &= u_{gh} \text{ and } u_g^* = u_{g^{-1}}\\ 
u_g x &= \psi_g(x) u_g \text{ for }x\in M
\end{align*}
Then $(xu_g)^* = u_g^* x^* = u_{g^{-1}} x^* = \psi_{g^{-1}}(x^*) u_{g^{-1}}$. 
It is known that when $\psi$ is outer, $M\rtimes G$ is also a $\rm II_1$ factor.

Let $\Rep(G)$ be the category of unitary $G$-representations, which is a symmetric braided tensor category. 
There is a canonical embedding $F:\Rep(G)\to\Bim(M\rtimes G)$ defined by $(H,\pi)\mapsto {}_{\pi}H\otimes M\rtimes G$.
Here the left and right $M\rtimes G$-module structures are given by
$$b_hu_h\rhd (\xi\otimes a_gu_g) : = \pi_h(\xi) \otimes (b_hu_h a_gu_g)
\qquad\text{ and }\qquad 
(\xi\otimes a_gu_g)\lhd b_hu_h: = \xi\otimes (a_gu_gb_hu_h)$$
and the $M\rtimes G$-valued inner product is given by
$$\langle \xi\otimes a_gu_g\mid \eta\otimes b_hu_h\rangle : = \langle \xi\mid \eta\rangle (a_gu_g)^*(b_hu_h).$$
For $(H,\pi),(K,\sigma)\in \Rep(G)$, $\eta\in H$, $\xi\in K$, we define the tensorator $F^2_{(H,\pi),(K,\sigma)}$ for $F$:
\begin{align*}
F^2_{(H,\pi),(K,\sigma)}: & F(H,\pi)\boxtimes_{M\rtimes G} F(K,\sigma)\xrightarrow{\sim} F( {}_{\pi\otimes \psi}(H\otimes K)) \\
& (\eta\otimes a_gu_g) \boxtimes (\xi \otimes b_hu_h) = (\eta\otimes 1)\boxtimes (\sigma_g(\xi)\otimes a_gu_gb_hu_h) \mapsto (\eta\otimes \sigma_g(\xi))\otimes a_gu_gb_hu_h
\end{align*}
Note that there is a natural symmetric braiding $\beta$ on $\Rep(G)$.
One can check that 
$$F(\beta_{(H,\pi),(K,\sigma)})((\eta\otimes \xi) \otimes a_gu_g)) = (\xi\otimes \eta)\otimes a_gu_g.$$

It is clear that $\{\eta_i\otimes u_1\}_i$ is a Pimsner-Popa $({}_\pi H\otimes M\rtimes G)_{M\rtimes_{\psi} G}$-basis, where $\{\eta_i\}$ is an ONB for $H$.
\\

Suppose $N$ is a normal subgroup, then $\Rep(G/N)$ is a braided tensor subcategory of $\Rep(G)$.
Indeed, we can construct a braided tensor functor $F:\Rep(G/N)\to \Rep(G)$ as follows: 
for an object $(H,\pi)\in \Rep(G/N)$, we define 
$$F(H,\pi):=(H,\pi'),$$
where $\pi':G\to B(H)$ is given by $\pi'(g) = \pi(gN)$. 

For a morphism $f\in \Rep(G/N)$, we define 
$$F(f):= f$$

Here, for $\rm II_1$ factor $M$, the group of approximately inner automorphisms $\Ai(M)$ and the group of centrally trivial automorphisms $\Ct(M)$ are normal subgroups of $\Out(M)$.
Therefore, for outer action $\psi:G\to \Aut(M)$, $G\cap \psi^{-1}(\Ct(M))$ and $G\cap \psi^{-1}(\Ai(M))$ are normal subgroups of $G$, and we have 
\begin{align*}
    \Rep(G/G\cap \psi^{-1}(\Ct(M))) &\hookrightarrow \Rep(G) \hookrightarrow \Bim(M\rtimes G),\\
    \Rep(G/G\cap \psi^{-1}(\Ai(M))) &\hookrightarrow \Rep(G) \hookrightarrow \Bim(M\rtimes G)
\end{align*}

\subsection{$\Rep(G/G\cap \psi^{-1}(\Ct(M)))$ is approximately inner}
In this section, we show that $\Rep(G/G\cap \psi^{-1}(\Ct(M)))\hookrightarrow \Bim(M\rtimes G)$ is approximately inner.
By unpacking the embedding $\Rep(G/G\cap \psi^{-1}(\Ct(M)) \hookrightarrow \Bim(M\rtimes G)$, every representation $(H,\pi)\in \Rep(G/G\cap \psi^{-1}(\Ct(M)))$ can be viewed as a $G$-representation, in particular $\pi_g=\id_H$ if $\psi_g\in \Ct(M)$. 

We first recall the definition and properties of a cocycle/coboundary of a group action.

\begin{defn}
Let $M$ be a $\rm II_1$ factor and $\psi: G\to \Aut(M)$ is an action.
We call $w:G\to U(M)$ a \emph{cocycle} (or 1-cocycle) if $w_g\psi_g(w_h)=w_{gh}$ for all $g,h\in G$. 
Equivalently, $\frac{1}{|G|}\sum_g w_g u_g$ is a projection in $M\rtimes G$.

A cocycle $w$ is called a \textit{coboundary} if there exists a unitary $t\in N$ such that $w_g = t^* \psi_g(t)$ for all $g\in G$.
\end{defn}

\begin{ex}
\label{ex:RepCocyle}
Suppose $\psi:G\to \Aut(M)$ is an action
and $(H,\pi)$ is a finite dimensional $G$-representation.
Then $\pi:G\to B(H)=B(H)\otimes \bbC\hookrightarrow  B(H)\otimes M$ is automatically a cocycle for the $G$-action $\tilde{\psi}:G \to \Aut(B(H)\otimes M)$ by $\tilde{\psi}_g(a\otimes m) = a \otimes \psi_g(m)$ for $g\in G$, $a\in B(H)$, and $m\in M$. 
Indeed, 
\begin{align*}
\pi(g)\tilde{\psi}_g(\pi(h)) = (\pi_g\otimes 1)\tilde{\psi}_g(\pi_h\otimes 1) = (\pi_g\otimes 1)(\pi_h\otimes \psi_g(1)) = \pi_g\pi_h \otimes 1  = \pi_{gh}\otimes 1 = \pi(gh).
\end{align*}
\end{ex}

Let $N=G\cap \psi^{-1}(\Inn(M))$, which is a normal subgroup of $G$.
For each $h\in N$, since $\psi_h$ is inner, there exists $v_h\in U(M)$ such that $\psi_h=\Ad(v_h)$. 

\begin{defn}\cite[\S1.2]{MR587749}
Given a group action $\psi:G\to\Aut(M)$, a \textit{section} is a choice of unitaries $\{v_h\}_{h\in N}$ such that
$$v_h v_k = \mu(h,k)v_{hk}\qquad\text{ for }h,k\in N,$$
where $\mu:N\times N\to U(1)$ is a 2-cocycle; and 
$$\psi_g(v_{g^{-1}hg}) = \lambda(g,h) v_h\qquad\text{ for }g\in G,\ h\in N,$$
where $\lambda:G\times N\to U(1)$ is compatible with $\psi$ and $\mu$.
\end{defn}

It seems $\mu$ and $\lambda$ are dependent to the choice of $\{v_h\}_{h\in N}$. In fact,
Jones proves that the cohomology class $[\mu,\lambda]$, which he calls the characteristic invariant, is an invariant of the group action $\psi$ \cite[Def.~2.8]{MR587749}.
It is essential for the classification of finite group actions on the hyperfinite $\rm II_1$ factor up to conjugacy \cite[Thm.~1.4.8]{MR587749}.

For our purpose, the notion of section provides a characterization whether a cocycle is a coboundary. 

\begin{fact}\cite[\S3.1]{MR587749}\label{fact:coboundaryTrEQ}
Suppose $\psi: G\to \Aut(M)$ is an action and $w$ is a cocycle. 
The following are equivalent:
\begin{enumerate}
\item $w$ is a coboundary;
\item the projections $\frac{1}{|G|}\sum_g w_g u_g$ and $\frac{1}{|G|}\sum_g u_g$ are von Neumann equivalent in $M\rtimes G$;
\item $\tr(w_hv_h)=\tr(v_h)$ for all $h\in G\cap\psi^{-1}(\Inn(M))$, where $\{v_h\}$ is a section of the action $\psi$.
\end{enumerate}
\end{fact}

\begin{rem}
When $\psi$ is outer, then any 1-cocycle is a coboundary. 
This is known as Ocneanu’s cocycle vanishing theorem.
\end{rem}

\begin{lem}\cite{MR454659}\label{lem:NotCT->OuterOnMomega}
For the embedding $\Phi:\Aut(M)\to \Aut(M_\omega)$ given by $\alpha\mapsto \alpha_\omega$,
we have $\Ct(M)= \Phi^{-1}(\Inn(M_\omega))$ .
\end{lem}
\begin{proof}
If $\alpha\in \Ct(M)$, then $\alpha_\omega$ is the identity on $M_\omega$, which is inner. Thus, $\Ct(M)\subseteq \Phi^{-1}(\Inn(M_\omega))$.

Suppose there is an $\alpha$ that is not centrally trivial
but the induced action $\alpha_\omega$ is inner, then $\alpha_\omega=\Ad V$ for some unitary $V=(v_n)\in M_\omega$. 
Note that if $\alpha$ is not centrally trivial, then there exists $X=(x_n)\in M_\omega$ such that $\alpha_\omega(X)\ne X$ in $M_\omega$. 
Moreover, there exists $\epsilon>0$ such that we can pick a subsequence (still denoted as $X=(x_n)$) with $\|\alpha(x_n)-x_n\|_2>\epsilon$ for every $n$.
For each fixed $n$, since $X$ is a central sequence, there exists $(k_n)$ such that $\|[v_n,x_{k_n}]\|_2=\|v_nx_{k_n}v_n^* - x_{k_n}\|_2 < \frac{1}{n}$.
This implies that for $Y = (x_{k_n})\in M_\omega$, $\alpha_\omega(Y) - Y = VYV^* - Y = 0$, which is a contradiction.
\end{proof}

\begin{lem}\label{lem:innerPsitildePsi}
For $\tilde{\psi}:G\to \Aut(B(H)\otimes M)$ given by $\tilde{\psi}_g := \id_{B(H)}\otimes \psi_g$, 
we have
$$G\cap \tilde{\psi}^{-1}(\Inn(B(H)\otimes M)) = G\cap \psi^{-1}(\Inn(M)).$$
\end{lem}
\begin{proof}
It is clear that $\supseteq$ holds, since if $\psi_g$ is inner, then $\tilde{\psi}_g= \id\otimes \psi_g$ is inner. 
We shall prove the other inclusion.
Suppose $\tilde{\psi}_g$ is inner, then there exists $U\in B(H)\otimes M$ such that $\tilde{\psi}_g=\Ad(U)$.
In particular, 
$$\tilde{\psi}_g(1\otimes m) = U(1\otimes m)U^*=1\otimes \psi_g(m).$$ 
We can rewrite this as
$$U(1\otimes m) = (1\otimes \psi_g(m))U.$$
Since $H$ is finite-dimensional, we can write $U=(U_{ij})_{i,j}$ where $U_{ij}\in M$ via the identification $B(H)\otimes M\cong M_n(M)$.
Additionally, $1\otimes m = m I_n$ and $1\otimes \psi_g(m) = \psi_g(m) I_n$, where $I_n\in B(H)$ is the identity matrix, then the above equation gives 
$$(U_{ij})(mI_n) = (\psi_g(m)I_n)(U_{ij}).$$
This implies 
$$U_{ij}m=\psi_g(m)U_{ij}$$
holds for all $i,h$ and $m\in M$. Since $U$ is a unitary, there exists a $U_{ij}\ne 0$. Choose one such nonzero $U_{ij}=:x$, then $xm=\psi_g(m)x$ for all $m\in M$, we shall show $\psi_g$ is inner.

First, 
\begin{align*}
    x^*xm &= x^*(xm)= x^* (\psi_g(m)x) = (x^* \psi_g(m))x = (\psi_g(m)^* x)^* x \\
    &= (\psi_g(m^*)x)^*x = (x m^*)^* x = mx^*x
\end{align*} holds for all $m\in M$
this implies $x^*x \in Z(M)\cong \bbC 1$.
Similarly,
\begin{align*}
    \psi_g(m) xx^* &=  (\psi_g(m) x)x^* = xmx^* = x(xm^*)^* \\
    & = x(\psi_g(m^*)x)^* = xx^*\psi_g(m^*)^* = xx^*\psi_g(m)
\end{align*}holds for all $m\in M$,
this implies $xx^*\in Z(M)\cong \bbC 1$.
Combining these two results, let $\mu = \tr(x^*x)^{\frac{1}{2}}=\tr(xx^*)^{\frac{1}{2}}$, we have $u: = \frac{x}{\mu}$ is a unitary.
Then 
$$um = \mu xm=\mu \psi_g(m) x = \psi_g(m) u,$$
which implies $\psi_g$ is inner.
\end{proof}

\begin{nota}\label{nota:approx}
We use the same convention as in \cite[Nota.~4.3]{MR4753059}. For norm bounded nets $(f_n)_n,(g_n)_n\subseteq\Hom(X_M\to Y_M)$, and $f,g\in \Hom(X_M\to Y_M)$, we write
\begin{itemize}
    \item $f_n\approx g_n$ if $\lim\limits_{n\to\omega}\|f_n-g_n\|_2=0$;
    \item $f_n\approx f$ if $\lim\limits_{n\to\omega}\|f_n-f\|_2=0$.
\end{itemize}
Therefore, $f\approx g$ if and only if $f=g$.     
\end{nota}

\begin{thm}\label{thm:NotCT->AI}
$\Rep(G/G\cap \psi^{-1}(\Ct(M)))\hookrightarrow \Bim(M\rtimes G)$ is approximately inner.
\end{thm}
\begin{proof}
We denote $K=G\cap \psi^{-1}(\Ct(M))$.
The outer action $\psi:G\to \Aut(M)$ gives an induced action $\psi_\omega=\Phi\circ \psi$ on the central sequence algebra $\psi_\omega:G\to \Aut(M_\omega)$.
For a given $(H,\pi)\in \Rep(G/K)\subseteq \Rep(G)$, by Example \ref{ex:RepCocyle}, $\pi: G\to B(H)\hookrightarrow B(H)\otimes M_\omega$ is a cocycle for $\tilde{\psi}_\omega$.

By Lemma \ref{lem:NotCT->OuterOnMomega}, if the induced action $\psi_h$ is inner, then $\psi_h$ is centrally trivial.
Note that for any section $\{v_h\}$ of $\tilde{\psi}_\omega$, by Lemma \ref{lem:innerPsitildePsi}, since 
$$h\in G\cap\tilde{\psi}_\omega^{-1}(\Inn(B(H)\otimes M_\omega)) = G\cap\psi_\omega^{-1}(\Inn(M_\omega))= G\cap\psi^{-1}(\Ct(M)),$$
we have $\pi_h=\id_H$ so that $\tr(\pi_hv_h) = \tr(v_h)$.
According to Fact \ref{fact:coboundaryTrEQ}, the cocycle $\pi$ is a coboundary, i.e., there exists a unitary $t\in B(H)\otimes M_\omega$ such that $\pi_g = t^*(\tilde{\psi}_\omega)_g(t)$ for all $g\in G$.

Let $\{\eta_i\}$ be an ONB for $H$ and $\{\varepsilon_{ij}\}$ be the corresponding system of matrix units for $B(H)$.
We write $t=\sum_{i,j} \varepsilon_{ij}\otimes t_{ij}$ for $t_{ij} = (t_{ij;n})_{n\in\omega}\in M_\omega$. Then we have
\begin{align*}
t^* & = \sum_{i,j} \varepsilon_{ij}^*\otimes t_{ij}^* = \sum_{i,j} \varepsilon_{ij}\otimes t_{ji}^*, \\
tt^* & = \sum_{i,j}\varepsilon_{ij}\otimes \sum_{k} t_{ik} t_{jk}^*, \\
t^*t & = \sum_{i,j}\varepsilon_{ij}\otimes \sum_{k} t_{ki}^* t_{kj}.
\end{align*}
Note that $t\in U(B(H)\otimes M_\omega)$, we have
$$\sum_{k} t_{ik} t_{jk}^* = \delta_{i=j} = \sum_{k} t_{ki}^* t_{kj}\quad \text{ in }M_\omega.$$
In other words,
$$\lim_{n\to \omega}\sum_{k} t_{ik;n} t_{jk;n}^* = \delta_{i=j} = \lim_{n\to \omega}\sum_{k} t_{ki;n}^* t_{kj;n}.$$
Note that $(\tilde{\psi}_\omega)_g = \id_H\otimes (\psi_\omega)_g$, 
$$\pi_g = t^*(\tilde{\psi}_\omega)_g(t) = \sum_{i,j} \varepsilon_{ij}\otimes \sum_k t_{ki}^*(\psi_\omega)_g(t_{kj}) \in B(H)\otimes M_\omega,$$
then 
$$\pi_g(\eta_j\otimes b_hu_h) = \sum_l \eta_l\otimes \sum_k t_{kl}^*(\psi_\omega)_g(t_{kj})b_hu_h = \lim_{n\to \omega} \sum_l \eta_l\otimes \sum_k t_{kl;n}^*\psi_g(t_{kj;n})b_hu_h.$$

Let $x_i^{(n)}=\sum_j \eta_j\otimes t_{ij;n}^*u_1$ and we shall show $\{x_i^{(n)}\}_i$ is an approximately inner $({}_{\pi}H\otimes M\rtimes G) _{M\rtimes G}$ basis.

First, for $x= \eta_k\otimes a_gu_g \in {}_{\pi}H\otimes M\rtimes G$,
\begin{align*}
\left\| x - \sum_{i} x_i^{(n)}\langle x_i^{(n)}|x\rangle \right\|_2 
& = \left\|\eta_k\otimes a_gu_g - \sum_i \sum_j \eta_j \otimes t_{ij;n}^*u_1\left\langle \sum_j \eta_j \otimes t_{ij;n}^* u_1\middle| \eta_k\otimes a_gu_g\right\rangle \right\|_2 \\
& = \left\|\eta_k\otimes a_gu_g - \sum_i \sum_j \eta_j \otimes t_{ij;n}^*  t_{ik;n}a_gu_g \right\|_2 \\
& = \left\|\eta_k\otimes a_gu_g - \sum_j \eta_j\otimes \sum_i t_{ij;n}^*  t_{ik;n}a_gu_g \right\|_2 \\
& \approx \left\|\eta_k\otimes a_gu_g - \sum_j \eta_j \otimes \delta_{j=k} a_gu_g \right\|_2 \\
& = 0.
\end{align*}
For $a= a_gu_g\in M\rtimes G$,
{\allowdisplaybreaks
\begin{align*}
\left\|a\rhd x_i^{(n)}-x_i^{(n)}\lhd a\right\|_2 
& =  \left\|a_gu_g \rhd \sum_j \eta_j\otimes t_{ij;n}^*u_1-\sum_j \eta_j\otimes t_{ij;n}^*u_1\lhd a_gu_g\right\|_2 \\
& = \left\|\pi_g\sum_j \eta_j\otimes a_g u_g t_{ij;n}^*u_1-\sum_j \eta_j\otimes t_{ij;n}^*u_1a_gu_g\right\|_2 \\ 
& \approx \left\|\sum_j \sum_l \eta_l\otimes \sum_k t_{kl;n}^*\psi_g(t_{kj;n}) a_g u_gt_{ij;n}^*u_1-\sum_j \eta_j\otimes t_{ij;n}^*a_gu_g\right\|_2 \\
& = \left\|\sum_l \eta_l\otimes \sum_k \sum_j t_{kl;n}^*\psi_g(t_{kj;n}) a_g \psi_g(t_{ij;n}^*)u_g-\sum_j \eta_j\otimes t_{ij;n}^*a_gu_g\right\|_2 \\
& \approx \left\|\sum_l \eta_l\otimes \sum_k  t_{kl;n}^*\sum_j\psi_g(t_{kj;n})  \psi_g(t_{ij;n}^*)a_gu_g-\sum_j \eta_j\otimes t_{ij;n}^*a_gu_g\right\|_2 \\
& = \left\|\sum_l \eta_l\otimes \sum_k t_{kl;n}^* \psi_g\left(\sum_j t_{kj;n}  t_{ij;n}^*\right)a_gu_g-\sum_j \eta_j\otimes t_{ij;n}^*a_gu_g\right\|_2 \\
& \approx \left\|\sum_l \eta_l\otimes \sum_k t_{kl;n}^* \delta_{k=i} a_gu_g-\sum_j \eta_j\otimes t_{ij;n}^*a_gu_g\right\|_2 \\
& = 0
\end{align*}
}
In the fifth approximation, we use the fact that $(t_{ij;n}^*)_n$ is a central sequence, so is $(\psi_g(t_{ij;n}^*))_n$.

Therefore, ${}_{\pi}H\otimes M\rtimes G$ is approximately inner in $\Bim(M\rtimes G)$.
\end{proof}

\subsection{$\Rep(G/G\cap \psi^{-1}(\Ai(M)))$ is centrally trivial}
In this section, we show that $\Rep(G/G\cap \psi^{-1}(\Ai(M)))\hookrightarrow \Bim(M\rtimes G)$ is centrally trivial.
By unpacking the embedding $\Rep(G/G\cap \psi^{-1}(\Ai(M)) \hookrightarrow \Bim(M\rtimes G)$, every representation $(H,\pi)\in \Rep(G/G\cap \psi^{-1}(\Ai(M)))$ can be viewed as a $G$-representation, in particular $\pi_g=\id_H$ if $\psi_g\in \Ai(M)$.

The following proposition in Jones' notes is essential. 
We provide the proof for the convenience to the reader.
\begin{prop}\cite{chiNotes} \label{Prop:nonAI}
Suppose $\alpha\in\Aut(M)$ is not approximately inner on a McDuff $\rm II_1$ factor $M$. 
Then if $(x_n)\in M^\omega$ with $\|x_n\alpha(y)-yx_n\|_2\to 0$ for all $y\in M$,
then $\|x_n\|_2\to 0$.
\end{prop}
\begin{proof}
We prove the contrapositive. 
Suppose there is a bounded sequence $(x_n)$ with $\|x_n\alpha(y)-yx_n\|_2\to 0$ and $\|x_n\|_2> \delta>0$.
We can embed $M$ in the $M_2(M^\omega)$ by
$$\sigma(y) = \begin{bmatrix}
y & 0 \\
0 & \alpha(y) 
\end{bmatrix}$$
For $x=(x_n)\in M^\omega$, the condition $\|x_n \alpha(y)-yx_n\|_2\to 0$ for all $y\in M$ can be interpreted as
$$\begin{bmatrix}
0 & 0 \\
x & 0 
\end{bmatrix} \in \sigma(M)'\cap M_2(M^\omega).$$

We claim that $\sigma(M)'\cap M_2(M^\omega)$ is a $\rm II_1$-factor in this case. 
Suppose $\begin{bmatrix}
a & b \\
c & d 
\end{bmatrix}\in Z(\sigma(M)'\cap M_2(M^\omega))$,
then $\begin{bmatrix}
a & b \\
c & d 
\end{bmatrix}$ must commute with $\begin{bmatrix}
1 & 0 \\
0 & 0 
\end{bmatrix}$ and $\begin{bmatrix}
0 & 0 \\
0 & 1 
\end{bmatrix}$, 
so $b=c=0$.
Moreover,
since $\begin{bmatrix}
a & 0 \\
0 & d 
\end{bmatrix}$ commutes with $\sigma(M)$, $a$ and $d$ are in $M_\omega=M'\cap M^\omega$.
Since $M_\omega\subseteq \sigma(M)'\cap M_2(M^\omega)$, $a$ and $d$
must be in $Z(M_\omega)\cong \bbC$, then $a=\lambda$, $b=\theta$ for some scalars $\lambda,\theta$.
Since $\begin{bmatrix}
\lambda & 0 \\
0 & \theta 
\end{bmatrix}$ commutes with $\begin{bmatrix}
0 & 0 \\
x & 0 
\end{bmatrix}$, so $\lambda=\theta$.
Therefore, $\sigma(M)'\cap M_2(M^\omega)$ is a $\rm II_1$-factor and hence projections
$\begin{bmatrix}
1 & 0 \\
0 & 0 
\end{bmatrix}$ and $\begin{bmatrix}
0 & 0 \\
0 & 1 
\end{bmatrix}$ are equivalent.
 This gives a unitary $U\in M^\omega$ with $\begin{bmatrix}
0 & 0 \\
U & 0 
\end{bmatrix}\in \sigma(M)'\cap M_2(M^\omega)$ and 
$\lim\limits_{n\to\omega} \Ad u_n = \alpha$ for $U=(u_n)$,
Therefore, $\alpha$ is approximately inner.
\end{proof}

\begin{cor}\label{cor:CSinMrtimesG}
Suppose the outer action $\psi:G\to \Aut(M)$ is not approximately inner, then $$(M\rtimes G)_\omega=(M_\omega)^G.$$
\end{cor}
\begin{proof}
Suppose $(a_n)\in (M_\omega)^G$, for $m u_g\in M\rtimes G$, 
\begin{align*}
\|a_n mu_g-mu_g a_n\|_2 & = \| (a_n m -m \psi_g(a_n))u_g \|_2 = \| a_n m -m \psi_g(a_n) \|_2 \\
& \le \|a_n m -m a_n\|_2 + \|m (a_n-\psi_g(a_n)\|_2 \to 0
\end{align*}
the first term goes to 0 since $(a_n)\in M_\omega$, the second term goes to 0 since $(a_n)\in (M_\omega)^G$.

For the other direction, suppose $(b_n)\in (M\rtimes G)_\omega$, $b_n=\sum\limits_g b_{n,g}u_g$ and $(b_{n,g})_n\in M^\omega$. 
For $m\in M$, 
\begin{align*}
\|b_n m -m b_n\|_2 = \left\|\sum_g b_{n,g}u_g m-m b_{n,g}u_g \right\|_2  = \left\|\sum_g (b_{n,g}\psi_g(m)-m b_{n,g})u_g\right\|_2 \to 0
\end{align*}
This implies $\|b_{n,g}\psi_g(m)-m b_{n,g}\|_2\to 0$ for each $g\in G$ and any $m\in M$. By Proposition \ref{Prop:nonAI}, since $\psi_g$ is not approximately inner for all $g\ne e$, the bounded sequence $\|b_{n,g}\|_2\to 0$. 
For $g\in G$, 
\begin{align*}
\|b_{n,e} u_g - u_g b_{n,e}\|_2 & = \|(b_{n,e}-\psi_g(b_{n,e}))u_g\|_2 \to 0,
\end{align*}
which follows that $\|b_{n,e}-\psi_g(b_{n,e})\|_2\to 0$,
then 
\begin{align*}
\|(b_{n,e}-E_{M^G}(b_{n,e})\|_2 = \left\|b_{n,e}-\frac{1}{|G|}\sum_g\psi_g(b_{n,e})\right\|_2  
 \le \frac{1}{|G|}\sum_g \left\|b_{n,e}-\psi_g(b_{n,e})\right\|_2 \to 0
\end{align*}
Therefore $(b_n)=(b_{n,e}) = (E_{M^G}(b_{n,e})) = E_{(M_\omega)^G}((b_{n,e}))\in (M_\omega)^G$.
\end{proof}

\begin{thm}\label{thm:NotAI->CT}
$\Rep(G/G\cap \psi^{-1}(\Ai(M)))\hookrightarrow \Bim(M\rtimes G)$ is centrally trivial.
\end{thm}
\begin{proof}
Let $(H,\pi)\in \Rep(G/G\cap \psi^{-1}(\Ai(M)))$ and ${}_\pi H\otimes M\rtimes G$ is the corresponding bimodule in $\Bim(M\rtimes G)$. 
We shall show that for each central sequence $(x_n)\subseteq M\rtimes G$ and $\xi\in {}_\pi H\otimes M\rtimes G$, $\|x_n\rhd \xi-\xi\lhd x_n\|\to 0$.

We write $x_n=\sum_g x_{n,g} u_g$. For each $g\in G$, $(x_{n,g})$ is a bounded sequence and since
$$
\left\|\left(\sum_{g}x_{n,g} u_g\right)a - a\left(\sum_{g}x_{n,g} u_g\right)\right\|_2 \to 0
$$
for every $a\in M$, 
we have $\|x_{n,g} \psi_g(a) - ax_{n,g}\|_2\to 0$.
According to Proposition \ref{Prop:nonAI}, $\|x_{n,g}\|_2\to 0$ if $\psi_g\notin \Ai(M)$.

For $\xi=\eta\otimes a_ku_k\in {}_\pi H\otimes M\rtimes G$, we have
{\allowdisplaybreaks
\begin{align*}
\|x_n\rhd \xi -\xi \lhd x_n\|_2 & =
\left\|\sum_{g\in G} x_{n,g}u_g\rhd (\eta\otimes a_ku_k) - (\eta\otimes a_ku_k)\lhd \sum_{g\in G} x_{n,g}u_g\right\|_2 \\
& =  \left\|\sum_{g\in G} \pi_g(\eta)\otimes x_{n,g}u_ga_ku_k - \sum_{g\in G}\eta\otimes a_ku_kx_{n,g}u_g\right\|_2 \\
& \approx \left\|\sum_{h\in G\cap \psi^{-1}(\Ai(M))} \pi_h(\eta)\otimes x_{n,h}u_ha_ku_k - \sum_{h\in G\cap \psi^{-1}(\Ai(M))}\eta\otimes a_ku_kx_{n,h}u_h\right\|_2 \\
& = \left\|\sum_{h\in G\cap \psi^{-1}(\Ai(M))} \eta\otimes x_{n,h}u_ha_ku_k - \sum_{h\in G\cap \psi^{-1}(\Ai(M))}\eta\otimes a_ku_kx_{n,h}u_h\right\|_2 \\
& = \left\|\eta\otimes \sum_{h\in G\cap \psi^{-1}(\Ai(M))}(x_{n,h}u_ha_ku_k - a_ku_kx_{n,h}u_h)\right\|_2 \\
& \approx \left\|\eta\otimes \sum_{g\in G}(x_{n,g}u_ga_ku_k - a_ku_kx_{n,g}u_g)\right\|_2 \\
& = \left\|\eta\otimes (x_na_ku_k - a_ku_kx_n)\right\|_2 \\
& \approx 0
\end{align*}
}
In the third and sixth approximation, we use the fact that $\|x_{n,g}\|_2\to 0$ for $\psi_g\notin \Ai(M)$.
The fourth equality holds since $\pi_h=\id_H$ for $h\in G\cap \psi^{-1}(\Ai(M))$.
The last approximation holds since $(x_n)_n\subseteq M\rtimes G$ is a central sequence.
\end{proof}

\subsection{$\Rep(G/KL)\subseteq \Chi(M\rtimes G)$ is a unitary braided tensor subcategory}

Let $K=G\cap \psi^{-1}(\Ct(M))$ and $L=G\cap \psi^{-1}(\Ai(M))$.
Note that $[\Ct(M),\Ai(M)]\subseteq \Inn(M)$, and $\psi:G\to\Aut(M)$ is outer action, $K$ and $L$ commute in $G$.
Let $KL$ be the smallest group generated by $K$ and $L$. Since $K$ and $L$ are normal subgroups in $G$, $KL$ is also a normal subgroup.
In Theorem \ref{thm:NotCT->AI} and Theorem \ref{thm:NotAI->CT}, we showed that $\Rep(G/K)\subseteq \aiBim(M\rtimes G)$ and $\Rep(G/L)\subseteq \ctBim(M\rtimes G)$ are full tensor subcategories.
Then $\Rep(G/KL)$, as a full tensor subcategory of both $\Rep(G/K)$ and $\Rep(G/L)$, is a full tensor subcategory of $ \aiBim(M\rtimes G)\cap \ctBim(M\rtimes G)=\Chi(M\rtimes G)$.

In this subsection, we moreover show that $\Rep(G/KL)$ is a unitary braided subcategory of $\Chi(M\rtimes G)$.

Denote $N=M\rtimes G$. For $(H,\pi),(V,\sigma)\in \Rep(G)$, we will show 
the following diagram commutes
\[
\begin{tikzcd}[column sep=6em]
({}_{\pi} H\otimes N)\boxtimes_N ({}_{\sigma} V\otimes N)
\arrow{r}{u_{{}_{\pi} H\otimes N, {}_{\sigma} V\otimes N}}
\arrow[swap]{d}{F^2_{(H,\pi),(V,\sigma)}} 
& ({}_{\sigma} V\otimes N)\boxtimes_N ({}_{\pi} H\otimes N)
\\
{}_{\pi\otimes \sigma} (H\otimes V)\otimes N
\arrow{r}{F(\beta_{(H,\pi),(V,\sigma)})}
& {}_{\sigma\otimes \pi} (V\otimes H)\otimes N
\arrow[swap]{u}{\left(F^2_{(V,\sigma),(H,\pi)}\right)^\dag} 
\end{tikzcd}
\]

Let $\{\eta_i\}$ and $\{\xi_j\}$ be ONB for $H$ and $V$ respectively.
From Theorem \ref{thm:NotCT->AI}, we can take $\{x_i^{(n)}\}_i$ as approximately inner $({}_{\pi} H\otimes N)_N$-basis with $x_i^{(n)} = \sum_j \eta_j\otimes t_{ij;n}^*$ where $t\in U(B(H)\otimes M_\omega)$ and $\pi_g = t^*\tilde{\psi}_g(t)$. 
We take $\{\xi_j\otimes u_1\}_j$ as the $({}_{\sigma} V\otimes N)_N$-basis.

For $(\eta_k\otimes a_gu_g)\boxtimes (\xi_l\otimes b_hu_h) \in ({}_{\pi} H\otimes N)\boxtimes_N ({}_{\sigma} V\otimes N)$, on one hand,

{\allowdisplaybreaks
\begin{align*}
u_{{}_{\pi} H\otimes N, {}_{\sigma} K\otimes N} & [(\eta_k\otimes a_gu_g)\boxtimes (\xi_l\otimes b_hu_h) ] \\
& \approx \sum_{i,j}  (\xi_j\otimes u_1) \boxtimes x_i^{(n)} \left\langle  x_i^{(n)}\boxtimes (\xi_j\otimes u_1) \middle| (\eta_k\otimes a_gu_g)\boxtimes (\xi_l\otimes b_hu_h)\right\rangle \\
& = \sum_{i,j}  (\xi_j\otimes u_1) \boxtimes x_i^{(n)} \left\langle  \xi_j\otimes u_1 \middle| \left\langle x_i^{(n)} \middle| (\eta_k\otimes a_gu_g)\right\rangle\rhd (\xi_l\otimes b_hu_h)\right\rangle \\
& = \sum_{i,j}  (\xi_j\otimes u_1) \boxtimes x_i^{(n)} \left\langle  \xi_j\otimes u_1 \middle| t_{ik;n}a_gu_g\rhd (\xi_l\otimes b_hu_h)\right\rangle \\
& = \sum_{i,j}  (\xi_j\otimes u_1) \boxtimes x_i^{(n)} \left\langle  \xi_j\otimes u_1 \middle| \sigma_g(\xi_l)\otimes t_{ik;n}a_gu_g b_hu_h\right\rangle \\
& = \sum_{i,j}  (\xi_j\otimes u_1) \boxtimes x_i^{(n)}  \langle \xi_j | \sigma_g(\xi_l)\rangle t_{ik;n}a_gu_g b_hu_h \\
& = \sum_{i,j} (\xi_j\langle \xi_j | \sigma_g(\xi_l)\rangle \otimes u_1) \boxtimes x_i^{(n)} t_{ik;n}a_gu_g b_hu_h \\
& = \sum_{i} (\sigma_g(\xi_l)\otimes u_1) \boxtimes x_i^{(n)} t_{ik;n}a_gu_g b_hu_h \\
& \approx \sum_{i} (\sigma_g(\xi_l)\otimes u_1) \boxtimes \left(\sum_q \eta_q\otimes t_{iq;n}^*\right) t_{ik;n}a_gu_g b_hu_h \\
& = (\sigma_g(\xi_l)\otimes u_1) \boxtimes \sum_q \eta_q\otimes \left(\sum_{i} t_{iq;n}^* t_{ik;n}\right)a_gu_g b_hu_h \\
& \approx (\sigma_g(\xi_l)\otimes u_1) \boxtimes \sum_q \eta_q\otimes \delta_{q=k} a_gu_g b_hu_h \\
& = (\sigma_g(\xi_l)\otimes u_1) \boxtimes \eta_k\otimes a_gu_gb_hu_h.
\end{align*}
On the other hand,
\begin{align*}
\left(F^2_{(K,\sigma),(H,\pi)}\right)^\dag & \circ F(\beta_{(H,\pi),(K,\sigma)}) \circ F^2_{(H,\pi),(K,\sigma)} [(\eta_k\otimes a_gu_g)\boxtimes (\xi_l\otimes b_hu_h)] \\
& = \left(F^2_{(K,\sigma),(H,\pi)}\right)^\dag \circ F(\beta_{(H,\pi),(K,\sigma)}) [(\eta_k\otimes \sigma_g(\xi_l))\otimes a_gu_gb_hu_h] \\
& = \left(F^2_{(K,\sigma),(H,\pi)}\right)^\dag [(\sigma_g(\xi_l)\otimes \eta_k)\otimes a_gu_gb_hu_h] \\
& = (\sigma_g(\xi_l)\otimes u_1) \boxtimes \eta_k\otimes a_gu_gb_hu_h.
\end{align*}
The results agree.
Therefore, $\Rep(G/KL)$ is a unitary braided subcategory of $\Chi(M\rtimes G)$.

However, $\Rep(G/KL)$ is not a central subcategory of $\Chi(M\rtimes G)$ in general. 
We will give the argument in the appendix.

\begin{thm}\label{Thm:Rep(G)subcatChi}
Suppose $\psi:G\to\Aut(M)$ is a finite group outer action on a McDuff $\rm II_1$ factor $M$.
Then $\Rep(G/KL)$ is a unitary braided full subcategory of $\Chi(M\rtimes G)$, where $K=G\cap \psi^{-1}(\Ct(M))$ and $L=G\cap \psi^{-1}(\Ai(M))$.
\end{thm}

\section{The generalization of Connes' short exact sequence}
\label{section:ConnesSES}
In this section, we first give a computable formula for the deequivariantization on $\Chi(M\rtimes G)$. 
Then we prove that the gauging construction is a honest categorical generalization of the Connes' short exact sequence for the ordinary $\chi(M\rtimes G)$.

\subsection{The gauging construction} 
The gauging construction for braided tensor categories containing $\Rep(G)$ is well-established in the literature. We refer to \cite{MR2183964} for the formulation in terms of $G$-crossed braided tensor categories arising from orbifold conformal field theories, and to \cite{MR2183279,MR2609644} for the categorical theory of equivariantization and de-equivariantization of fusion categories.

Suppose $G$ is a finite group, a \textit{$G$-crossed braided tensor category} is a $G$-graded tensor category $\cC=\bigoplus_{g\in G} \cC_g$ equipped with an action $T:G\to\Aut_\otimes(\cC)$, $g\mapsto T_g$ satisfying $T_g(\cC_h)\subseteq \cC_{ghg^{-1}}$, together with a $G$-crossed braiding
$$c_{x,y}:x\otimes y\xrightarrow{\sim} T_g(y)\otimes x,\qquad \text{ for }x\in \cC_g,$$
satisfying the coherences.

Now let $\cB$ be a unitary braided tensor category containing $\Rep(G)$ as a braided full tensor subcategory, and let $\Fun(G)\in \Rep(G)$ be a commutative Q-system (unital unitary separable Frobenius algebra), and $A$ be the image of $\Fun(G)$ in $\cB$. 
The ($G$-)deequivariantization $\cB_G$ is the category $\cB_A$ of right $A$-modules in $\cB$. 
This category carries a canonical structure of a $G$-crossed braided tensor category.
Concretely, if $X\in\cB_G$ is indecomposable with right $A$-action $\rho_X:X\otimes A\to X$, then the braiding in $\cB$ induces a left $A$-action
\[\lambda_X:=\rho_X\circ \beta^{-1}_{A,X}:A\otimes X\to X.\label{eq:leftaction} \tag{$*$}\]
For indecomposable $X$, there exists a unique $\partial_X\in \Aut_{\alge}(A)\cong \Aut_{\alge}(\Fun(G))\cong G$ such that 
\[\lambda_X\circ (\partial_X\otimes \id_X)=\rho_X\circ \beta_{A,X}\label{eq:leftGaction} \tag{$**$}\]
The assignment $X\mapsto \partial_X$ defines the homogeneous degree of $X$, hence a grading 
$$\cB_G=\bigoplus_{g\in G}(\cB_G)_g.$$
The tensor product over $A$ is compatible with this grading $\partial_{X\otimes_A Y}=\partial_X \partial_Y$, and together with the induced $G$-action and crossed braiding. 
For $Y\in \cB_A$ with right action $\rho_Y:Y\otimes A\to Y$ and $g\in \Aut_\alge(A)$,
$$T_g(Y):=(Y,\rho_Y^g),\qquad\text{where } \rho_Y^g=\rho_Y\circ(\id_Y\otimes g).$$
Then for $X\in \cB_A$ with grading $\partial_X$,
the $G$-crossed braiding $c_{X,Y}:X\otimes_A Y\to T_{\partial_X}(Y)\otimes_A X$ is given by
\[c_{X,Y}:= \left(p_{X,T_g(Y)}^A\right)^\dag\circ \beta_{X,Y}\circ p_{X,Y}^A,\label{eq:GcrossedBraid} \tag{$\dag$}\]
where $p_{X,Y}:X\otimes Y\to X\otimes_A Y$ is the projection given by the coequalizer $X\otimes A\otimes Y \rightrightarrows X\otimes Y$.

Conversely, if $\cC$ is a $G$-crossed braided tensor category, one may form its $G$-equivariantization $\cC^G$, whose objects are pairs $(x,\{u_g\}_{g\in G})$ consisting of an object $x\in \cC$ together with coherent isomorphisms $u_g:T_g(x)\xrightarrow{\sim}x$.
The $G$-equivariantization $\cC^G$ is again a braided tensor category, and it contains $\Rep(G)$ as a braided full tensor subcategory.
Deequivariantization and equivariantization establish a fundamental correspondence:
If $\cB$ is a braided fusion category containing $\Rep(G)$ as a braided fusion subcategory, then its deequivariantization $\cB_G$ is a $G$-crossed braided fusion category.
Conversely, if $\cC$ is a $G$-crossed braided fusion category, then its equivariantization $\cC^G$ is a braided fusion category containing $\Rep(G)$.
These processes are inverses up to equivalence: $(\cB_G)^G \cong \cB$ and $(\cC^G)_G \cong \cC$.

This correspondence provides an equivalence between:
\[
\begin{tikzpicture}[node distance=5cm, every node/.style={align=center}]
\node (A) {$\left\{ \parbox{4.5cm}{\rm \centering Braided tensor categories \\ containing $\Rep(G)$ \\ as braided subcategory} \right\}$};
\node (B) [right=of A] {$\left\{ \parbox{3.5cm}{\centering $G$-crossed braided \\ tensor categories} \right\}$};
\draw[->] ([yshift=3pt] A.east) -- ([yshift=3pt] B.west) 
    node[midway, above] {deequivariantization};
\draw[<-] ([yshift=-3pt] A.east) -- ([yshift=-3pt] B.west) 
    node[midway, below] {$G$-equivariantization};
\end{tikzpicture}
\]

From the viewpoint of physics, this pair of constructions is the categorical form of gauging a finite global symmetry in a $2+1$-dimensional topological phase. 
A braided fusion category describes the anyons of the ungauged phase, while a $G$-crossed braided extension incorporates extrinsic symmetry defects labeled by group elements $g\in G$.
Passing from the defect theory to the equivariantized category corresponds to promoting the global symmetry $G$ to a local gauge symmetry: the formerly extrinsic $G$-defects become dynamical excitations of the gauged theory. Conversely, deequivariantization may be viewed as “ungauging,” in which one identifies a Tannakian $\Rep(G)$ sector and passes back to the defect-enriched theory. This gauging paradigm plays a central role in the modern description of symmetry-enriched topological phases, symmetry defects, see e.g. \cite{10.1103}.

\subsection{Graphical calculus and realization}
In order to compute the deequivariantization on $\Chi(M\rtimes G)$, we use the graphical calculus. 
We refer readers the graphical calculus and realization to \cite{MR4419534}.
For the part of the graphical calculus we need here, we refer the section 4 and section 5 from \cite{MR4753059}.

We use the 2D graphical calculus on $\Bim(M)$.
We denote ${}_{\bbC}M_M$ by a dashed line with shading on the right hand side of the line by $M$.
$$
\tikzmath{\filldraw[\rColor, rounded corners=5, very thin, baseline=1cm] (0,0) rectangle (.6,.6);}=M
\qquad
\qquad
\tikzmath{
\filldraw[white, rounded corners=5, very thin, baseline=1cm] (0,0) rectangle (.6,.6);
\draw[dotted, rounded corners=5pt] (0,0) rectangle (.6,.6);
}=\bbC
\qquad
\qquad
\tikzmath{
\begin{scope}
\clip[rounded corners = 5] (0,0) rectangle (.6,.6);
\filldraw[\rColor] (.3,0) rectangle (.6,.6);
\end{scope}
\draw[dashed] (.3,0) -- (.3,.6);
\draw[dotted, rounded corners=5pt] (.3,0) -- (0,0) -- (0,.6) -- (.3,.6);
}={}_\bbC M_M.
\qquad\qquad
\tikzmath{
\filldraw[\rColor, rounded corners=5, very thin, baseline=1cm] (0,0) rectangle (.6,.6);
\draw[\XColor,thick] (.3,0) -- (.3,.6);
}=X
$$

For $X\in \Bim(M)$, the realization 
$|X|:=\Hom_{\bbC-M}({}_\bbC M_M \to {}_\bbC M \boxtimes_M X_M)$,
whose elements are denoted by
\[
\tikzmath{
\begin{scope}
\clip[rounded corners = 5] (-.3,-.7) rectangle (.7,.7);
\filldraw[\rColor] (0,-.7) -- (0,0) -- (-.2,0) -- (-.2,.7) -- (.7,.7) -- (.7,-.7);
\end{scope}
\draw[dashed] (0,-.7) -- (0,0);
\draw[dashed] (-.2,0) -- (-.2,.7);
\draw[\XColor,thick] (.2,0) -- (.2,.7);
\roundNbox{unshaded}{(0,0)}{.3}{.1}{.1}{\scriptsize{$x$}};
}
\in |X| 
:=
\Hom_{\bbC - M}({}_{\bbC}M_M \to {}_{\bbC}M\boxtimes_M X_M).
\]

We say $\{x_i\}\subseteq |X|$ is an $|X|_M$ basis, if
\[
\sum_i |x_i\rangle\langle x_i| 
=\sum_i 
\tikzmath{
\begin{scope}
\clip[rounded corners = 5] (-.4,-1.2) rectangle (.7,1.2);
\filldraw[\rColor] (-.2,-1.2) -- (-.2,-.5) -- (0,-.5) -- (0,.5) -- (-.2,.5) -- (-.2,1.2) -- (.7,1.2) -- (.7,-1.2);
\end{scope}
\draw[dashed] (-.2,.5) -- (-.2,1.2);
\draw[dashed] (0,-.5) -- (0,.5);
\draw[dashed] (-.2,-1.2) -- (-.2,-.5);

\draw[\XColor,thick] (.2,.5) -- (.2,1.2);
\draw[\XColor,thick] (.2,-1.2) -- (.2,-.5);
\roundNbox{unshaded}{(0,.5)}{.3}{.1}{.1}{\scriptsize{$x_i$}};
\roundNbox{unshaded}{(0,-.5)}{.3}{.1}{.1}{\scriptsize{$x_i^\dag$}};
}
=
\tikzmath{
\begin{scope}
\clip[rounded corners = 5] (-.55,-.5) rectangle (.3,.5);
\filldraw[\rColor] (-.4,-.5) -- (-.4,.5) -- (.7,.5) -- (.7,-.5);
\end{scope}
\draw[dashed] (-.4,-.5) -- (-.4,.5);
\draw[\XColor,thick] (0,-.5) -- (0,.5);
}
=\id_{|X|}.
\]
Similarly, if $X\in\aiBim(M)$, we say $\{x_i^{(n)}\}$ is an approximately inner $|X|_M$ basis if
\[
\sum_i 
\tikzmath{
\begin{scope}
\clip[rounded corners = 5] (-.4,-1.2) rectangle (.7,1.2);
\filldraw[\rColor] (-.2,-1.2) -- (-.2,-.5) -- (0,-.5) -- (0,.5) -- (-.2,.5) -- (-.2,1.2) -- (.7,1.2) -- (.7,-1.2);
\end{scope}
\draw[dashed] (-.2,.5) -- (-.2,1.2);
\draw[dashed] (0,-.5) -- (0,.5);
\draw[dashed] (-.2,-1.2) -- (-.2,-.5);

\draw[\XColor,thick] (.2,.5) -- (.2,1.2);
\draw[\XColor,thick] (.2,-1.2) -- (.2,-.5);
\roundNbox{unshaded}{(0,.5)}{.3}{.1}{.1}{\scriptsize{$x_i^{(n)}$}};
\roundNbox{unshaded}{(0,-.5)}{.3}{.1}{.1}{\scriptsize{$(x_i^{(n)})^\dag$}};
}
\approx
\tikzmath{
\begin{scope}
\clip[rounded corners = 5] (-.55,-.5) rectangle (.3,.5);
\filldraw[\rColor] (-.4,-.5) -- (-.4,.5) -- (.7,.5) -- (.7,-.5);
\end{scope}
\draw[dashed] (-.4,-.5) -- (-.4,.5);
\draw[\XColor,thick] (0,-.5) -- (0,.5);
}
=\id_X
\qquad\text{and}
\qquad
a\rhd x_i^{(n)} =
\tikzmath{
\begin{scope}
\clip[rounded corners = 5] (-.3,-.7) rectangle (.7,1.7);
\filldraw[\rColor] (0,-.7) -- (0,0) -- (-.2,0) -- (-.2,1.7) -- (.7,1.7) -- (.7,-.7);
\end{scope}
\draw[dashed] (0,-.7) -- (0,0);
\draw[dashed] (-.2,0) -- (-.2,1.7);
\draw[\XColor,thick] (.2,0) -- (.2,1.7);
\roundNbox{unshaded}{(-.2,1)}{.25}{0}{0}{\scriptsize{$a$}};
\roundNbox{unshaded}{(0,0)}{.3}{.1}{.1}{\scriptsize{$x_i^{(n)}$}};
}
\approx
\tikzmath{
\begin{scope}
\clip[rounded corners = 5] (-.3,-1.7) rectangle (.7,.7);
\filldraw[\rColor] (0,-1.7) -- (0,0) -- (-.2,0) -- (-.2,.7) -- (.7,.7) -- (.7,-1.7);
\end{scope}
\draw[dashed] (0,-1.7) -- (0,0);
\draw[dashed] (-.2,0) -- (-.2,.7);
\draw[\XColor,thick] (.2,0) -- (.2,.7);
\roundNbox{unshaded}{(0,-1)}{.25}{0}{0}{\scriptsize{$a$}};
\roundNbox{unshaded}{(0,0)}{.3}{.1}{.1}{\scriptsize{$x_i^{(n)}$}};
}
=
x_i^{(n)}\lhd a,
\quad \forall a\in M,
\]
where $\approx$ uses the convention in Notation \ref{nota:approx}.

By using the regular basis and approximately inner basis, the braiding in $\Chi(M)$ can be defined as the $\|\cdot\|_2$-limits of $M$-finite rank operators, namely: 
Suppose $\{x_i^{(n)}\}$ and $\{y_j^{(n)}\}$ are approximately inner $X_M$ and $Y_M$ basis, $\{x_i\}$ and $\{y_j\}$ are regular $X_M$ and $Y_M$ basis respectively, then the unitary braiding in $\Chi(M)$ are given by
\[
u_{X,Y} = 
\tikzmath{
\begin{scope}
\clip[rounded corners = 5] (-.55,-.5) rectangle (.7,.5);
\filldraw[\rColor] (-.4,-.5) -- (-.4,.5) -- (.7,.5) -- (.7,-.5);
\end{scope}
\draw[dashed] (-.4,-.5) -- (-.4,.5);
\draw[\YColor,thick] (.4,-.5) node[below]{$\scriptstyle Y$} .. controls ++(90:.45cm) and ++(270:.45cm) .. (0,.5);
\filldraw[\rColor] (.2,0) circle (.05cm);
\draw[\XColor,thick] (0,-.5) node[below]{$\scriptstyle X$} .. controls ++(90:.45cm) and ++(270:.45cm) .. (.4,.5);
}
\approx
\sum_{i,j}
\tikzmath{
\begin{scope}
\clip[rounded corners = 5] (-.35,-2.2) rectangle (1,2.2);
\filldraw[\rColor] (-.2,-2.2) -- (-.2,-1.5) -- (0,-1.5) -- (0,-.5) -- (.3,-.5) -- (.3,.5) -- (0,.5) -- (0,1.5) -- (-.2,1.5) -- (-.2,2.2) -- (1.05,2.2) -- (1.05,-2.2);
\end{scope}
\draw[dashed] (-.2,-2.2) -- (-.2,-1.5);
\draw[\XColor,thick] (.2,-2.2) -- (.2,-1.5);
\draw[dashed] (0,-1.5) -- (0,-.5);
\draw[\YColor,thick] (.6,-2.2) -- (.6,-.5);
\draw[dashed] (.3,-.5) -- (.3,.5);
\draw[dashed] (0,.5) -- (0,1.5);
\draw[\XColor,thick] (.6,.5) -- (.6,2.2);
\draw[dashed] (-.2,1.5) -- (-.2,2.2);
\draw[\YColor,thick] (.2,1.5) -- (.2,2.2);
\roundNbox{unshaded}{(0,1.5)}{.3}{.1}{.1}{\scriptsize{$y_j$}};
\roundNbox{unshaded}{(.3,.5)}{.3}{.15}{.15}{\scriptsize{$x_i^{(n)}$}};
\roundNbox{unshaded}{(.3,-.5)}{.3}{.15}{.15}{\scriptsize{$y_j^\dag$}};
\roundNbox{unshaded}{(0,-1.5)}{.3}{.1}{.1}{\scriptsize{$(x_i^{(\!n\!)}\!)^\dag$}};
}
\qquad\qquad
u_{Y,X}^\dag =
\tikzmath{
\begin{scope}
\clip[rounded corners = 5] (-.55,-.5) rectangle (.7,.5);
\filldraw[\rColor] (-.4,-.5) -- (-.4,.5) -- (.7,.5) -- (.7,-.5);
\end{scope}
\draw[dashed] (-.4,-.5) -- (-.4,.5);
\draw[\XColor,thick] (0,-.5) .. controls ++(90:.45cm) and ++(270:.45cm) .. (.4,.5);
\filldraw[\rColor] (.2,0) circle (.05cm);
\draw[\YColor,thick] (.4,-.5) .. controls ++(90:.45cm) and ++(270:.45cm) .. (0,.5);
}
\approx
\sum_{i,j}
\tikzmath{
\begin{scope}
\clip[rounded corners = 5] (-.35,-2.2) rectangle (1,2.2);
\filldraw[\rColor] (-.2,-2.2) -- (-.2,-1.5) -- (0,-1.5) -- (0,-.5) -- (.3,-.5) -- (.3,.5) -- (0,.5) -- (0,1.5) -- (-.2,1.5) -- (-.2,2.2) -- (1.05,2.2) -- (1.05,-2.2);
\end{scope}
\draw[dashed] (-.2,-2.2) -- (-.2,-1.5);
\draw[\XColor,thick] (.2,-2.2) -- (.2,-1.5);
\draw[dashed] (0,-1.5) -- (0,-.5);
\draw[\YColor,thick] (.6,-2.2) -- (.6,-.5);
\draw[dashed] (.3,-.5) -- (.3,.5);
\draw[dashed] (0,.5) -- (0,1.5);
\draw[\XColor,thick] (.6,.5) -- (.6,2.2);
\draw[dashed] (-.2,1.5) -- (-.2,2.2);
\draw[\YColor,thick] (.2,1.5) -- (.2,2.2);
\roundNbox{unshaded}{(0,1.5)}{.3}{.1}{.1}{\scriptsize{$y_j^{(n)}$}};
\roundNbox{unshaded}{(.3,.5)}{.3}{.15}{.15}{\scriptsize{$x_i$}};
\roundNbox{unshaded}{(.3,-.5)}{.3}{.15}{.15}{\scriptsize{$(y_j^{(\!n\!)}\!)^\dag$}};
\roundNbox{unshaded}{(0,-1.5)}{.3}{.1}{.1}{\scriptsize{$x_i^\dag$}};
}
\]

Suppose $M$ is a $\rm II_1$ factor.
A Q-system in $\Bim(M)$ is a triple $(Q,M,i)$, where $Q\in \Bim(M)$, the multiplication $m:Q\boxtimes_M Q\to Q$ and the unit $i:M\to Q$ are bounded bimodule intertwiners such that $Q$ is a unitary separable Frobenius algebra.
Its realization 
$$|Q| := \Hom_{\bbC - M}({}_{\bbC}M_M \to {}_{\bbC}M\boxtimes_M Q_M),$$
is a von Neumann algebra containing $M$ with a conditional expectation $E_M:|Q|\to M$. 
When $Q$ is connected, i.e., $\dim(\Hom_{M-M}(M\to Q))=1$, then $|Q|$ is also a $\rm II_1$ factor.

\begin{defn}[{\cite[\S4.2]{MR4419534}}]
\label{defn:RealizationOfBimodules}
Suppose $Q\in\Bim(M)$ is a Q-system and $X$ is a $Q$-$Q$ bimodule.
The realization $|X|:=\Hom_{\bbC-M}({}_\bbC M_M \to {}_\bbC M \boxtimes_M X_M)$ of $X$ is a $|Q|$-$|Q|$ bimodule 
with $|Q|$-$|Q|$ bimodule action and right $|Q|$-valued inner product respectively by
\[
p\rhd x\lhd q :=
\tikzmath{
\begin{scope}
\clip[rounded corners = 5] (-1,-1.7) rectangle (.9,2);
\filldraw[\rColor] (.2,-1.7) -- (.2,-1) -- (0,-.7) .. controls ++(90:.2cm) and ++(270:.2cm) .. (-.2,-.3) -- (-.4,.3) .. controls ++(90:.2cm) and ++(270:.2cm) .. (-.6,.7) -- (-.8,1) -- (-.8,2.8) -- (.9,2.8) -- (.9,-1.7);
\end{scope}
\draw[dashed] (.2,-1.7) -- (.2,-1);
\draw[dashed] (0,-.7) .. controls ++(90:.2cm) and ++(270:.2cm) .. (-.2,-.3);
\draw[dashed] (-.4,.3) .. controls ++(90:.2cm) and ++(270:.2cm) .. (-.6,.7);
\draw[dashed] (-.8,1) -- (-.8,2);

\draw[\QsColor,thick] (-.4,1.3) arc (180:90:.4cm);
\draw[\XColor,thick] (0,0) -- (0,2); 
\draw[\QsColor,thick] (.4,-.7) -- (.4,.3) arc (0:90:.4cm);
\filldraw[\XColor] (0,1.7) circle (.05cm);
\filldraw[\XColor] (0,.7) circle (.05cm);
\roundNbox{unshaded}{(-.6,1)}{.3}{.1}{.1}{\scriptsize{$p$}};
\roundNbox{unshaded}{(-.2,0)}{.3}{.1}{.1}{\scriptsize{$x$}};
\roundNbox{unshaded}{(.2,-1)}{.3}{.1}{.1}{\scriptsize{$q$}};
}
\qquad\text{ and }\qquad
\langle x_1|x_2\rangle_{|Q|}^{|X|}
:=
\tikzmath{
\begin{scope}
\clip[rounded corners = 5] (-.4,-1.4) rectangle (.9,1.4);
\filldraw[\rColor] (0,-1.4) -- (0,-.5) -- (-.2,-.5) -- (-.2,.5) -- (0,.5) -- (0,1.4) -- (1.1,1.4) -- (1.1,-1.4);
\end{scope}
\draw[dashed] (0,.5) -- (0,1.4);
\draw[dashed] (0,-1.4) -- (0,-.5);
\draw[dashed] (-.2,-.5) -- (-.2,.5);
\draw[\XColor,thick] (.2,-.5) -- (.2,.5);
\draw[\QsColor,thick] (.2,0) arc (-90:0:.4cm) -- (.6,1.4);
\filldraw[\XColor] (.2,0) circle (.05cm);
\roundNbox{unshaded}{(0,.7)}{.3}{.1}{.1}{\scriptsize{$x_1^\dag$}};
\roundNbox{unshaded}{(0,-.7)}{.3}{.1}{.1}{\scriptsize{$x_2$}};
}
\]
Hence $|X|\in \Bim(|Q|)$ is a $\rm W^*$-correspondence. 
\end{defn}

\begin{rem}
Given a group action $\psi:G\to \Aut(M)$, this induces a fully faithful tensor functor $F:\Hilb(G)\to \Bim(M)$ by $g\mapsto {}_{\psi_g}L^2(M)$. 
Let $Q=\bbC[G]\in \Hilb(G)$. 
We know that the category of $Q$-$Q$ bimodules in $\Hilb(G)$ is tensor equivalent to $\Rep(G)$. 
Therefore, the realization $|\cdot|_F$ induces a tensor embedding $\Rep(G)\cong\Bim(Q)\to \Bim(|Q|)=\Bim(M\rtimes G)$. This functor coincides with the one described at the beginning of section \ref{section:Rep(G)subcat}.
\end{rem}

\begin{facts}\label{facts:CommQsysInChi}
Suppose $Q\in \Chi(M)$ is a commutative Q-system and $X\in\aiBim(|Q|)$.
The following facts based on \cite[Prop.~5.10 \& 5.13]{MR4753059} hold:
\begin{enumerate}
\item If $\{q_j^{(n)}\}$ is an approximately inner $|Q|_M$ basis, then $(q_j^{(n)})_n$ is a central sequence in $|Q|$.
\item Suppose $\{x_i^{(n)}\}$ is an approximately inner $|X|_{|Q|}$ basis, then 
$\{c_{i,j}^{(n)}\}=\{x_i^{(n)}\lhd q_j^{(n)}\}$ is an approximately inner $|X|_M$ basis.
\item Suppose $\{x_i\}$ is an $|X|_{|Q|}$ basis, and $\{q_j\}$ is a $|Q|_M$ basis, then $\{c_{i,j}\}=\{x_i\lhd q_j\}$ is a $|X|_M$ basis.
\item If $\{x_i\}$ is an $X_M$ basis, then $\{x_i\}$ is also an $|X|_{|Q|}$ basis; if $\{x_i\}$ is an approximate $X_M$ basis, then $\{x_i\}$ is also an approximate $|X|_{|Q|}$ basis.
\end{enumerate}
\end{facts}

\subsection{Deequivariantization on $\Chi(M\rtimes G)$}

Suppose $\psi:G\to \Aut(M)$ is outer and $\psi$ is not approximately inner.
According to Theorem \ref{Thm:Rep(G)subcatChi}, $\Rep(G/K)$ is a unitary braided fully subcategory of $\Chi(M\rtimes G)$.

Let $A$ be the image of $\Fun(G/K)$ in $\Chi(M\rtimes G)$, i.e., $A=\Fun(G/K)\otimes L^2(M\rtimes G)$.

Suppose $X\in \Chi(M\rtimes G)_A$, by (\ref{eq:leftaction}), the left $A$-module structure is given by
\[
\lambda_X=
\tikzmath{
\begin{scope}
\clip[rounded corners = 5] (-.7,0) rectangle (.3,1);
\filldraw[\rColor] (-.7,-.5) -- (-.7,1.3) -- (.7,1.3) -- (.7,-.5);
\end{scope}
\draw[\XColor,thick] (0,0) -- (0,1);
\draw[\QsColor,thick] (-.4,0) -- (-.4,.2) arc (180:90:.4cm);
\filldraw[\XColor] (0,.6) circle (.05cm);
}
: = 
\tikzmath{
\begin{scope}
\clip[rounded corners = 5] (-.3,-.5) rectangle (.7,1.3);
\filldraw[\rColor] (-.4,-.5) -- (-.4,1.3) -- (.7,1.3) -- (.7,-.5);
\end{scope}
\draw[\QsColor,thick] (0,-.5) .. controls ++(90:.45cm) and ++(270:.45cm) .. (.4,.5) arc (0:90:.4cm);
\filldraw[\rColor] (.2,0) circle (.05cm);
\draw[\XColor,thick] (.4,-.5) .. controls ++(90:.45cm) and ++(270:.45cm) .. (0,.5) -- (0,1.3);
\filldraw[\XColor] (0,.9) circle (.05cm);
}
=\rho_X\circ u_{A,X}^\dag
\]
By (\ref{eq:leftGaction}), for indecomposable right $A$-module $X\in \Chi(M\rtimes G)_A$, there exists $\partial_X\in \Aut(A)$ such that 
\[
\rho_X\circ u_{A,X}^\dag=
\tikzmath{
\begin{scope}
\clip[rounded corners = 5] (-.7,2.3) rectangle (.3,4);
\filldraw[\rColor] (-.7,2.3) rectangle (.3,4);
\end{scope}
\draw[\XColor,thick] (0,2.3) .. controls ++(90:.45cm) and ++(270:.45cm) .. (-.4,3.3) -- (-.4,4);
\draw[\QsColor,thick,knotrColor] (-.4,2.3) .. controls ++(90:.45cm) and ++(270:.45cm) .. (0,3.3) arc (0:90:.4cm);
\filldraw[\XColor] (-.4,3.7) circle (.05cm);
}
=
\tikzmath{
\begin{scope}
\clip[rounded corners = 5] (-.9,1.4) rectangle (.3,4);
\filldraw[\rColor] (-.9,1.4) rectangle (.3,4);
\end{scope}
\draw[\QsColor,thick] (-.4,1.4) -- (-.4,2.3) .. controls ++(90:.45cm) and ++(270:.45cm) .. (0,3.3) arc (0:90:.4cm);
\draw[\XColor,thick,knotrColor] (0,1.4) -- (0,2.3) .. controls ++(90:.45cm) and ++(270:.45cm) .. (-.4,3.3) -- (-.4,4);
\filldraw[\XColor] (-.4,3.7) circle (.05cm);
\roundNbox{unshaded}{(-.4,2)}{.25}{0}{0}{\scriptsize{$\partial_X$}};
}
=
\rho_X\circ u_{X,A}\circ (\partial_X\boxtimes \id_X)
\]

We first show that the realization gives a fully faithful embedding from the deequivariantization of $\Chi(M\rtimes G)$ to $\bigoplus_{g\in G/K}g\ctBim(|A|)\cap \aiBim(|A|)$.
\begin{prop}\label{prop:AiAmod->Ai|A|}
If the bimodule $X$ a right $A$-module in $\aiBim(M\rtimes G)$,
then $|X|\in\aiBim(|A|)$.
\end{prop}
\begin{proof}
Since $X$ is a right $A$-module and $X$ can be equipped with a left $A$-module structure, $X$ is an $A$-$A$ bimodule. 
The the realization $|X|$ is an $|A|$-$|A|$ bimodule.

Suppose $\{x_i^{(n)}\}$ is an approximate inner $X_{M\rtimes G}$-basis, $\{q_j\}$ is an $A_{M\rtimes G}$-basis. 
For every $q\in |A|$,
\[
q x_i^{(n)}
=
\tikzmath{
\begin{scope}
\clip[rounded corners = 5] (-.45,-1.2) rectangle (1.05,1.6);
\filldraw[\rColor] (.3,-1.2) -- (.3,-.5) -- (0,-.5) -- (0,.5) -- (-.2,.5) -- (-.2,2) -- (1.05,2) -- (1.05,-1.2);
\end{scope}
\draw[dashed] (-.2,.5) -- (-.2,1.6);
\draw[dashed] (0,-.5) -- (0,.5);
\draw[dashed] (.3,-1.2) -- (.3,-.5);
\draw[\QsColor,thick] (.2,.8) arc (180:90:.4cm);
\draw[\XColor,thick] (.6,-.5) -- (.6,1.6);
\filldraw[\XColor] (.6,1.2) circle (.05cm);
\roundNbox{unshaded}{(0,.5)}{.3}{.1}{.1}{\scriptsize{$q$}};
\roundNbox{unshaded}{(.3,-.5)}{.3}{.15}{.15}{\scriptsize{$x_i^{(n)}$}};
}
=
\tikzmath{
\begin{scope}
\clip[rounded corners = 5] (-1,-.7) rectangle (.5,3.1);
\filldraw[\rColor] (-.2,-.7) -- (-.2,0) -- (-.4,.3) .. controls ++(90:.2cm) and ++(270:.2cm) .. (-.6,.7) -- (-.8,1) -- (-.8,3.1) -- (.9,3.1) -- (.9,-.7);
\end{scope}
\draw[dashed] (-.2,-.7) -- (-.2,0);
\draw[dashed] (-.4,.3) .. controls ++(90:.2cm) and ++(270:.2cm) .. (-.6,.7);
\draw[dashed] (-.8,1) -- (-.8,3.1);
\draw[\QsColor,thick] (-.4,1.3) .. controls ++(90:.45cm) and ++(270:.45cm) .. (0,2.3) arc (0:90:.4cm);
\filldraw[\rColor] (-.2,1.8) circle (.05cm);
\filldraw[\XColor] (-.4,2.7) circle (.05cm);
\draw[\XColor,thick] (0,0) -- (0,1.3) .. controls ++(90:.45cm) and ++(270:.45cm) .. (-.4,2.3) -- (-.4,3.1);
\roundNbox{unshaded}{(-.6,1)}{.3}{.1}{.1}{\scriptsize{$q$}};
\roundNbox{unshaded}{(-.2,0)}{.3}{.1}{.1}{\scriptsize{$x_i^{(n)}$}};
}
\approx
\sum_{j,k}
\tikzmath{
\begin{scope}
\clip[rounded corners = 5] (-.45,-4.2) rectangle (1,2.6);
\filldraw[\rColor] (.3,-4.2) -- (.3,-3.5) -- (0,-3.5) -- (0,-2.5) -- (-.2,-2.5) -- (-.2,-1.5) -- (0,-1.5) -- (0,-.5) -- (.3,-.5) -- (.3,.5) -- (0,.5) -- (0,1.5) -- (-.2,1.5) -- (-.2,2.6) -- (1,2.6) -- (1,-4.2);
\end{scope}
\draw[dashed] (.3,-4.2) -- (.3,-3.5);
\draw[dashed] (0,-3.5) -- (0,-2.5);
\draw[dashed] (-.2,-2.5) -- (-.2,-1.5);
\draw[\QsColor,thick] (.2,-2.5) -- (.2,-1.5);
\draw[dashed] (0,-1.5) -- (0,-.5);
\draw[\XColor,thick] (.6,-3.5) -- (.6,-.5);
\draw[dashed] (.3,-.5) -- (.3,.5);
\draw[dashed] (0,.5) -- (0,1.5);
\draw[\QsColor,thick] (.6,.5) -- (.6,1.8) arc (0:90:.4cm);
\draw[dashed] (-.2,1.5) -- (-.2,2.2);
\draw[\XColor,thick] (.2,1.5) -- (.2,2.6);
\filldraw[\XColor] (.2,2.2) circle (.05cm);
\roundNbox{unshaded}{(0,1.5)}{.3}{.1}{.1}{\scriptsize{$x_j^{(n)}$}};
\roundNbox{unshaded}{(.3,.5)}{.3}{.15}{.15}{\scriptsize{$q_k$}};
\roundNbox{unshaded}{(.3,-.5)}{.3}{.15}{.15}{\scriptsize{$(x_j^{(n)})^\dag$}};
\roundNbox{unshaded}{(0,-1.5)}{.3}{.1}{.1}{\scriptsize{$q_k^\dag$}};
\roundNbox{unshaded}{(0,-2.5)}{.3}{.1}{.1}{\scriptsize{$q$}};
\roundNbox{unshaded}{(.3,-3.5)}{.3}{.15}{.15}{\scriptsize{$x_i^{(n)}$}};
}
\approx
\sum_{j,k}
\tikzmath{
\begin{scope}
\clip[rounded corners = 5] (-.45,-4.2) rectangle (1,2.6);
\filldraw[\rColor] (.3,-4.2) -- (.3,-3.5) -- (0,-3.5) -- (0,-2.5) -- (.3,-2.5) -- (.3,-1.5) -- (0,-1.5) -- (0,-.5) -- (.3,-.5) -- (.3,.5) -- (0,.5) -- (0,1.5) -- (-.2,1.5) -- (-.2,2.6) -- (1,2.6) -- (1,-4.2);
\end{scope}
\draw[dashed] (.3,-4.2) -- (.3,-3.5);
\draw[dashed] (0,-3.5) -- (0,-2.5);
\draw[dashed] (.3,-2.5) -- (.3,-1.5);
\draw[\QsColor,thick] (.6,-3.5) -- (.6,-2.5);
\draw[dashed] (0,-1.5) -- (0,-.5);
\draw[\XColor,thick] (.6,-1.5) -- (.6,-.5);
\draw[dashed] (.3,-.5) -- (.3,.5);
\draw[dashed] (0,.5) -- (0,1.5);
\draw[\QsColor,thick] (.6,.5) -- (.6,1.8) arc (0:90:.4cm);
\draw[dashed] (-.2,1.5) -- (-.2,2.6);
\draw[\XColor,thick] (.2,1.5) -- (.2,2.6);
\filldraw[\XColor] (.2,2.2) circle (.05cm);
\roundNbox{unshaded}{(0,1.5)}{.3}{.1}{.1}{\scriptsize{$x_j^{(n)}$}};
\roundNbox{unshaded}{(.3,.5)}{.3}{.15}{.15}{\scriptsize{$q_k$}};
\roundNbox{unshaded}{(.3,-.5)}{.3}{.15}{.15}{\scriptsize{$(x_j^{(n)})^\dag$}};
\roundNbox{unshaded}{(.3,-1.5)}{.3}{.15}{.15}{\scriptsize{$x_i^{(n)}$}};
\roundNbox{unshaded}{(.3,-2.5)}{.3}{.15}{.15}{\scriptsize{$q_k^\dag$}};
\roundNbox{unshaded}{(.3,-3.5)}{.3}{.15}{.15}{\scriptsize{$q$}};
}
\approx
\sum_{j,k}
\tikzmath{
\begin{scope}
\clip[rounded corners = 5] (-.45,-4.2) rectangle (1,2.6);
\filldraw[\rColor] (.3,-4.2) -- (.3,-3.5) -- (0,-3.5) -- (0,-2.5) -- (.3,-2.5) -- (.3,-1.5) -- (0,-1.5) -- (0,-.5) -- (-.2,-.5) -- (-.2,.5) -- (0,.5) -- (0,1.5) -- (-.2,1.5) -- (-.2,2.6) -- (1,2.6) -- (1,-4.2);
\end{scope}
\draw[dashed] (.3,-4.2) -- (.3,-3.5);
\draw[dashed] (0,-3.5) -- (0,-2.5);
\draw[dashed] (.3,-2.5) -- (.3,-1.5);
\draw[\QsColor,thick] (.6,-3.5) -- (.6,-2.5);
\draw[dashed] (0,-1.5) -- (0,-.5);
\draw[\XColor,thick] (.2,-.5) -- (.2,.5);
\draw[dashed] (-.2,-.5) -- (-.2,.5);
\draw[dashed] (0,.5) -- (0,1.5);
\draw[\QsColor,thick] (.6,-1.5) -- (.6,1.8) arc (0:90:.4cm);
\draw[dashed] (-.2,1.5) -- (-.2,2.6);
\draw[\XColor,thick] (.2,1.5) -- (.2,2.6);
\filldraw[\XColor] (.2,2.2) circle (.05cm);
\roundNbox{unshaded}{(0,1.5)}{.3}{.1}{.1}{\scriptsize{$x_j^{(n)}$}};
\roundNbox{unshaded}{(0,.5)}{.3}{.1}{.1}{\scriptsize{$(x_j^{(\!n\!)})^\dag$}};
\roundNbox{unshaded}{(0,-.5)}{.3}{.1}{.1}{\scriptsize{$x_i^{(n)}$}};
\roundNbox{unshaded}{(.3,-1.5)}{.3}{.15}{.15}{\scriptsize{$q_k$}};
\roundNbox{unshaded}{(.3,-2.5)}{.3}{.15}{.15}{\scriptsize{$q_k^\dag$}};
\roundNbox{unshaded}{(.3,-3.5)}{.3}{.15}{.15}{\scriptsize{$q$}};
}
\approx
\tikzmath{
\begin{scope}
\clip[rounded corners = 5] (-.45,-1.2) rectangle (1.05,1.6);
\filldraw[\rColor] (.3,-1.2) -- (.3,-.5) -- (0,-.5) -- (0,.5) -- (-.2,.5) -- (-.2,1.6) -- (1.05,1.6) -- (1.05,-1.2);
\end{scope}
\draw[dashed] (-.2,.5) -- (-.2,1.6);
\draw[dashed] (0,-.5) -- (0,.5);
\draw[dashed] (.3,-1.2) -- (.3,-.5);
\draw[\XColor,thick] (.2,.5) -- (.2,1.6);
\draw[\QsColor,thick] (.6,-.5) -- (.6,.8) arc (0:90:.4cm);
\filldraw[\XColor] (.2,1.2) circle (.05cm);
\roundNbox{unshaded}{(0,.5)}{.3}{.1}{.1}{\scriptsize{$x_i^{(n)}$}};
\roundNbox{unshaded}{(.3,-.5)}{.3}{.15}{.15}{\scriptsize{$q$}};
}
= 
x_i^{(n)} q
\]
The second equality uses the definition of left action (\ref{eq:leftaction});
the third approximation is from the definition of the under braiding $u_{X,A}^\dag$;
the forth approximation is from the fact that $\{x_i^{(n)}\}$ is approximately inner $X_{M\rtimes G}$ basis and $\langle q_k\mid q\rangle\in M\rtimes G$;
the fifth approximation is because $\left(\langle x_j^{(n)}\mid x_i^{(n)}\rangle\right)$ is a central sequence in $(M\rtimes G)_\omega$ and $|Q|$ is centrally trivial.

Therefore $\{x_i^{(n)}\}$ is an approximately inner $|X|_{|A|}$ basis.
\end{proof}

For $g\in \Aut_\alge(A)$, we denote $g\ctBim(|A|)$ to the category of $|A|$-$|A|$ bimodules $(X,\lambda^g_X)$, where $X\in \ctBim(|A|)$ with the same right $|A|$-action and a new left $|A|$-action $\lambda^g_X:=\lambda_X\circ (g\otimes \id_X):|A|\boxtimes_{|A|} X\to X$.

\begin{prop}\label{prop:CtAmod->GCt|A|}
If bimodule $X$ is an indecomposable right $A$-module in $\ctBim(M\rtimes G)$,
then $|X|\in {}_{\partial X} \ctBim(|A|)$.
\end{prop}
\begin{proof}
Suppose $\{x_j\}$ is an $X_{M\rtimes G}$-basis and $\{q_k^{(n)}\}$ is an approximately inner $A_{M\rtimes G}$-basis.
For every central sequence $(q_n)\in |Q|_\omega$ and $x\in |X|$.
\[
q_n\rhd_{\partial_X} x
=
\tikzmath{
\begin{scope}
\clip[rounded corners = 5] (-1,-.7) rectangle (.5,4);
\filldraw[\rColor] (-.2,-.7) -- (-.2,0) -- (-.4,.3) .. controls ++(90:.2cm) and ++(270:.2cm) .. (-.6,.7) -- (-.8,1) -- (-.8,4) -- (.9,4) -- (.9,-.7);
\end{scope}
\draw[dashed] (-.2,-.7) -- (-.2,0);
\draw[dashed] (-.4,.3) .. controls ++(90:.2cm) and ++(270:.2cm) .. (-.6,.7);
\draw[dashed] (-.8,1) -- (-.8,4);
\draw[\QsColor,thick] (-.4,1.3) -- (-.4,2.3) .. controls ++(90:.45cm) and ++(270:.45cm) .. (0,3.3) arc (0:90:.4cm);
\draw[\XColor,thick,knotrColor] (0,0) -- (0,2.3) .. controls ++(90:.45cm) and ++(270:.45cm) .. (-.4,3.3) -- (-.4,4);
\filldraw[\XColor] (-.4,3.7) circle (.05cm);
\roundNbox{unshaded}{(-.4,2)}{.25}{0}{0}{\scriptsize{$\partial_X$}};
\roundNbox{unshaded}{(-.6,1)}{.3}{.1}{.1}{\scriptsize{$q_n$}};
\roundNbox{unshaded}{(-.2,0)}{.3}{.1}{.1}{\scriptsize{$x$}};
}
=
\tikzmath{
\begin{scope}
\clip[rounded corners = 5] (-1,-.7) rectangle (.5,3);
\filldraw[\rColor] (-.2,-.7) -- (-.2,0) -- (-.4,.3) .. controls ++(90:.2cm) and ++(270:.2cm) .. (-.6,.7) -- (-.8,1) -- (-.8,3) -- (.9,3) -- (.9,-.7);
\end{scope}
\draw[dashed] (-.2,-.7) -- (-.2,0);
\draw[dashed] (-.4,.3) .. controls ++(90:.2cm) and ++(270:.2cm) .. (-.6,.7);
\draw[dashed] (-.8,1) -- (-.8,3);
\draw[\XColor,thick] (0,0) -- (0,1.3) .. controls ++(90:.45cm) and ++(270:.45cm) .. (-.4,2.3) -- (-.4,3);
\filldraw[\rColor] (-.2,1.8) circle (.05cm);
\draw[\QsColor,thick] (-.4,1.3) .. controls ++(90:.45cm) and ++(270:.45cm) .. (0,2.3) arc (0:90:.4cm);
\filldraw[\XColor] (-.4,2.7) circle (.05cm);
\roundNbox{unshaded}{(-.6,1)}{.3}{.1}{.1}{\scriptsize{$q_n$}};
\roundNbox{unshaded}{(-.2,0)}{.3}{.1}{.1}{\scriptsize{$x$}};
}
\approx
\sum_{j,k}
\tikzmath{
\begin{scope}
\clip[rounded corners = 5] (-.45,-4.2) rectangle (1,2.5);
\filldraw[\rColor] (.3,-4.2) -- (.3,-3.5) -- (0,-3.5) -- (0,-2.5) -- (-.2,-2.5) -- (-.2,-1.5) -- (0,-1.5) -- (0,-.5) -- (.3,-.5) -- (.3,.5) -- (0,.5) -- (0,1.5) -- (-.2,1.5) -- (-.2,2.5) -- (1,2.5) -- (1,-4.2);
\end{scope}
\draw[dashed] (.3,-4.2) -- (.3,-3.5);
\draw[dashed] (0,-3.5) -- (0,-2.5);
\draw[dashed] (-.2,-2.2) -- (-.2,-1.5);
\draw[\QsColor,thick] (.2,-2.2) -- (.2,-1.5);
\draw[dashed] (0,-1.5) -- (0,-.5);
\draw[\XColor,thick] (.6,-3.5) -- (.6,-.5);
\draw[dashed] (.3,-.5) -- (.3,.5);
\draw[dashed] (0,.5) -- (0,1.5);
\draw[\QsColor,thick] (.6,.5) -- (.6,1.8) arc (0:90:.4cm);
\draw[dashed] (-.2,1.5) -- (-.2,2.5);
\draw[\XColor,thick] (.2,1.5) -- (.2,2.5);
\filldraw[\XColor] (.2,2.2) circle (.05cm);
\roundNbox{unshaded}{(0,1.5)}{.3}{.1}{.1}{\scriptsize{$x_j$}};
\roundNbox{unshaded}{(.3,.5)}{.3}{.15}{.15}{\scriptsize{$q_k^{(n)}$}};
\roundNbox{unshaded}{(.3,-.5)}{.3}{.15}{.15}{\scriptsize{$x_j^\dag$}};
\roundNbox{unshaded}{(0,-1.5)}{.3}{.1}{.1}{\scriptsize{$(q_k^{(\!n\!)})^\dag$}};
\roundNbox{unshaded}{(0,-2.5)}{.3}{.1}{.1}{\scriptsize{$q_n$}};
\roundNbox{unshaded}{(.3,-3.5)}{.3}{.15}{.15}{\scriptsize{$x$}};
}
\approx
\sum_{j,k}
\tikzmath{
\begin{scope}
\clip[rounded corners = 5] (-.45,-4.2) rectangle (1,2.5);
\filldraw[\rColor] (.3,-4.2) -- (.3,-3.5) -- (0,-3.5) -- (0,-2.5) -- (.3,-2.5) -- (.3,-1.5) -- (0,-1.5) -- (0,-.5) -- (.3,-.5) -- (.3,.5) -- (0,.5) -- (0,1.5) -- (-.2,1.5) -- (-.2,2.5) -- (1,2.5) -- (1,-4.2);
\end{scope}
\draw[dashed] (.3,-4.2) -- (.3,-3.5);
\draw[dashed] (0,-3.5) -- (0,-2.5);
\draw[dashed] (.3,-2.5) -- (.3,-1.5);
\draw[\QsColor,thick] (.6,-3.5) -- (.6,-2.5);
\draw[dashed] (0,-1.5) -- (0,-.5);
\draw[\XColor,thick] (.6,-1.5) -- (.6,-.5);
\draw[dashed] (.3,-.5) -- (.3,.5);
\draw[dashed] (0,.5) -- (0,1.5);
\draw[\QsColor,thick] (.6,.5) -- (.6,1.8) arc (0:90:.4cm);
\draw[dashed] (-.2,1.5) -- (-.2,2.5);
\draw[\XColor,thick] (.2,1.5) -- (.2,2.5);
\filldraw[\XColor] (.2,2.2) circle (.05cm);
\roundNbox{unshaded}{(0,1.5)}{.3}{.1}{.1}{\scriptsize{$x_j$}};
\roundNbox{unshaded}{(.3,.5)}{.3}{.15}{.15}{\scriptsize{$q_k^{(n)}$}};
\roundNbox{unshaded}{(.3,-.5)}{.3}{.15}{.15}{\scriptsize{$x_j^\dag$}};
\roundNbox{unshaded}{(.3,-1.5)}{.3}{.15}{.15}{\scriptsize{$x$}};
\roundNbox{unshaded}{(.3,-2.5)}{.3}{.15}{.15}{\scriptsize{$(q_k^{(\!n\!)})^\dag$}};
\roundNbox{unshaded}{(.3,-3.5)}{.3}{.15}{.15}{\scriptsize{$q_n$}};
}
\approx
\sum_{j,k}
\tikzmath{
\begin{scope}
\clip[rounded corners = 5] (-.45,-4.2) rectangle (1,2.5);
\filldraw[\rColor] (.3,-4.2) -- (.3,-3.5) -- (0,-3.5) -- (0,-2.5) -- (.3,-2.5) -- (.3,-1.5) -- (0,-1.5) -- (0,-.5) -- (-.2,-.5) -- (-.2,.5) -- (0,.5) -- (0,1.5) -- (-.2,1.5) -- (-.2,2.5) -- (1,2.5) -- (1,-4.2);
\end{scope}
\draw[dashed] (.3,-4.2) -- (.3,-3.5);
\draw[dashed] (0,-3.5) -- (0,-2.5);
\draw[dashed] (.3,-2.5) -- (.3,-1.5);
\draw[\QsColor,thick] (.6,-3.5) -- (.6,-2.5);
\draw[dashed] (0,-1.5) -- (0,-.5);
\draw[\XColor,thick] (.2,-.5) -- (.2,.5);
\draw[dashed] (-.2,-.5) -- (-.2,.5);
\draw[dashed] (0,.5) -- (0,1.5);
\draw[\QsColor,thick] (.6,-1.5) -- (.6,1.8) arc (0:90:.4cm);
\draw[dashed] (-.2,1.5) -- (-.2,2.5);
\draw[\XColor,thick] (.2,1.5) -- (.2,2.5);
\filldraw[\XColor] (.2,2.2) circle (.05cm);
\roundNbox{unshaded}{(0,1.5)}{.3}{.1}{.1}{\scriptsize{$x_j$}};
\roundNbox{unshaded}{(0,.5)}{.3}{.1}{.1}{\scriptsize{$x_j^\dag$}};
\roundNbox{unshaded}{(0,-.5)}{.3}{.1}{.1}{\scriptsize{$x$}};
\roundNbox{unshaded}{(.3,-1.5)}{.3}{.15}{.15}{\scriptsize{$q_k^{(n)}$}};
\roundNbox{unshaded}{(.3,-2.5)}{.3}{.15}{.15}{\scriptsize{$(q_k^{(\!n\!)})^\dag$}};
\roundNbox{unshaded}{(.3,-3.5)}{.3}{.15}{.15}{\scriptsize{$q_n$}};
}
\approx
\tikzmath{
\begin{scope}
\clip[rounded corners = 5] (-.45,-1.2) rectangle (1.05,1.6);
\filldraw[\rColor] (.3,-1.2) -- (.3,-.5) -- (0,-.5) -- (0,.5) -- (-.2,.5) -- (-.2,1.6) -- (1.05,1.6) -- (1.05,-1.2);
\end{scope}
\draw[dashed] (-.2,.5) -- (-.2,1.6);
\draw[dashed] (0,-.5) -- (0,.5);
\draw[dashed] (.3,-1.2) -- (.3,-.5);
\draw[\XColor,thick] (.2,.5) -- (.2,1.6);
\draw[\QsColor,thick] (.6,-.5) -- (.6,.8) arc (0:90:.4cm);
\filldraw[\XColor] (.2,1.2) circle (.05cm);
\roundNbox{unshaded}{(0,.5)}{.3}{.1}{.1}{\scriptsize{$x$}};
\roundNbox{unshaded}{(.3,-.5)}{.3}{.15}{.15}{\scriptsize{$q_n$}};
}
=
x q_n
\]
The first equality uses the definition of left action (\ref{eq:leftaction});
the second equality is by (\ref{eq:leftGaction});
the third approximation is from the definition of the over braiding $u_{A,X}$;
the forth approximation is from the fact that $(\langle q_k^{(n)} \mid q_n\rangle)$ is a central sequence and $X$ is centrally trivial;
the fifth approximation is because $\{q_k^{(n)}\}$ is approximately inner $|Q|_{M\rtimes G}$ basis and $\langle x_j\mid x\rangle\in M\rtimes G$.
\end{proof}

\begin{prop}
Realization gives a $G/K$-crossed braided unitary equivalence
$$| \cdot |:\Chi(M\rtimes G)_A\to \bigoplus_{g\in G/K}g \ctBim(|A|)\cap \aiBim(|A|).$$
\end{prop}
\begin{proof}
By proposition \ref{prop:AiAmod->Ai|A|} and \ref{prop:CtAmod->GCt|A|}, the realization functor is fully faithful. 
Especially, $|X|\in g \ctBim(|A|)\cap \aiBim(|A|)$ when $\partial_X=g$. 

We now show that the realization functor is essentially surjective, i.e, every bimodule in $\bigoplus_{g\in G/K}g \ctBim(|A|)\cap \aiBim(|A|)$ is unitary isomorphic to a realization of a bimodule in $\Chi(M\rtimes G)_A$. 
It is suffices to consider an indecomposable $A$-$A$ bimodule $X\in \Bim(M)$ such that $|X|\in g \ctBim(|A|)\cap \aiBim(|A|)$ for some $g\in G/K$. 

By Facts \ref{facts:CommQsysInChi}(2), suppose $\{x_i^{(n)}\}$ is an approximately inner $|X|_{|A|}$-basis and $\{q_j^{(n)}\}$ is an approximately inner $|A|_{M\rtimes G}$-basis, then $\{x_i^{(n)}\lhd q_j^{(n)}\}$ is an approximately inner $|X|_M$-basis. 

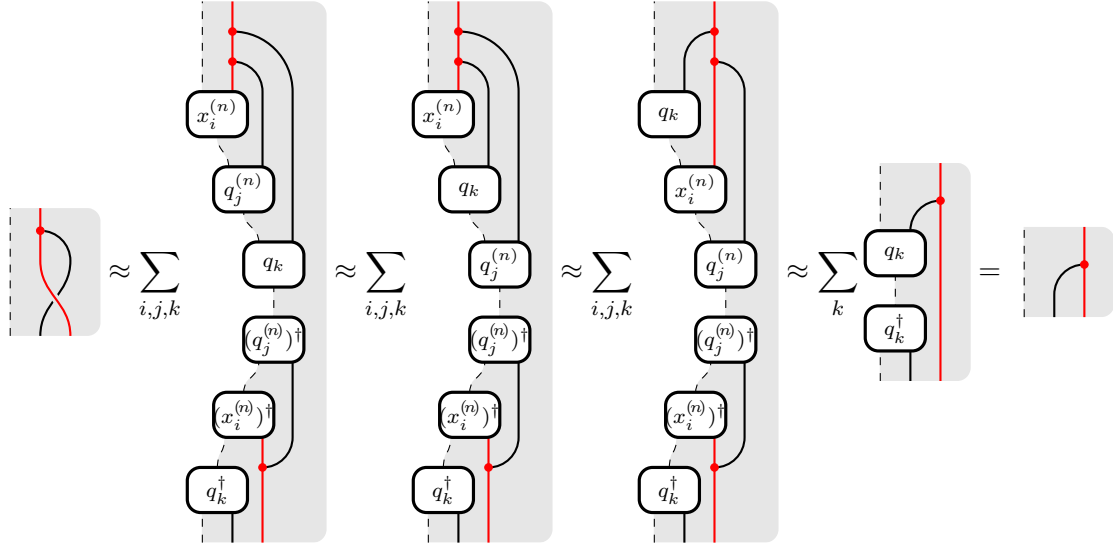
\[
\tikzmath{
\begin{scope}
\clip[rounded corners = 5] (-1,1.3) rectangle (.4,3);
\filldraw[\rColor] (-.2,-.7) -- (-.2,0) -- (-.4,.3) .. controls ++(90:.2cm) and ++(270:.2cm) .. (-.6,.7) -- (-.8,1) -- (-.8,3) -- (.9,3) -- (.9,-.7);
\end{scope}
\draw[dashed] (-.8,1.3) -- (-.8,3);
\draw[\QsColor,thick] (-.4,1.3) .. controls ++(90:.45cm) and ++(270:.45cm) .. (0,2.3) arc (0:90:.4cm);
\draw[\XColor,thick,knotrColor] (0,1.3) .. controls ++(90:.45cm) and ++(270:.45cm) .. (-.4,2.3) -- (-.4,3);
\filldraw[\XColor] (-.4,2.7) circle (.05cm);
}
\approx \sum_{i,j,k}
\tikzmath{
\begin{scope}
\clip[rounded corners = 5] (-.75,-4.2) rectangle (1.1,3);
\filldraw[\rColor] (-.55,-4.2) -- (-.55,-3.5) -- (-.35,-3.2) .. controls ++(90:.2cm) and ++(270:.2cm) .. (-.25,-2.8) -- (0,-2.2) .. controls ++(90:.2cm) and ++(270:.2cm) .. (.2,-1.8) -- (.4,-1.5) -- (.4,-.5) -- (.2,-.2) .. controls ++(90:.2cm) and ++(270:.2cm) .. (0,.2) -- (-.2,.8) .. controls ++(90:.2cm) and ++(270:.2cm) .. (-.35,1.2) -- (-.55,1.5) -- (-.55,3) -- (1.1,3) -- (1.1,-4.2);
\end{scope}
\draw[dashed] (-.55,1.5) -- (-.55,3);
\draw[dashed] (-.2,.8) .. controls ++(90:.2cm) and ++(270:.2cm) .. (-.35,1.2); 
\draw[dashed] (.2,-.2) .. controls ++(90:.2cm) and ++(270:.2cm) .. (0,.2);
\draw[dashed] (.4,-1.5) -- (.4,-.5);
\draw[dashed] (0,-2.2) .. controls ++(90:.2cm) and ++(270:.2cm) .. (.2,-1.8);
\draw[dashed] (-.35,-3.2) .. controls ++(90:.2cm) and ++(270:.2cm) .. (-.25,-2.8);
\draw[dashed] (-.55,-4.2) -- (-.55,-3.5);

\draw[\QsColor,thick] (.65,-.5) -- (.65,1.8) arc (0:90:.8cm);
\draw[\QsColor,thick] (.25,.5) -- (.25,1.8) arc (0:90:.4cm);
\draw[\XColor,thick] (-.15,1.5) -- (-.15,3);
\draw[\QsColor,thick] (.65,-1.8) -- (.65,-2.8) arc (0:-90:.4cm);
\draw[\XColor,thick] (.25,-4.2) -- (.25,-2.5);
\draw[\QsColor,thick] (-.15,-4.2) -- (-.15,-3.5);

\filldraw[\XColor] (-.15,2.6) circle (.05cm);
\filldraw[\XColor] (-.15,2.2) circle (.05cm);
\filldraw[\XColor] (.25,-3.2) circle (.05cm);
\roundNbox{unshaded}{(-.35,1.5)}{.3}{.1}{.1}{\scriptsize{$x_i^{(n)}$}};
\roundNbox{unshaded}{(0,.5)}{.3}{.1}{.1}{\scriptsize{$q_j^{(n)}$}};
\roundNbox{unshaded}{(.4,-.5)}{.3}{.1}{.1}{\scriptsize{$q_k$}};
\roundNbox{unshaded}{(.4,-1.5)}{.3}{.1}{.1}{\scriptsize{$(q_j^{(\!n\!)})^\dag$}};
\roundNbox{unshaded}{(0,-2.5)}{.3}{.1}{.1}{\scriptsize{$(x_i^{(\!n\!)})^\dag$}};
\roundNbox{unshaded}{(-.35,-3.5)}{.3}{.1}{.1}{\scriptsize{$q_k^\dag$}};
}
\approx \sum_{i,j,k}
\tikzmath{
\begin{scope}
\clip[rounded corners = 5] (-.75,-4.2) rectangle (1.1,3);
\filldraw[\rColor] (-.55,-4.2) -- (-.55,-3.5) -- (-.35,-3.2) .. controls ++(90:.2cm) and ++(270:.2cm) .. (-.25,-2.8) -- (0,-2.2) .. controls ++(90:.2cm) and ++(270:.2cm) .. (.2,-1.8) -- (.4,-1.5) -- (.4,-.5) -- (.2,-.2) .. controls ++(90:.2cm) and ++(270:.2cm) .. (0,.2) -- (-.2,.8) .. controls ++(90:.2cm) and ++(270:.2cm) .. (-.35,1.2) -- (-.55,1.5) -- (-.55,3) -- (1.1,3) -- (1.1,-4.2);
\end{scope}
\draw[dashed] (-.55,1.5) -- (-.55,3);
\draw[dashed] (-.2,.8) .. controls ++(90:.2cm) and ++(270:.2cm) .. (-.35,1.2); 
\draw[dashed] (.2,-.2) .. controls ++(90:.2cm) and ++(270:.2cm) .. (0,.2);
\draw[dashed] (.4,-1.5) -- (.4,-.5);
\draw[dashed] (0,-2.2) .. controls ++(90:.2cm) and ++(270:.2cm) .. (.2,-1.8);
\draw[dashed] (-.35,-3.2) .. controls ++(90:.2cm) and ++(270:.2cm) .. (-.25,-2.8);
\draw[dashed] (-.55,-4.2) -- (-.55,-3.5);

\draw[\QsColor,thick] (.65,-.5) -- (.65,1.8) arc (0:90:.8cm);
\draw[\QsColor,thick] (.25,.5) -- (.25,1.8) arc (0:90:.4cm);
\draw[\XColor,thick] (-.15,1.5) -- (-.15,3);
\draw[\QsColor,thick] (.65,-1.8) -- (.65,-2.8) arc (0:-90:.4cm);
\draw[\XColor,thick] (.25,-4.2) -- (.25,-2.5);
\draw[\QsColor,thick] (-.15,-4.2) -- (-.15,-3.5);

\filldraw[\XColor] (-.15,2.6) circle (.05cm);
\filldraw[\XColor] (-.15,2.2) circle (.05cm);
\filldraw[\XColor] (.25,-3.2) circle (.05cm);
\roundNbox{unshaded}{(-.35,1.5)}{.3}{.1}{.1}{\scriptsize{$x_i^{(n)}$}};
\roundNbox{unshaded}{(0,.5)}{.3}{.1}{.1}{\scriptsize{$q_k$}};
\roundNbox{unshaded}{(.4,-.5)}{.3}{.1}{.1}{\scriptsize{$q_j^{(n)}$}};
\roundNbox{unshaded}{(.4,-1.5)}{.3}{.1}{.1}{\scriptsize{$(q_j^{(\!n\!)})^\dag$}};
\roundNbox{unshaded}{(0,-2.5)}{.3}{.1}{.1}{\scriptsize{$(x_i^{(\!n\!)})^\dag$}};
\roundNbox{unshaded}{(-.35,-3.5)}{.3}{.1}{.1}{\scriptsize{$q_k^\dag$}};
}
\approx \sum_{i,j,k}
\tikzmath{
\begin{scope}
\clip[rounded corners = 5] (-.75,-4.2) rectangle (1.1,3);
\filldraw[\rColor] (-.55,-4.2) -- (-.55,-3.5) -- (-.35,-3.2) .. controls ++(90:.2cm) and ++(270:.2cm) .. (-.25,-2.8) -- (0,-2.2) .. controls ++(90:.2cm) and ++(270:.2cm) .. (.2,-1.8) -- (.4,-1.5) -- (.4,-.5) -- (.2,-.2) .. controls ++(90:.2cm) and ++(270:.2cm) .. (0,.2) -- (-.2,.8) .. controls ++(90:.2cm) and ++(270:.2cm) .. (-.35,1.2) -- (-.55,1.5) -- (-.55,3) -- (1.1,3) -- (1.1,-4.2);
\end{scope}
\draw[dashed] (-.55,1.5) -- (-.55,3);
\draw[dashed] (-.2,.8) .. controls ++(90:.2cm) and ++(270:.2cm) .. (-.35,1.2); 
\draw[dashed] (.2,-.2) .. controls ++(90:.2cm) and ++(270:.2cm) .. (0,.2);
\draw[dashed] (.4,-1.5) -- (.4,-.5);
\draw[dashed] (0,-2.2) .. controls ++(90:.2cm) and ++(270:.2cm) .. (.2,-1.8);
\draw[dashed] (-.35,-3.2) .. controls ++(90:.2cm) and ++(270:.2cm) .. (-.25,-2.8);
\draw[dashed] (-.55,-4.2) -- (-.55,-3.5);

\draw[\QsColor,thick] (.65,-.5) -- (.65,1.8) arc (0:90:.4cm);
\draw[\XColor,thick] (.25,.5) -- (.25,3);
\draw[\QsColor,thick] (-.15,1.5) -- (-.15,2.2)  arc (180:90:.4cm);
\draw[\QsColor,thick] (.65,-1.8) -- (.65,-2.8) arc (0:-90:.4cm);
\draw[\XColor,thick] (.25,-4.2) -- (.25,-2.5);
\draw[\QsColor,thick] (-.15,-4.2) -- (-.15,-3.5);

\filldraw[\XColor] (.25,2.6) circle (.05cm);
\filldraw[\XColor] (.25,2.2) circle (.05cm);
\filldraw[\XColor] (.25,-3.2) circle (.05cm);
\roundNbox{unshaded}{(-.35,1.5)}{.3}{.1}{.1}{\scriptsize{$q_k$}};
\roundNbox{unshaded}{(0,.5)}{.3}{.1}{.1}{\scriptsize{$x_i^{(n)}$}};
\roundNbox{unshaded}{(.4,-.5)}{.3}{.1}{.1}{\scriptsize{$q_j^{(n)}$}};
\roundNbox{unshaded}{(.4,-1.5)}{.3}{.1}{.1}{\scriptsize{$(q_j^{(\!n\!)})^\dag$}};
\roundNbox{unshaded}{(0,-2.5)}{.3}{.1}{.1}{\scriptsize{$(x_i^{(\!n\!)})^\dag$}};
\roundNbox{unshaded}{(-.35,-3.5)}{.3}{.1}{.1}{\scriptsize{$q_k^\dag$}};
}
\approx \sum_k
\tikzmath{
\begin{scope}
\clip[rounded corners = 5] (-.6,-.7) rectangle (.8,2.2);
\filldraw[\rColor] (-.4,-.7) -- (-.4,0) -- (-.2,0) -- (-.2,1) -- (-.4,1) -- (-.4,2.2) -- (.8,2.2) -- (.8,-.7);
\end{scope}
\draw[dashed] (-.4,1) -- (-.4,2.2);
\draw[dashed] (-.2,0) -- (-.2,1);
\draw[dashed] (-.4,-.7) -- (-.4,0);

\draw[\XColor,thick] (.4,-.7) -- (.4,2.2);
\draw[\QsColor,thick] (0,1.3) arc (180:90:.4cm);
\draw[\QsColor,thick] (0,-.7) -- (0,0);
\filldraw[\XColor] (.4,1.7) circle (.05cm);
\roundNbox{unshaded}{(-.2,1)}{.3}{.1}{.1}{\scriptsize{$q_k$}};
\roundNbox{unshaded}{(-.2,0)}{.3}{.1}{.1}{\scriptsize{$q_k^\dag$}};
}
=
\tikzmath{
\begin{scope}
\clip[rounded corners = 5] (-.6,1) rectangle (.8,2.2);
\filldraw[\rColor] (-.4,1) -- (-.4,2.2) -- (.8,2.2) -- (.8,1);
\end{scope}
\draw[dashed] (-.4,1) -- (-.4,2.2);

\draw[\XColor,thick] (.4,1) -- (.4,2.2);
\draw[\QsColor,thick] (0,1) -- (0,1.3) arc (180:90:.4cm);
\filldraw[\XColor] (.4,1.7) circle (.05cm);
}
\]
The first approximation is from the definition of the under braiding $u_{X,A}^{\dag}$;
the second approximation is from Fact \ref{facts:CommQsysInChi}(1) that $(q_j^{(n)})$ is a central sequence in $|A|$;
the third approximation is because $\{x_i^{(n)}\}$ is an approximately inner $|X|_{|A|}$ basis so that $\|q_k x_i^{(n)}-x_i^{(n)}q_k\|_2\to 0$.

Similarly, suppose $\{x_i\}$ is an $|X|_{|A|}$-basis and $\{q_j\}$ is an $|A|_{M\rtimes G}$-basis, then by Facts \ref{facts:CommQsysInChi}(3) $\{x_i\lhd q_j\}$ is an $|X|_M$-basis. 
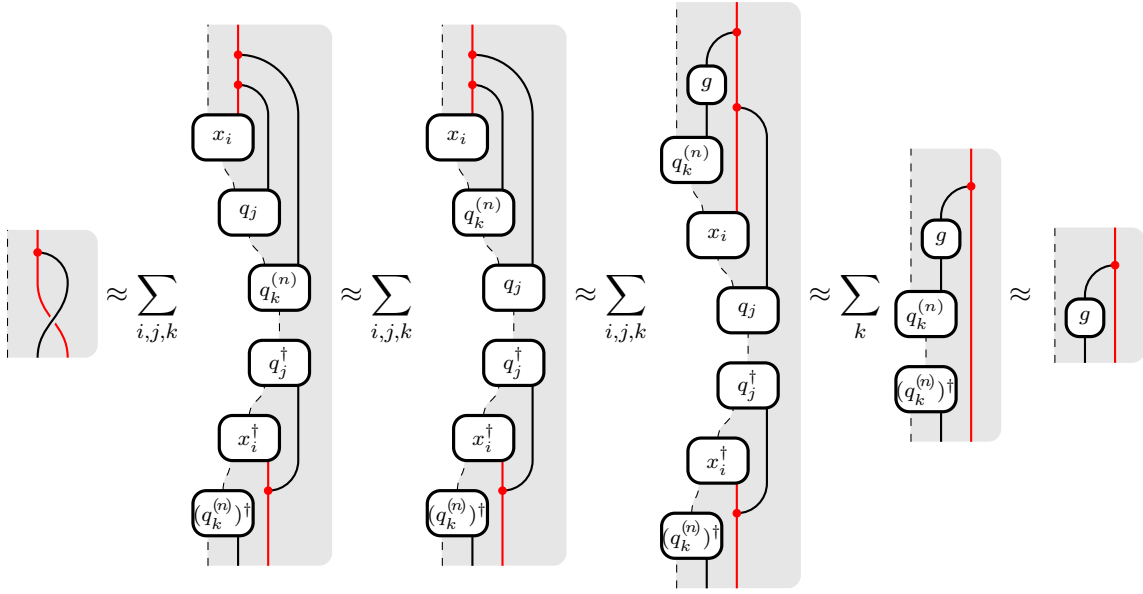
\[
\tikzmath{
\begin{scope}
\clip[rounded corners = 5] (-1,1.3) rectangle (.4,3);
\filldraw[\rColor] (-.2,-.7) -- (-.2,0) -- (-.4,.3) .. controls ++(90:.2cm) and ++(270:.2cm) .. (-.6,.7) -- (-.8,1) -- (-.8,3) -- (.9,3) -- (.9,-.7);
\end{scope}
\draw[dashed] (-.8,1.3) -- (-.8,3);
\draw[\XColor,thick] (0,1.3) .. controls ++(90:.45cm) and ++(270:.45cm) .. (-.4,2.3) -- (-.4,3);
\draw[\QsColor,thick,knotrColor] (-.4,1.3) .. controls ++(90:.45cm) and ++(270:.45cm) .. (0,2.3) arc (0:90:.4cm);
\filldraw[\XColor] (-.4,2.7) circle (.05cm);
}
\approx \sum_{i,j,k}
\tikzmath{
\begin{scope}
\clip[rounded corners = 5] (-.75,-4.2) rectangle (1.1,3);
\filldraw[\rColor] (-.55,-4.2) -- (-.55,-3.5) -- (-.35,-3.2) .. controls ++(90:.2cm) and ++(270:.2cm) .. (-.25,-2.8) -- (0,-2.2) .. controls ++(90:.2cm) and ++(270:.2cm) .. (.2,-1.8) -- (.4,-1.5) -- (.4,-.5) -- (.2,-.2) .. controls ++(90:.2cm) and ++(270:.2cm) .. (0,.2) -- (-.2,.8) .. controls ++(90:.2cm) and ++(270:.2cm) .. (-.35,1.2) -- (-.55,1.5) -- (-.55,3) -- (1.1,3) -- (1.1,-4.2);
\end{scope}
\draw[dashed] (-.55,1.5) -- (-.55,3);
\draw[dashed] (-.2,.8) .. controls ++(90:.2cm) and ++(270:.2cm) .. (-.35,1.2); 
\draw[dashed] (.2,-.2) .. controls ++(90:.2cm) and ++(270:.2cm) .. (0,.2);
\draw[dashed] (.4,-1.5) -- (.4,-.5);
\draw[dashed] (0,-2.2) .. controls ++(90:.2cm) and ++(270:.2cm) .. (.2,-1.8);
\draw[dashed] (-.35,-3.2) .. controls ++(90:.2cm) and ++(270:.2cm) .. (-.25,-2.8);
\draw[dashed] (-.55,-4.2) -- (-.55,-3.5);

\draw[\QsColor,thick] (.65,-.5) -- (.65,1.8) arc (0:90:.8cm);
\draw[\QsColor,thick] (.25,.5) -- (.25,1.8) arc (0:90:.4cm);
\draw[\XColor,thick] (-.15,1.5) -- (-.15,3);
\draw[\QsColor,thick] (.65,-1.8) -- (.65,-2.8) arc (0:-90:.4cm);
\draw[\XColor,thick] (.25,-4.2) -- (.25,-2.5);
\draw[\QsColor,thick] (-.15,-4.2) -- (-.15,-3.5);

\filldraw[\XColor] (-.15,2.6) circle (.05cm);
\filldraw[\XColor] (-.15,2.2) circle (.05cm);
\filldraw[\XColor] (.25,-3.2) circle (.05cm);
\roundNbox{unshaded}{(-.35,1.5)}{.3}{.1}{.1}{\scriptsize{$x_i$}};
\roundNbox{unshaded}{(0,.5)}{.3}{.1}{.1}{\scriptsize{$q_j$}};
\roundNbox{unshaded}{(.4,-.5)}{.3}{.1}{.1}{\scriptsize{$q_k^{(n)}$}};
\roundNbox{unshaded}{(.4,-1.5)}{.3}{.1}{.1}{\scriptsize{$q_j^\dag$}};
\roundNbox{unshaded}{(0,-2.5)}{.3}{.1}{.1}{\scriptsize{$x_i^\dag$}};
\roundNbox{unshaded}{(-.35,-3.5)}{.3}{.1}{.1}{\scriptsize{$(q_k^{(\!n\!)})^\dag$}};
}
\approx \sum_{i,j,k}
\tikzmath{
\begin{scope}
\clip[rounded corners = 5] (-.75,-4.2) rectangle (1.1,3);
\filldraw[\rColor] (-.55,-4.2) -- (-.55,-3.5) -- (-.35,-3.2) .. controls ++(90:.2cm) and ++(270:.2cm) .. (-.25,-2.8) -- (0,-2.2) .. controls ++(90:.2cm) and ++(270:.2cm) .. (.2,-1.8) -- (.4,-1.5) -- (.4,-.5) -- (.2,-.2) .. controls ++(90:.2cm) and ++(270:.2cm) .. (0,.2) -- (-.2,.8) .. controls ++(90:.2cm) and ++(270:.2cm) .. (-.35,1.2) -- (-.55,1.5) -- (-.55,3) -- (1.1,3) -- (1.1,-4.2);
\end{scope}
\draw[dashed] (-.55,1.5) -- (-.55,3);
\draw[dashed] (-.2,.8) .. controls ++(90:.2cm) and ++(270:.2cm) .. (-.35,1.2); 
\draw[dashed] (.2,-.2) .. controls ++(90:.2cm) and ++(270:.2cm) .. (0,.2);
\draw[dashed] (.4,-1.5) -- (.4,-.5);
\draw[dashed] (0,-2.2) .. controls ++(90:.2cm) and ++(270:.2cm) .. (.2,-1.8);
\draw[dashed] (-.35,-3.2) .. controls ++(90:.2cm) and ++(270:.2cm) .. (-.25,-2.8);
\draw[dashed] (-.55,-4.2) -- (-.55,-3.5);

\draw[\QsColor,thick] (.65,-.5) -- (.65,1.8) arc (0:90:.8cm);
\draw[\QsColor,thick] (.25,.5) -- (.25,1.8) arc (0:90:.4cm);
\draw[\XColor,thick] (-.15,1.5) -- (-.15,3);
\draw[\QsColor,thick] (.65,-1.8) -- (.65,-2.8) arc (0:-90:.4cm);
\draw[\XColor,thick] (.25,-4.2) -- (.25,-2.5);
\draw[\QsColor,thick] (-.15,-4.2) -- (-.15,-3.5);

\filldraw[\XColor] (-.15,2.6) circle (.05cm);
\filldraw[\XColor] (-.15,2.2) circle (.05cm);
\filldraw[\XColor] (.25,-3.2) circle (.05cm);
\roundNbox{unshaded}{(-.35,1.5)}{.3}{.1}{.1}{\scriptsize{$x_i$}};
\roundNbox{unshaded}{(0,.5)}{.3}{.1}{.1}{\scriptsize{$q_k^{(n)}$}};
\roundNbox{unshaded}{(.4,-.5)}{.3}{.1}{.1}{\scriptsize{$q_j$}};
\roundNbox{unshaded}{(.4,-1.5)}{.3}{.1}{.1}{\scriptsize{$q_j^\dag$}};
\roundNbox{unshaded}{(0,-2.5)}{.3}{.1}{.1}{\scriptsize{$x_i^\dag$}};
\roundNbox{unshaded}{(-.35,-3.5)}{.3}{.1}{.1}{\scriptsize{$(q_k^{(\!n\!)})^\dag$}};
}
\approx \sum_{i,j,k}
\tikzmath{
\begin{scope}
\clip[rounded corners = 5] (-.75,-4.2) rectangle (1.1,3.6);
\filldraw[\rColor] (-.55,-4.2) -- (-.55,-3.5) -- (-.35,-3.2) .. controls ++(90:.2cm) and ++(270:.2cm) .. (-.25,-2.8) -- (0,-2.2) .. controls ++(90:.2cm) and ++(270:.2cm) .. (.2,-1.8) -- (.4,-1.5) -- (.4,-.5) -- (.2,-.2) .. controls ++(90:.2cm) and ++(270:.2cm) .. (0,.2) -- (-.2,.8) .. controls ++(90:.2cm) and ++(270:.2cm) .. (-.35,1.2) -- (-.55,1.5) -- (-.55,3.6) -- (1.1,3.6) -- (1.1,-4.2);
\end{scope}
\draw[dashed] (-.55,1.5) -- (-.55,3.6);
\draw[dashed] (-.2,.8) .. controls ++(90:.2cm) and ++(270:.2cm) .. (-.35,1.2); 
\draw[dashed] (.2,-.2) .. controls ++(90:.2cm) and ++(270:.2cm) .. (0,.2);
\draw[dashed] (.4,-1.5) -- (.4,-.5);
\draw[dashed] (0,-2.2) .. controls ++(90:.2cm) and ++(270:.2cm) .. (.2,-1.8);
\draw[dashed] (-.35,-3.2) .. controls ++(90:.2cm) and ++(270:.2cm) .. (-.25,-2.8);
\draw[dashed] (-.55,-4.2) -- (-.55,-3.5);

\draw[\QsColor,thick] (.65,-.5) -- (.65,1.8) arc (0:90:.4cm);
\draw[\XColor,thick] (.25,.5) -- (.25,3.6);
\draw[\QsColor,thick] (-.15,1.5) -- (-.15,2.8)  arc (180:90:.4cm);
\draw[\QsColor,thick] (.65,-1.8) -- (.65,-2.8) arc (0:-90:.4cm);
\draw[\XColor,thick] (.25,-4.2) -- (.25,-2.5);
\draw[\QsColor,thick] (-.15,-4.2) -- (-.15,-3.5);

\filldraw[\XColor] (.25,3.2) circle (.05cm);
\filldraw[\XColor] (.25,2.2) circle (.05cm);
\filldraw[\XColor] (.25,-3.2) circle (.05cm);
\roundNbox{unshaded}{(-.15,2.5)}{.25}{0}{0}{\scriptsize{$g$}};
\roundNbox{unshaded}{(-.35,1.5)}{.3}{.1}{.1}{\scriptsize{$q_k^{(n)}$}};
\roundNbox{unshaded}{(0,.5)}{.3}{.1}{.1}{\scriptsize{$x_i$}};
\roundNbox{unshaded}{(.4,-.5)}{.3}{.1}{.1}{\scriptsize{$q_j$}};
\roundNbox{unshaded}{(.4,-1.5)}{.3}{.1}{.1}{\scriptsize{$q_j^\dag$}};
\roundNbox{unshaded}{(0,-2.5)}{.3}{.1}{.1}{\scriptsize{$x_i^\dag$}};
\roundNbox{unshaded}{(-.35,-3.5)}{.3}{.1}{.1}{\scriptsize{$(q_k^{(\!n\!)})^\dag$}};
}
\approx \sum_k
\tikzmath{
\begin{scope}
\clip[rounded corners = 5] (-.6,-.7) rectangle (.8,3.2);
\filldraw[\rColor] (-.4,-.7) -- (-.4,0) -- (-.2,0) -- (-.2,1) -- (-.4,1) -- (-.4,3.2) -- (.8,3.2) -- (.8,-.7);
\end{scope}
\draw[dashed] (-.4,1) -- (-.4,3.2);
\draw[dashed] (-.2,0) -- (-.2,1);
\draw[dashed] (-.4,-.7) -- (-.4,0);

\draw[\XColor,thick] (.4,-.7) -- (.4,3.2);
\draw[\QsColor,thick] (0,1.3) -- (0,2.3) arc (180:90:.4cm);
\draw[\QsColor,thick] (0,-.7) -- (0,0);
\filldraw[\XColor] (.4,2.7) circle (.05cm);
\roundNbox{unshaded}{(0,2)}{.25}{0}{0}{\scriptsize{$g$}};
\roundNbox{unshaded}{(-.2,1)}{.3}{.1}{.1}{\scriptsize{$q_k^{(n)}$}};
\roundNbox{unshaded}{(-.2,0)}{.3}{.1}{.1}{\scriptsize{$(q_k^{(\!n\!)})^\dag$}};
}
\approx
\tikzmath{
\begin{scope}
\clip[rounded corners = 5] (-.6,.4) rectangle (.8,2.2);
\filldraw[\rColor] (-.4,0) -- (-.4,2.2) -- (.8,2.2) -- (.8,0);
\end{scope}
\draw[dashed] (-.4,.4) -- (-.4,2.2);

\draw[\XColor,thick] (.4,.4) -- (.4,2.2);
\draw[\QsColor,thick] (0,.4) -- (0,1.3) arc (180:90:.4cm);
\filldraw[\XColor] (.4,1.7) circle (.05cm);
\roundNbox{unshaded}{(0,1)}{.25}{0}{0}{\scriptsize{$g$}};
}
\]
The first approximation is from the definition of the over braiding $u_{A,X}$;
the second approximation is from Facts \ref{facts:CommQsysInChi}(1) that $(q_k^{(n)})$ is a central sequence in $|A|$;
the third approximation is because $X$ is $g$-centrally trivial over $|A|$ and $(q_k^{(n)})$ is a central sequence in $|A|$ so that $\|q_k^{(n)}\rhd_g x_i-x_i\lhd q_k^{(n)}\|_2\to 0$.

Because $g\in \Aut(A)$ is an algebra automorphism, we have $g\circ i=i$, where $i:1\to A$ is the unit. 
For central sequence $(a_n)\in (M\rtimes G)_\omega$, since $A$ is centrally trivial, for all $q\in A$, $\|a_n q-q a_n\|_2\to 0$, so $(a_n)$ is also a central sequence on $|A|$. 
Then for $x\in |X|$, 
$$\|a_n\rhd x-x\lhd a_n\|_2 = \|a_n\rhd_g x-x\lhd a_n\|_2 \to 0,$$
so $X\in \ctBim(M\rtimes G)$.

Moreover, combining above equations, if $|X|\in g \ctBim(|A|)\cap \aiBim(|A|)$, $X$ is a right $A$-module in $\Chi(M\rtimes G)$ and and the realization preserve the grading $\partial_X=g$.

Similar to the proof of \cite[Thm.~5.15]{MR4753059}, the realization gives a $G/K$-crossed braided equivalence $\Chi(M\rtimes G)_A\to \bigoplus_{g\in G/K}g \ctBim(|A|)\cap \aiBim(|A|)$.
\end{proof}

\begin{rem}
In \cite[Thm.~5.15]{MR4753059}, they prove that for commutative Q-system $Q\in\Chi(M)$, the local extension $\Chi(M)_Q^{\loc}$ is braided equivalent to $\Chi(|Q|)$. 
In our case, $A$ is a commutative Q-system in $\Chi(M\rtimes G)$, then the local extension $\Chi(M\rtimes G)_A^{\loc}$ is braided equivalent to $\Chi(|A|)=\ctBim(|A|)\cap \aiBim(|Q|)$. 
Note that in the group case, the local extension is the $e$-grading part of the deequivariantization. 
By comparing the formulas, we can see $\bigoplus_{g\in G/K}g \ctBim(|A|)\cap \aiBim(|A|)$, which is equivalent to the deequivariantization of $\Chi(M\rtimes G)$, is a complete $G/K$-crossed braided extension of $\Chi(|A|)$.
\end{rem}

In the following, we use the Morita equivalence of the basic construction $M\rtimes K\subseteq M\rtimes G\subseteq |A|$ to give a more computable formula.

For $g\in \Aut_\alge(A)$, we define an $|A|$-$|A|$ bimodule $g L^2(|A|)$ by twisting the left action by $\lambda^g:=\lambda\circ (g\otimes\id_{L^2(|A|)})$. 
Since $\Aut_\alge(A)\cong G/K$, this gives an action 
$$\Psi:G/K\to \Bim(|A|)\qquad \text{by}\qquad g\mapsto gL^2(|A|).$$

Note that $M\rtimes K$ and $|A|$ are Morita equivalent by the bimodule $Z={}_{M\rtimes G}L^2(M\rtimes G)_{|A|}$ such that 
$$Z\boxtimes_{|A|}\overline{Z}\cong \id_{L^2(M\rtimes K)}\in\Bim(M\rtimes K), \qquad \qquad \overline{Z}\boxtimes_{M\rtimes K} Z\cong \id_{L^2(|A|)}\in \Bim(|A|).$$
Then 
$$\Ad(Z):\Bim(|A|)\to \Bim(M\rtimes K)\qquad\text{by}\qquad X\mapsto Z\boxtimes_{|A|}X\boxtimes_{|A|}\overline{Z}$$
is a tensor equivalence.

Then using the Morita equivalence, we can define an induced action, without loss of generality, still denoted as $\psi:G/K\to \Bim(M\rtimes K)$ by
$$\psi:=\Ad(Z)\circ \Psi.$$
Then by the construction, for $g\in G/K$, we have
$$Z\boxtimes_{|A|}g L^2(P)\boxtimes_{|A|}\overline{Z}\cong {}_{\psi_g}L^2(M\rtimes K)\in \Bim(M\rtimes K).$$

\begin{thm}\label{thm:G/KdeequivofChi(MG)}
Suppose $G$ is a finite group, $M$ is a McDuff $\rm II_1$ factor. 
Suppose $\psi:G\to \Aut(M)$ is an outer action and $\psi$ is not approximately inner. 
Then the $G/K$-deequivariantization of $\Chi(M\rtimes G)$
$$\Chi(M\rtimes G)_{G/K}\cong \bigoplus_{g\in G/K} ({}_{\psi_g}L^2(M\rtimes K)\boxtimes_{M\rtimes K} \ctBim(M\rtimes K))\cap \aiBim(M\rtimes K)$$
as unitary $G/K$-crossed braided tensor categories, where $K=\psi^{-1}(\Ct(M))$.
\end{thm}
\begin{proof}
According to \cite[Thm.~A]{MR4753059}, $\Chi(M)\cong \Chi(N)$ if $M$ and $N$ are Morita equivalent. 
In our case, $\Ad(Z):\ctBim(|A|)\to \ctBim(M\rtimes K)$ and $\Ad(Z):\aiBim(|A|)\to \aiBim(M\rtimes K)$ are tensor equivalences, and $\Ad(Z)$ is the braided equivalence on $\Chi$.
Therefore, recall the construction of $g\ctBim(|A|)$, a centrally trivial bimodule twists the left action by $g\in \Aut_\alge(A)$, the Morita equivalence gives an equivalence 
$$\Ad(Z):g\ctBim(|A|)\to {}_{\psi_g}L^2(M\rtimes K)\boxtimes_{M\rtimes K}\ctBim(M\rtimes K).$$
Recall the $G/K$-crossed braiding construction (\ref{eq:GcrossedBraid}),
the Morita equivalence gives a $G/K$-crossed braided equivalence
\[\Ad(Z):\bigoplus_{g\in G/K}g \ctBim(|A|)\cap \aiBim(|A|)\to \bigoplus_{g\in G/K} ({}_{\psi_g}L^2(M\rtimes K)\boxtimes_{M\rtimes K} \ctBim(M\rtimes K))\cap \aiBim(M\rtimes K)\]
\end{proof}

\subsection{Gauging and short exact sequence}
In this section, we will show that the equivariantization/deequivariantization procedure induces a short exact sequence with the first entry being $\widehat{G}$.

Suppose $G$ is a finite group. 
Let $\cC$ be a rigid $\rm C^*$ tensor category with simple unit $1_\cC$ and $\rho$ is a (strict) $G$-action on $\cC$ by tensor autoequivalences $\rho:G\to \uAo(\cC)$.
A \textit{$G$-equivariant structure} on object $x\in \cC$ is a family of unitary isomorphisms $\{u_g\in \cC(\rho_g(x)\to x)\}_{g\in G}$ such that 
$$u_{gh}=u_g\circ \rho_g(u_h).$$

\begin{lem}\label{lem:GequivonUnit}
The group of $G$-equivariant structures on the tensor unit $1_\cC$ is isomorphic to $\widehat{G}=\Hom(G\to U(1))$.
\end{lem}
The proof of this lemma follows from the facts that $\rho_g(1_\cC)\cong 1_{\cC}$ and $\End_\cC(1_\cC)\cong \bbC$ directly.

\begin{prop} Suppose $G$ is a finite group, $\cB$ is a braided $\rm C^*$ tensor category and $\Rep(G)\hookrightarrow \cB$ is a unitary braided fully faithful tensor subcategory. Then
$$0\to \widehat{G}\to \Inv(\cB)\to \Inv(\cB_G)_{\rm{lift}}\to 0$$ is a short exact sequence, where 
$\Inv(-)$ is group of isomorphism classes of invertible objects and 
$$\Inv(\cB_G)_{\rm{lift}}=\{[y]\in \Inv(\cB_G)\mid y\ \mathrm{admits\ a\ unitary\ } G\mathrm{-equivariant\ structure}\}$$
\end{prop}
\begin{proof}

There is a canonical tensor functor to the deequivariantization $F:\cB\to \cB_G$ by $X\mapsto X\otimes A$ such that 
$$F(x)\otimes_A F(y) = (x\otimes A)\otimes_A (y\otimes A)\cong (x\otimes y)\otimes A = F(x\otimes y).$$
If $x\in \Inv(\cB)$ with inverse $x^*$, i.e., $x\otimes x^*\cong 1_{\cB}\cong x^*\otimes x$, then $F(x)=x\otimes A$ is invertible in $\cB_G$ with inverse $x^*\otimes A$.
Therefore, $F$ induces a group homomorphism 
$$\Pi:\Inv(\cB)\to \Inv(\cB_G),\qquad\quad [x]\mapsto [x\otimes A].$$

By the equivariantization/deequivariantization correspondence, $\cB\cong (\cB_G)^G$, the $G$-equivariant objects in $\cB_G$. 
Under this equivalence, $F$ is the forgetful functor that drops the $G$-equivariant structure (only remembers the underlying $A$-modules).
Therefore, the image of $\Pi$ consists exactly of those invertible objects of $\cB_G$ that admits a $G$-equivariant structure.
Define 
$$\Inv(\cB_G)_{\rm{lift}}=\{[y]\in \Inv(\cB_G)\mid y\ \mathrm{admits\ a\ unitary\ } G\mathrm{-equivariant\ structure}\}$$
Then the image of $\Pi$ is in $\Inv(\cB_G)_{\rm{lift}}$.

Every character $\gamma\in \widehat{G}$ gives a 1-dimensional $G$-representation $\bbC_\gamma$, which is an invertible object in $\Rep(G)$.
Then we define 
$$\delta:\widehat{G}\to \Inv(\cB),\qquad\quad \gamma\mapsto [\bbC_\gamma].$$
It is clearly a group homomorphism since $\Rep(G)\subseteq \cB$ is monoidal.

Now we show exactness.
\begin{itemize}[leftmargin=*]
\item $\delta$ is injective. 

Since $\Rep(G)\subseteq \cB$ is a full tensor subcategory, $\delta$ is injective.
\item $\im(\delta)\subseteq \ker(\Pi)$.

For $\gamma\in \widehat{G}$, we know that $\bbC_\gamma\otimes A\cong A$ as $G$-representations. Moreover, $\bbC_\gamma\otimes A\cong A$ as right $A$-modules. Therefore, $\Pi(\delta(\gamma))=[A]$.
\item $\ker(\Pi)\subseteq \im(\delta)$. 

Suppose $[x]\in \Inv(\cB)$ with $\Pi([x])=[A]$, which means $x\otimes A\cong A$ as right $A$-modules in $\cB_G$. 
Now use the equivalence $\cB\cong (\cB_G)^G$. 
Under this equivalence, $x$ corresponds to a $G$-equivariant object whose underlying $A$-module is $x\otimes A\cong A$.
According to Lemma \ref{lem:GequivonUnit}, when $\cC=\cB_G$, its tensor unit is $A$ as right $A$-module, then the group of $G$-equivariant structure on $A$ is isomorphic to $\widehat{G}$. This gives $\ker(\Pi)\subseteq \im(\delta)$.
\item $\Pi$ is surjective.

Let $[y]\in \Inv(\cB_G)_{\rm{lift}}$. 
We choose a unitary $G$-equivariant structure $u$ on $y$, then $(y,u)\in (\cB_G)^G$.
Since $y$ is invertible in $\cB$, $(y,u)$ is invertible in $(\cB_G)^G$.
Under the equivalence $\cB\cong (\cB_G)^G$, there is an invertible object $x\in \cC$ corresponds to $(y,u)$. 
By forgetting the $G$-equivariant structure, we have $\Pi([x])=[y]$. 
In other words, every liftable $y$ lands in the image of $\Pi$, so $\Pi$ is surjective. \qedhere
\end{itemize}
\end{proof}

\subsection{Invertible liftable objects}
To complete the story, we shall show the invertible liftable objects in the $G/K$-deequivariantization of $\Chi(M\rtimes G)$ is actually equivalent to $(\GCt(M)\cap \GInn(M))/\Inn(M)$, the third term from the Connes' short exact sequence. 

\begin{lem}\label{lem:uApproxbyv}
Suppose $M$ is a $\rm II_1$ factor, $A\subseteq M$ is a finite von Neumann subalgebra with a trace-preserving conditional expectation $E_A:M\to A$. 
Let $u\in U(M)$, $\varepsilon >0$ is small, and $a\in (A)_1$.
If $\|u-a\|_2<\varepsilon$, then there exists a unitary $v\in A$ such that $\|u-v\|_2<\varepsilon+\sqrt{2\varepsilon}$.
In particular, if $(a_n)\subseteq (A)_1$ such that $\|u-a_n\|_2\to 0$, then one can choose $v_n\in U(A)$ such that $\|u-v_n\|_2\to 0$.
\end{lem}
\begin{proof}
Write the polar decomposition of $a$ in $A$ as $a=q|a|$, where $q\in A$ is a partial isometry.
Since $A$ is a finite von Neumann algebra, the partial isometry $q$ can extend to a unitary $v$ and $a=v|a|$. 

We now estimate $\|a-v\|_2$.
First we have
$$\|a-v\|_2^2 = \|v|a|-v\|_2^2 = \| |a|-1\|_2 = \tr((1-|a|)^2)$$
Since $0\le |a|\le 1$ and $(1-t)^2\le 1-t^2$ for $t\in [0,1]$, by functional calculus, we have
$$\|a-v\|_2^2=\tr((1-|a|)^2)\le \tr(1-|a|^2)=1-\|a\|_2^2$$
From $\|u-a\|_2<\varepsilon$, $\|a\|_2> 1-\varepsilon$, then $\|a\|_2^2>(1-\varepsilon)^2>1-2\varepsilon$. This implies
$$\|a-v\|_2^2\le 1-\|a\|_2^2 <2\varepsilon.$$
Therefore, 
\[\|u-v\|_2\le \|u-a\|_2+\|a-v\|_2 <\varepsilon+\sqrt{2\varepsilon}.\qedhere
\]
\end{proof}

The following result is appeared in Jones' $\chi(M)$ notes with a sketch proof. 
Here we provide a complete proof.
\begin{prop}\cite{chiNotes} \label{prop:AinM,aiM->aiA}
Suppose $M$ is a separable $\rm II_1$ factor, $A\subseteq M$ is a subalgebra with conditional expectation $E_A:M\to A$ such that for any central sequence $(a_n)_n$ in $M$, $\lim\limits_{n\to\omega}\|E_A(a_n)-a_n\|_2=0$.
Then if $\alpha\in \Ai(M)$, there exist a unitary $w\in M$ and a sequence of unitaries $(z_n)_n$ in $A$ such that $\alpha=\Ad(w)\circ \lim\limits_{n\to\omega}\Ad(z_n)$.
\end{prop}
\begin{proof}
Recall that the pointwise topology on $\Aut(M)$ is given by $\beta_n\to \beta$ if $\|\beta_n(x)-\beta(x)\|_2\to 0$ for all $x\in M$. 
In other words, let $\delta>0$ and a finite set $F\subseteq (M)_1$, the neighborhood $V_{F,\delta}(\beta)$ of $\beta\in \Aut(M)$ is given by 
$$V_{F,\delta}(\beta) = \{\gamma\in \Aut(M)\mid \max_{x\in F}\|\gamma(x)-\beta(x)\|_2<\delta\}.$$

First we prove the following claim.

\medskip\noindent\underline{Lemma.} For every $\varepsilon>0$, there exist a finite set $F\subseteq (M)_1$ and $\delta>0$ such that whenever $u\in U(M)$ and $\Ad(u)\in V_{F,\delta}(\id)$, there exists $a\in (A)_1$ such that $\|u-a\|_2<\varepsilon$.

\smallskip\noindent\underline{Proof of Lemma}: By contradiction, suppose there is some $\varepsilon>0$, for every finite subset $F\subseteq (M)_1$ and $\delta>0$, there is a unitary $u_{F,\delta}\in M$ such that $\Ad(u_{F,\delta})\in V_{F,\delta}(\id)$ and $\|u_{F,\delta}-a\|_2\ge \varepsilon$ for all $a\in (A)_1$.
Since $M$ is separable, there is a countable dense subset $\{x_m\}\subseteq (M)_1$ in $\|\cdot\|_2$-norm topology.
Let $F_n=\{x_1,x_2,\dots,x_n\}$ and $\delta_n=\frac{1}{n}$, we obtain a sequence $(u_n)$ such that
\begin{itemize}
\item  $\Ad(u_n)\in V_{F_n,\delta_n}(\id)$, i.e., $\|\Ad(u_n)(x_i)-x_i\|_2 = \|[u_n,x_i]\|_2\le \frac{1}{n}$ for all $1\le i\le n$;
\item  $\|u_n-a\|_2\ge \varepsilon$ for all $a\in (A)_1$.
\end{itemize}
Then $(u_n)$ is a central sequence in $M$: Indeed, for $x\in (M)_1$, 
\begin{align*}
\|[u_n,x]\|_2 & = \|[u_n,x_i]+ [u_n,x-x_i] \|_2 \\
& \le \|[u_n,x_i]\|_2 + 2\|u_n\| \|x-x_i\|_2 \\
& = \|[u_n,x_i]\|_2 + 2\|x-x_i\|_2 
\end{align*}
then for any $\varepsilon>0$, we pick $i$ such that $\|x-x_i\|_2< \frac{\varepsilon}{4}$ and pick $N\ge i$ such that $\frac{1}{N}< \frac{\varepsilon}{2}$, which follows that $\|[u_n,x]\|_2<\varepsilon$ for all $n\ge N$.
However, since $E_A(u_n)\in (A)_1$, we have 
$$\varepsilon\le \liminf_{a\in (A)_1}\|u_n -a\|_2 \le \lim\|u_n - E_A(u_n)\|_2 = 0,$$
which is a contradiction. This proves the lemma. 
\vspace{.2cm}

We can additionally refine the above \underline{Lemma} to require $a\in (A)_1$ to be unitary by Lemma \ref{lem:uApproxbyv}.

Now since $\alpha\in\Ai(M)$, there exists a sequence of unitaries in $M$ such that $\Ad(u_n)\to \alpha$ in $\|\cdot\|_2$.
Note that $\Ad(u_n^*u_{n+1})=\Ad(u_n^*)\circ\Ad(u_{n+1})\to \alpha^{-1}\circ \alpha=\id$ pointwise.
For all $n\in\bbN$, there exists a finite subset $F_n\subseteq(M)_1$ and $\delta_n>0$ such that $\Ad(u_n^*u_{n+1})\in V_{F_n,\delta_n}(\id)$. 
Then by the refined claim, there exists a unitary $v_n\in A$ such that $\|u_n^* u_{n+1} - v_n\|_2<2^{-n}$, equivalently speaking, $\|u_nv_n u_{n+1}^*-1\|_2<2^{-n}$. 
Then $\prod\limits_{i=1}^n u_nv_n u_{n+1}^*$ is a Cauchy sequence and we denote its limit as $t$.

Note that $\prod\limits_{i=1}^n v_i = u_1^*\left(\prod\limits_{i=1}^n u_nv_n u_{n+1}^*\right) u_{n+1}$. Then $\lim\limits_{n\to\infty}\Ad\left(\prod\limits_{i=1}^n v_i\right)$ exists and equals $\Ad(u_1^*t)\circ \alpha$. 
Since $v_i\in U(A)$, we define the unitary $z_n=\prod\limits_{i=1}^n v_i\in A$ and $w=t^* u_1$, we have the desired result.
\end{proof}

\begin{prop}\label{prop:InvLiftAI<->GInt}
Suppose outer action $\psi:G\to \Aut(M)$ is not approximately inner, $K=\psi^{-1}(\Ct(M))\cap G$.
The group of invertible $G/K$-liftable bimodules in $\aiBim(M\rtimes K)$ is equivalent to the group of $G$-approximately inner automorphisms $\alpha\in \GInn(M)$ modulo inner automorphisms, i.e.,
$$\Inv(\aiBim(M\rtimes K))_{\mathsf{lift}}\cong \GInn(M)/\Inn(M).$$    
\end{prop}
\begin{proof}
Let $N=M\rtimes K$.
Suppose $\alpha\in \GInn(M)$, there exists unitaries $(z_n)\subseteq M^G$ such that $\alpha=\lim \Ad(z_n)$. 
Define $\widetilde{\alpha}\in\Aut(N)$ by $\widetilde{\alpha}(au_k)=\alpha(a)u_k$ for $a\in M$, $k\in K$.
Since $z_n\in M^G\subseteq M^K$, $z_n$ commutes with every $u_k$, we have $\widetilde{\alpha}=\lim \Ad(z_n)$ and the bimodule ${}_{\widetilde{\alpha}}L^2(N)\in \Inv(\aiBim(N))$. 
Because $\alpha$ is implemented by $G$-fixed unitaries, 
$$\psi_g\alpha \psi_g^{-1}(a) = \psi_g \lim z_n \psi_g^{-1}(a)z_n^* =\psi_g \lim  \psi_g^{-1}(z_naz_n^*) = \lim z_naz_n^* = \alpha(a),$$
for all $a\in M$, $g\in G$, so $\psi_g\alpha \psi_g^{-1}=\alpha$ and consequently $\psi_g\widetilde{\alpha}\psi_g^{-1}=\widetilde{\alpha}$.
Thus the identity maps give a unitary $G/K$-equivariant structure on ${}_{\widetilde{\alpha}}L^2(N)$ showing it is liftable.

If $\alpha'=\Ad(v)\circ \alpha$ with $v\in U(M)$, then
$\widetilde\alpha'=\Ad(v)\circ \widetilde\alpha\in\Aut(N)$, hence ${}_{\widetilde\alpha'}L^2(N)\cong {}_{\widetilde\alpha}L^2(N)$.
So the construction depends only on the class of $\alpha$ in $\GInn(M)/\Inn(M)$.

Now we prove the converse. Suppose $Y\in \Inv(\aiBim(N))_{\mathsf{lift}}$, we shall show $Y\cong {}_{\widetilde{\alpha}}L^2(N)$ for some $\alpha\in\GInn(M)$.

Because $Y$ is invertible and approximately inner, $Y\cong {}_{\beta}L^2(N)$ for some $\beta\in \Ai(N)$. 
Since $\psi:G\to \Aut(M)$ is not approximately inner, so is its restriction on $K=\psi^{-1}(\Ct(M))\cap G$, then by Corollary \ref{cor:CSinMrtimesG}, $(M\rtimes K)_\omega \cong (M_\omega)^K$.
In other words, for every central sequence $(a_n)\in N_\omega$, $\|a_n-E_{M^K}(a_n)\|_2\to 0$.
By Proposition \ref{prop:AinM,aiM->aiA}, there exist unitaries $w\in N$ and $z_n\in M^K$ such that $\beta=\Ad(w)\circ \lim \Ad(z_n)$. 
Since we work with equivalence class of bimodules, we may replace $\beta$ by $\alpha=\lim \Ad(z_n)$, and $Y\cong {}_{\alpha}L^2(N)$ as well.
Moreover, since $\alpha=\lim \Ad(z_n)$, $z_n\in M^K$, $\alpha\in \Ai(M)$.

In this case, since $z_n\in M^K$, we have $[z_n,u_k]=0$ for all $k\in K$, then
$\alpha(u_k)=\lim z_n u_k z_n^* = \lim u_k z_nz_n^*  = u_k$.

The $G/K$-action on $\aiBim(N)$ is given by
\[
\rho_{\bar{g}}(X):={}_{\psi_{\bar{g}}}X_{\psi_{\bar{g}}}
=
({}_{\psi_{\bar{g}}}L^2(N))\boxtimes_{N} X\boxtimes_{N}({}_{\psi_{\bar{g}}^{-1}}L^2(N)).
\]
Since $Y\cong {}_{\alpha}L^2(N)$ is liftable, there exists a unitary $G/K$-equivariant structure
\[
v_{\bar{g}}\in \Bim(N)(\rho_{\bar{g}}(Y)\to Y),\qquad \bar g\in G/K,
\]
such that $v_{\bar g\bar h}=v_{\bar{g}}\circ \rho_{\bar{g}}(v_{\bar h}).$
Because $\rho_{\bar{g}}(Y)\cong {}_{\psi_{\bar{g}}\alpha\psi_{\bar{g}}^{-1}}L^2(N)$,
each $v_{\bar{g}}$ is given by a unitary $t_{\bar{g}}\in U(N)$ such that
\[
\psi_{\bar{g}}\alpha\psi_{\bar{g}}^{-1}=\Ad(t_{\bar{g}})\circ \alpha.
\tag{$1$}
\]
From (1), we have
\begin{align*}
\psi_{\bar{g}\bar{h}}\alpha\psi_{\bar{g}\bar{h}}^{-1}
& = \psi_{\bar{g}} (\psi_{\bar{h}} \alpha \psi_{\bar{h}}^{-1} )\psi_{\bar{h}}^{-1}   \\
& = \psi_{\bar{g}} (\Ad(t_{\bar{h}})\circ\alpha) \psi_{\bar{g}}^{-1}\\
& = \Ad(\psi_{\bar{g}}(t_{\bar{h}}))\circ \psi_{\bar{g}} \alpha \psi_{\bar{g}}^{-1}\\
& = \Ad(\psi_{\bar{g}}(t_{\bar{h}}))\circ \Ad(t_{\bar{g}})\circ\alpha\\
& = \Ad(\psi_{\bar{g}}(t_{\bar{h}})t_{\bar{g}})\circ\alpha
\end{align*}
On the other hand, $$\psi_{\bar{g}\bar{h}}\alpha\psi_{\bar{g}\bar{h}}^{-1} = \Ad(t_{\bar{g}t_{\bar{h}}})\circ \alpha.$$
Comparing the two implementations, after normalizing scalars, we may assume
\[
t_{\bar g\bar h}=\psi_{\bar{g}}(t_{\bar h})t_{\bar{g}},
\qquad \forall\bar g,\bar h\in G/K.
\tag{2}
\]
Applying (1) to $a\in N$. Since
\[
\psi_{\bar{g}}\beta\psi_{\bar{g}}^{-1}(a)
=
\psi_{\bar{g}}\!\left(\lim z_n\psi_{\bar g^{-1}}(a)z_n^*\right)
=
\lim \psi_{\bar{g}}(z_n)a\psi_{\bar{g}}(z_n)^*,
\]
and
\[
\Ad(t_{\bar{g}})\circ\beta(a)
=
\lim t_{\bar{g}}z_naz_n^*t_{\bar{g}}^*,
\]
we get
\[
\|\Ad(\psi_{\bar{g}}(z_n))(a)-\Ad(t_{\bar{g}}z_n)(a)\|_2\to 0
\qquad\forall a\in N.
\]
so for $c_{\bar{g};n}:=z_n^*t_{\bar{g}}^*\psi_{\bar{g}}(z_n)\in N$,
we have $(c_{\bar{g};n})\in N_\omega = (M_\omega)^K$.
Then 
$$\|t_{\bar{g}}^* -z_n c_{\bar{g};n} \psi_{\bar{g}}(z_n^*)\|_2= \|z_n^*t_{\bar{g}}^*\psi_{\bar{g}}(z_n)-c_{\bar{g};n}\|_2\to 0.$$
Note that $\psi_{\bar{g}}$ preserves $M^K$, then $z_n^*,\ \psi_{\bar{g}}(z_n)\in M^K$.
Since $c_{\bar{g}}=(c_{\bar{g};n})\in (M_\omega)^K$, and $M^K$ is closed in $\|\cdot\|_2$-norm,
we have $t_{\bar{g}}\in M^K$.

Moreover, $\{c_{\bar{g}}\}_{\bar g\in G/K}$ is a $1$-cocycle for the induced action $\psi_\omega:G/K\to \Aut(M_\omega)$.
Indeed, writing $z=(z_n)_n\in U(M^\omega)$, using $(2)$, we compute
\begin{align*}
c_{\bar{g}}\psi_{\bar g}(c_{\bar h})
&= z^* t_{\bar{g}}^*\psi_{\bar{g}}(z)\,\psi_{\bar{g}}(z^*t_{\bar h}^*\psi_{\bar h}(z)) \\
&= z^*t_{\bar{g}}^*\psi_{\bar{g}}(t_{\bar h}^*)\psi_{\bar g\bar h}(z) \\
&= z^*(t_{\bar{g}}\psi_{\bar{g}}(t_{\bar h}))^*\psi_{\bar g\bar h}(z) \\
&= z^*t_{\bar g\bar h}^*\psi_{\bar g\bar h}(z) \\
&= c_{\bar g\bar h}. 
\end{align*}

By the same cocycle-untwisting argument as in Theorem \ref{thm:NotCT->AI}
(equivalently, Fact \ref{fact:coboundaryTrEQ}
applied to the induced action on $M_\omega$), there exists $s\in U(M_\omega)$ such that
$$c_{\bar{g}}=s^*\psi_{\bar g}(s),\qquad \forall\bar g\in G/K.$$

Since $M_\omega=(M_\omega)^K$, we may choose a representing sequence
$s=(s_n)$ with $s_n\in U(M^K)$. Set $w_n:=z_n s_n^*\in U(M^K)$.
Because $s\in M_\omega$ is central, we still have $\alpha=\lim\Ad(w_n)$.
Also, from the defining equation for $c_{\bar{g}}$,
$\psi_{\bar{g}}(w)=t_{\bar{g}}w$ in $M^\omega$, i.e.,
\[
\|\psi_{\bar{g}}(w_n)-t_{\bar{g}}w_n\|_2\to 0,
\qquad \forall\bar g\in G/K.
\tag{3}
\]

The action $\psi:G/K\to \Aut(M\rtimes K)$ restricts to a well-defined induced action $\bar{\psi}:G/K\to \Aut(M^K)$ given by $\bar{\psi}_{\bar{g}}(x)=\psi_{\bar{g}}(x)$ for all $x\in M^K$. 
Since $\psi$ is outer, $\bar{\psi}$ is also an outer action.
By (2), since $t_{\bar{g}}\in M^K$, 
$$t_{\bar{g}\bar{h}}=\psi_{\bar{g}}(t_{\bar{h}})t_{\bar{g}}=\bar{\psi}_{\bar{g}}(t_{\bar{h}})t_{\bar{g}},$$
so $t:G/K\to U(M^K)$ is a 1-cocycle of the outer action $\bar{\psi}:G/K\to \Aut(M^K)$. 
(Warning: this is not the usual order for 1-cocycle, but can be done by taking the adjoint)

By Fact \ref{fact:coboundaryTrEQ},
there exists a unitary $v\in U(M^K)$ such that 
$$t_{\bar{g}} = \bar{\psi}_{\bar{g}}(v^*)v = \psi_{\bar{g}}(v^*)v, \quad \forall \bar{g} \in G/K$$
Substitute this coboundary expression into (3), we have
\[\|\psi_{\bar{g}}(w_n)-\psi_{\bar{g}}(v^*)vw_n\|_2 = \|\psi_{\bar{g}}(v w_n)-v w_n\|_2 \to 0.\tag{4}\]

Let $\alpha'=\Ad(v)\circ\alpha=\Ad(v)\circ \lim \Ad(w_n) = \lim\Ad(vw_n)\in \Ai(M)$. 
Since $\alpha'$ and $\alpha$ are equivalent modulo $\Int(M)$, this means $[\alpha']=[\alpha]$.

Since also $vw_n\in M^K$, the sequence is asymptotically $G$-fixed by (4). Therefore
\[
u_n:=E_{M^G}(vw_n)= \frac{1}{|G/K|}\sum_{\bar{g}\in G/K} \psi_{\bar{g}}(vw_n)\in M^G
\]
satisfies
\[
\|u_n-vw_n\|_2 \le \frac{1}{|G/K|}\|\psi_{\bar{g}}(vw_n)-vw_n\|_2\to 0.
\]
After taking polar parts as in Lemma \ref{lem:uApproxbyv}, we may assume $u_n\in U(M^G)$ and still
\[
\|u_n-vw_n\|_2\to 0.
\]
Because $vw_n$ implements $\alpha'$ up to $\omega$-limit, the same is true for $u_n$, so $\alpha'\in \GInn(M)$.
Hence the isomorphism class of $Y$ determines a class
\[
[\alpha]\in \GInn(M)/\Inn(M).
\]

Therefore, we establish the equivalence
\[
\Inv(\aiBim(M\rtimes K))_{\mathsf{lift}}\cong \GInn(M)/\Inn(M).\qedhere
\]
\end{proof}

\begin{thm}
Suppose outer action $\psi:G\to \Aut(M)$ is not approximately inner, $K=\psi^{-1}(\Ct(M))\cap G$.
The group of invertible $G/K$-liftable bimodules in the $G/K$-deequivariantization $\Chi(M\rtimes G)_{G/K}$ is equivalent to the group of the intersection of $G$-centrally trivial automorphisms and $G$-approximately inner automorphisms $\GCt(M)\cap \GInn(M)$ modulo inner automorphisms, i.e.,
$$\Scale[0.91]{\Inv\left(\bigoplus\limits_{g\in G/K} ({}_{\psi_g}L^2(M\rtimes K)\boxtimes_{M\rtimes K}\ctBim(M\rtimes K))\cap \aiBim(M\rtimes K)\right)_{\mathsf{lift}}\cong (\GCt(M)\cap\GInn(M))/\Inn(M).}$$   
\end{thm}
\begin{proof}
Let $N:=M\rtimes K$. 
By Proposition \ref{prop:InvLiftAI<->GInt}, $\Inv(\aiBim(N))_{\mathsf{lift}}\cong \GInn(M)/\Inn(M)$,
every liftable invertible approximately inner $N$-$N$ bimodule is of the form ${}_{\widetilde{\alpha}}L^2(N)$
for some $\alpha\in \GInn(M)$, where $\widetilde{\alpha}(au_k)=\alpha(a)u_k$, $a\in M$, $k\in K$.
Moreover, the isomorphism class of ${}_{\widetilde{\alpha}}L^2(N)$ depends only on the class of
$\alpha$ in $\GInn(M)/\Inn(M)$.

We claim that for $\alpha\in \GInn(M)$ and $g\in G/K$,
\[
{}_{\widetilde{\alpha}}L^2(N)\in
\bigl({}_{\psi_g}L^2(N)\boxtimes_N \ctBim(N)\bigr)\cap \aiBim(N) 
\qquad \Longleftrightarrow\qquad \alpha_\omega=(\psi_g)_\omega \qquad \text{on } M_\omega
\]

Indeed, since \({}_{\psi_g}L^2(N)\) is invertible, the above condition is equivalent to ${}_{\psi_g^{-1}\widetilde{\alpha}}L^2(N)\in \ctBim(N)$.
For an invertible bimodule ${}_{\theta}L^2(N)$, being centrally trivial is equivalent to $\theta\in \Ct(N)$, i.e., $\theta_\omega=\id_{N_\omega}$.
Hence
\[
{}_{\psi_g^{-1}\widetilde{\alpha}}L^2(N)\in \ctBim(N)
\quad\Longleftrightarrow\quad
(\psi_g^{-1}\widetilde{\alpha})_\omega=\id_{N_\omega}
\]

By Corollary \ref{cor:CSinMrtimesG},
$$N_\omega=(M\rtimes K)_\omega\cong (M_\omega)^K=M_\omega.$$
Under this identification, $\widetilde{\alpha}_\omega$ acts as $\alpha_\omega$, since $\widetilde{\alpha}$
 fixes every \(u_k\), and the induced action of \(g\in G/K\) on \(N_\omega\)
is exactly \((\psi_g)_\omega\). Therefore
\[
(\psi_g^{-1}\widetilde{\alpha})_\omega=\id_{N_\omega}
\quad\Longleftrightarrow\quad
(\psi_g^{-1}\alpha)_\omega=\id_{M_\omega}
\quad\Longleftrightarrow\quad
\alpha_\omega=(\psi_g)_\omega.
\]

By definition, the condition \(\alpha_\omega=(\psi_g)_\omega\) for some \(g\in G/K\) is exactly
the condition \(\alpha\in \GCt(M)\). Hence a liftable invertible approximately inner bimodule
\[
{}_{\widetilde{\alpha}}L^2(N)\in \bigoplus_{g\in G/K}
\bigl({}_{\psi_g}L^2(N)\boxtimes_N \ctBim(N)\bigr)\cap \aiBim(N) \quad \Longleftarrow \alpha\in \GCt(M)\cap \GInn(M).
\]

Thus the correspondence $[{}_{\widetilde{\alpha}}L^2(N)]\mapsto [\alpha]$
induces a bijection
\[
\Inv\left(\bigoplus_{g\in G/K}
\left({}_{\psi_g}L^2(N)\boxtimes_N \ctBim(N)\right)\cap \aiBim(N)\right)_{\mathsf{lift}}
\cong
(\GCt(M)\cap \GInn(M))/\Inn(M). \qedhere
\]
\end{proof}

\section{Example}\label{section:Example}
\subsection{Infinite tensor product of non-Gamma $\rm II_1$ factors}
Suppose $N_i$ are non-Gamma $\rm II_1$ factors, $M=\overline{\bigotimes}_{i=1}^\infty N_i$ is the countable infinite tensor product of $N_i$, which is a McDuff $\rm II_1$ factor.
We denote $M_n:=\overline{\bigotimes}_{i=1}^n N_i$ and $P_n=\overline{\bigotimes}_{i>n} N_i$ for $n\in\bbN$, then $M=M_n\overline{\otimes} P_n$ and $M=\left(\bigcup_{n=1}^\infty M_n\right)''$.

Since $M_n$ is a finite tensor product of non-Gamma $\rm II_1$ factors, $M_n$ is also non-Gamma. 
By \cite[Prop.~3.2.1]{MR2661553} $M_n$ has spectral gap in $M$, i.e.,
for every $\varepsilon>0$, there exist $u_1,\dots,u_n\in U(M_n)$ and $\delta>0$ such that if $a\in M$ satisfies $\|[a,u_i]\|_2\le \delta \|a\|_2$ for every $i=1,\dots,n$, then $\|a-E_{M_n'\cap M}(a)\|_2\le \varepsilon \|a\|_2$ \cite[Def.~2.1]{MR2661553}.

\begin{lem}\label{lem:CsinM}
Let $E_{P_n}=\tr_{M_n}\otimes \id_{P_n}:M\to P_n$ be conditional expectation. 
For $(a_n)\in M^\omega$, $(a_n)$ is a central sequence if and only if $\lim\limits_{n\to\omega} \|E_{P_k}(a_n)-a_n\|_2=0$ for each $k$.
\end{lem}
\begin{proof}
Since $M$ is an inductive limit of subfactors $M_n$, which has spectral gap in $M$. Then by \cite[Lem.~2.2]{MR2661553}, $M_\omega=M'\cap M^\omega=\bigcap_n (M_n'\cap M)^\omega$.
\end{proof}

\begin{lem}\label{lem:homPopa2.4}
Suppose $Q\subseteq P$ is a $\rm II_1$ subfactor. Let $r\ge 1$, $A:=M_r(\bbC)\otimes P$ with normalized trace $\tr_A:=\tr_r\otimes \tr_P$.
Suppose $p\in M_r(\bbC)\otimes (Q'\cap P)$ is a nonzero projection and 
$\alpha:P\to pAp$ is a normal unital $*$-homomorphism.     
Define $\rho:Q\to pAp$ by $\rho(x):=p(1_r\otimes x)$.
It is a well-defined unital $*$-homomorphism since $p$ commutes with $1_r\otimes Q$.

Suppose there is some $0<\delta<1$ such that 
$$\|\alpha(v)-\rho(v)\|_{2,p}\le \delta,\qquad\quad\forall v\in U(Q),$$
where $\|\cdot\|_{2,p}$ is the $L^2$-norm on $pAp$ given by 
$$\tr_{p}(x):=\tr_{pAp}(x)=\tr_A(x)/\tr_A(p).$$
Then there exists a nonzero partial isometry $w\in pAp$ such that 
$$\alpha(x)w=w\rho(x)=wp(1_r\otimes x),\qquad\quad \forall x\in Q.$$
\end{lem}
\begin{proof}
For each $v\in U(Q)$, let $y_v:=\alpha(v)\rho(v)^*$.
Since $\alpha(v)\in pAp$ and $\rho(v)\in pAp$ is unitary, we have $y_v\in pAp$ and
$$\|y_v-p\|_{2,p}=\|\alpha(v)\rho(v)^*-p\|_{2.p} = \|\alpha(v)-\rho(v)\|_{2,p}\le \delta.$$
Let 
$$K_0:=\{y_v\mid v\in U(Q)\}\subseteq pAp,$$
and $K$ be the $\|\cdot\|_{2,p}$-closed convex hull of $K_0$ in the Hilbert space $L^2(pAp,\tr_p)$.
Every element of $K_0$ is in the closed ball of radius $\delta$ around $p$, so is for every element of $K$:
$$\|\eta-p\|_{2,p}\le \delta,\qquad\quad\forall\eta\in K.$$
Since $K$ is a nonempty closed convex subset of a Hilbert space, there exists a unique element $\xi\in K$ with minimal $L^2$-norm.

For $u\in U(Q)$, define $T_u:L^2(pAp,\tr_p)\to L^2(pAp,\tr_p)$ by
$$T_u(\eta):=\alpha(u)\eta\rho(u)^*.$$
This is an isometry since $\alpha(u)$ and $\rho(u)$ are unitaries in $pAp$.
Moreover, $T_u(K_0)=K_0$ since for $v\in U(Q)$,
$$T_u(y_v)=\alpha(u)\alpha(v)\rho(v)^*\rho(u)^*=\alpha(uv)\rho(uv)^*=y_{uv},$$
and hence $T_u(K)=K$.
By the uniqueness of the minimal norm element in $K$, we have 
$$T_u(\xi)=\xi,\qquad\quad\forall u\in U(Q).$$
In other words,
$$\alpha(u)\xi=\xi\rho(u),\qquad\quad \forall u\in U(Q).$$
By linearity and normality, 
$$\alpha(x)\xi=\xi\rho(x),\qquad\quad \forall x\in Q.$$
This $\xi$ is nonzero, since 
$$\|\xi-p\|_{2,p}\ \le \delta<1=\|p\|_{2,p}.$$

Take the polar decomposition of $\xi=w|\xi|$ in $pAp$, where $w\in pAp$ is a nonzero partial isometry. We shall show $w$ also intertwines $\alpha$ and $\rho$.
Note that $\alpha(x)\xi=\xi\rho(x)$ for all $ x\in Q$. 
Take adjoints, we have $\xi^*\alpha(x^*)=\rho(x^*)\xi^*$. 
Replacing $x$ by $x^*$, we have 
$$\xi^*\alpha(x)=\rho(x)\xi^*,\qquad\quad \forall x\in Q.$$
Then 
$$\xi^*\xi\rho(x)=\xi^*\alpha(x)\xi=\rho(x)\xi^*\xi,\qquad\quad \forall x\in Q.$$
Thus $\xi^*\xi\in \rho(Q)'\cap pAp$, which implies 
$|\xi|\in \rho(Q)'\cap pAp$.
Then by the uniqueness of polar decomposition, we have 
\[
\alpha(x) w=w\rho(x),\qquad\quad \forall x\in Q. \qedhere    
\]

\end{proof}

\begin{prop}\label{prop:ctBimX=YPn}
A bifinite bimodule $X\in\ctBim(M)$ if and only if there exist a finite $n$ and $Y\in \Bim(M_n)$ such that $X\cong Y\overline{\otimes} L^2(P_n)$.
\end{prop}
\begin{proof}
By Lemma \ref{lem:CsinM}, it is clear that $Y\overline{\otimes} L^2(P_n)$ is centrally trivial.

Now suppose $X\in \ctBim(M)$. 
The category of bifinite bimodule is semisimple. Suppose $X\cong \bigoplus_{i=1}^t X_i$, $X_i$ are indecomposable $M$-$M$ bimodules. 
We know that $X_i\in \ctBim(M)$, since central triviality passes to direct summand. 
Suppose we know that $X_i\cong Y_i\ootimes L^2(P_{n_i})$, $Y_i\in \Bim(M_{n_i})$ for all $i=1,\dots,t$. 
Let $m=\max_{i=1,\dots,t}\{n_i\}$. 
Since 
$$M=M_{n_i}\overline{\otimes} \left(\ootimes_{i=n_j+1}^m N_i\right)\overline{\otimes} P_m $$
we have $L^2(P_{n_i})\cong L^2(\bigootimes_{i=n_j+1}^m N_i)\ootimes L^2(P_m)$, which follows that 
\begin{align*}
X & \cong \bigoplus_{i=1}^t X_i\cong \bigoplus_{i=1}^t Y_i\ootimes L^2(P_{n_i})\cong \bigoplus_{i=1}^t (Y_i\ootimes L^2(\ootimes_{i=n_j+1}^m N_i))\ootimes L^2(P_m)\\
& \cong \left(\bigoplus_{i=1}^t (Y_i\ootimes L^2(\ootimes_{i=n_j+1}^m N_i))\right)\ootimes L^2(P_m).   
\end{align*}
Therefore, it suffices to prove the proposition for indecomposable bimodules.

Since $X$ is bifinite, $X\cong {}_{\alpha}[p(\bbC^r\otimes L^2(M))]$ for some projection $p\in M_r(\bbC)\otimes M$ and a normal unital $*$-homomorphism $\alpha:M\to p(M_r(\bbC)\otimes M)p$.
This $p$ is not necessarily to commute with any $P_n$. 
To fix this, since $M$ is inductive limit of $M_n$, one can find a sufficiently large $n$ and a projection $q\in M_s(\bbC)\otimes M_n\subseteq M_s(\bbC)\otimes M$ with
$$(\tr_s\otimes \tr_M)(q)=(\tr_r\otimes \tr_M)(p).$$
Then $p$ and $q$ are Murray-von Neumann equivalent in a matrix amplification of $M$.
This equivalence yields to a unitary isomorphisms of the right $M$-modules between $p(\bbC^r\otimes L^2(M))$ and $q(\bbC^s\otimes L^2(M))$.
Conjugating the left action by this right module isomorphism, we can replace the realization $(p,\alpha)$ by $(q,\alpha')$.
Thus,
$$X\cong {}_{\alpha'}[q(\bbC^s\otimes L^2(M))].$$
Since $M_n=P_n'\cap M$, we have 
$$q\in M_s(\bbC)\otimes M_n=M_s(\bbC)\otimes (P_n'\cap M).$$

Since $X$ is centrally trivial, for every central sequence $(a_k)\in M_\omega$,
$$\|a_k\rhd \xi-\xi\lhd a_k\|_2\to 0,\qquad\quad\forall \xi\in X$$
In the ${}_{\alpha'}[q(\bbC^s\otimes L^2(M))]$ model, this is equivalent to 
\begin{align*}
\|\alpha'(a_k)-q(1_s\otimes a_k)\|_{2,q}\to 0,\tag{1}  \label{eq:CTalphamodel}  
\end{align*}
where $\|\cdot\|_{2,q}$ is the $L^2$-norm on $qAq$ and $A=M_s(\bbC)\otimes M$.

By Lemma \ref{lem:CsinM}, a bounded sequence is central if and only if it is asymptotically contained in every $P_m$. 
In particular, any sequence $(v_k)$ is central if $v_k\in U(P_k)$ for each $k$.
Therefore, for a fixed small $\delta>0$, there exists $n$ such that 
$$\|\alpha'(v)-q(1_s\otimes v)\|_{2,q}\le \delta,\qquad\quad\forall v\in U(P_n).$$
Otherwise, for every $k$, one could choose $v_k\in U(P_k)$ with 
$$\|\alpha'(v_k)-q(1_s\otimes v_k)\|_{2,q}> \delta,$$
but $(v_k)$ is central, which contradicts to \eqref{eq:CTalphamodel}.

Now we are able to apply Lemma \ref{lem:homPopa2.4}.
Let $Q=P_n$, $P=M$ and $q\in M_s(\bbC)\otimes (Q'\cap P)$.
Because $q$ commutes with $1_s\otimes P_n$, we have the standard homomorphism 
$$\rho_n:P_n\to q(M_s(\bbC)\otimes M)q, \qquad \rho_n(x)=q(1_s\otimes x)$$
Since $X$ is centrally trivial, for a fixed small $\delta>0$, there exists $n$ such that 
$$\|\alpha'(v)-\rho_n(v)\|_{2,q}\le \delta, \qquad\quad\forall v\in U(P_n).$$
According to Lemma \ref{lem:homPopa2.4}, there exists a nonzero partial isometry $w\in qAq=q(M_s(\bbC)\otimes M)q$ such that 
$$\alpha'(x)w=wq(1_s\otimes x),\qquad\quad\forall x\in Q=P_n$$
Then we have
$$\alpha'(x)w=w\rho_n(x),\qquad\quad\forall x\in P_n$$

We define the algebraic subspace of $P_n$-central vectors in $X$:
$$X_0=\{\xi\in X\mid x\rhd \xi=\xi\lhd x,\ \forall x\in P_n\}.$$
There are several properties for $X_0$:
\begin{itemize}[leftmargin=*]
\item $X_0$ is non-zero: Consider above nonzero partial isometry $w$, for any $x\in P_n$,
$$x\rhd w = \alpha'(x)w=w\rho_n(x) = wq(1_s\otimes x) = w\lhd x,$$
so $w\in X_0$.
\item $X_0\in \Bim(M_n)$: Since $M_n$ strictly commutes with $P_n$ in $M$, the left and right $M_n$-actions preserve the centralizer condition.
\item $X_0$ has a $M_n$-valued inner product: 
For any $\xi,\eta\in X_0$ and $x\in P_n$, the usual $M$-valued inner from $X$ satisfies
$$x\langle \xi\mid\eta\rangle = \langle \xi\lhd x^*\mid\eta\rangle = \langle x^*\rhd\xi\mid\eta\rangle=\langle \xi\mid x\rhd\eta\rangle = \langle \xi\mid\eta\lhd x\rangle =\langle \xi\mid\eta\rangle x.$$
Thus $\langle \xi\mid\eta\rangle\in P_n'\cap M=M_n$.
\end{itemize}
We define a map from the algebraic tensor product $X_0 \odot P_n$ into $X$ by:
$$\Phi(\xi \otimes y) = \xi \lhd y.$$
This map is an isometry because
for $\xi_1,\xi_2\in X$ and $y_1,y_2\in P_n$,
$$\langle \Phi(\xi_1 \otimes y_1), \Phi(\xi_2 \otimes y_2) \rangle_M = \langle \xi_1 \lhd y_1, \xi_2 \lhd y_2 \rangle_M = y_1^* \langle \xi_1, \xi_2 \rangle_M y_2 = \langle \xi_1, \xi_2 \rangle_{M_n} \otimes y_1^* y_2$$
since $\langle \xi_1, \xi_2 \rangle_{M_n}\in M_n$ commutes with $y_1^*\in P_n$.
This $\Phi$ extends uniquely to an isometry $X_0\ootimes L^2(P_n)\to X$.

Now we show that $\Phi$ is an $M$-$M$ bimodule intertwiner.
Let $a\in M_n$ and $b\in P_n$.
\begin{itemize}[leftmargin=*]
\item Right action: 
In the spatial tensor product, the right action of $ab$ is given by $(\xi \otimes y) \lhd (ab) = (\xi \lhd a) \otimes (yb)$.
Applying $\Phi$ to this yields:$$\Phi((\xi \lhd a) \otimes (yb)) = (\xi \lhd a) \lhd (yb) = \xi \lhd (ayb),$$ 
while applying the right action of $ab$ directly to the image in $X$:
$$\Phi(\xi \otimes y) \lhd (ab) = (\xi \lhd y) \lhd (ab) = \xi \lhd (yab).$$
Since $a\in M_n$ and $y\in P_n$ commute, $ayb=yab$. The right actions match.
\item Left action: 
In the spatial tensor product, the left action of $ab$ is given by $(ab) \rhd (\xi \otimes y) = (a \rhd \xi) \otimes (by)$.
Applying $\Phi$ yields:
$$\Phi((a \rhd \xi) \otimes (by)) = (a \rhd \xi) \lhd (by)$$
Now, apply the left action of $ab$ directly to the image in $X$:$$(ab) \rhd \Phi(\xi \otimes y) = a \rhd (b \rhd (\xi \lhd y))$$
Because $X$ is a true bimodule, left and right actions commute, so $b \rhd (\xi \lhd y) = (b \rhd \xi) \lhd y$.
Because $\xi \in X_0$, the left action of $b \in P_n$ equals its right action: $b \rhd \xi = \xi \lhd b$.
Substituting this back in gives:
$$a \rhd ( (\xi \lhd b) \lhd y ) = a \rhd (\xi \lhd (by)) = (a \rhd \xi) \lhd (by).$$
Thus, the left actions also match.
\end{itemize}
Since $\Phi$ is an isometric $M$-$M$ bimodule intertwiner, the image of $X_0\ootimes L^2(P_n)$ is an $M$-$M$ sub-bimodule of $X$.
Because $X$ is assumed to be indecomposable, and $X_0$ is non-empty, we have 
$$X\cong X_0\ootimes L^2(P_n).$$
Since $X$ is bifinite, $X_0$ is also a bifinite bimodule over $M_n$.
\end{proof}

\begin{lem}\label{lem:YQ AI->Y=P}
Suppose $P, Q$ are $\rm II_1$ factors and $P$ is non-Gamma. 
Suppose $Y\in\Bim(P)$ is an indecomposable bifinite bimodule. If $Y\ootimes L^2(Q)$ and its dual are approximately inner in $\Bim(P\ootimes Q)$, then $Y\cong L^2(P)$.
\end{lem}
\begin{proof} 
Suppose $X:=Y\oplus L^2(P)$.
Let $M:=|X\boxtimes_P \overline{X}|$, $P\subseteq M$ is a finite index $\rm II_1$ subfactor. Since $P$ is non-Gamma, $M$ is also non-Gamma.
Since $X\ootimes L^2(Q)\cong (Y\ootimes L^2(Q)) \oplus L^2(P\ootimes Q)$ and its dual are approximately inner in $\Bim(P\ootimes Q)$, so is the tensor product $(X\ootimes L^2(Q))\boxtimes_{P\ootimes Q}(\overline{X\ootimes L^2(Q)})\cong (X\boxtimes_P\overline{X})\ootimes L^2(Q)$, then by \cite[Prop.~4.8]{MR4753059}, $P\ootimes Q\subseteq |(X\boxtimes_P\overline{X})\ootimes L^2(Q)|\cong M\ootimes Q$ is an approximately inner subfactor. 

By \cite[Def.~1.4 \& Prop.~1.2]{MR1317367}, the inclusion $(P\ootimes Q)'\cap (P\ootimes Q)^\omega \subseteq (P\ootimes Q)'\cap (M\ootimes Q)^\omega$ is $[M:P]^{-1}$-Markov, and has finite Pimsner-Popa basis of $(P\ootimes Q)^\omega\subseteq (M\ootimes Q)^\omega$.

Since $P$ is non-Gamma, and $P\subseteq M$ has finite index, by \cite[Prop.~3.2.1]{MR2661553}, $P\ootimes 1\subseteq M\ootimes Q$ has spectral gap, then
$$(P\ootimes 1)'\cap (M\ootimes Q)^\omega = ((P\ootimes 1)'\cap (M\ootimes Q))^\omega = ((P'\cap M)\ootimes Q)^\omega.$$

Furthermore, since $P\ootimes Q=(P\ootimes 1)\vee (1\ootimes Q)$, we have 
\begin{align*}
(P\ootimes Q)'\cap (M\ootimes Q)^\omega & = ((P\ootimes 1)'\cap (M\ootimes Q)^\omega)\cap ((1\ootimes Q)'\cap (M\ootimes Q)^\omega) \\
& = ((P'\cap M)\ootimes Q)^\omega \cap ((1\ootimes Q)'\cap (M\ootimes Q)^\omega) \\
& = (1\ootimes Q)'\cap ((P'\cap M)\ootimes Q)^\omega,
\end{align*}
since $((P'\cap M)\ootimes Q)^\omega\subseteq (M\ootimes Q)^\omega$.
Thus, there is a finite Pimsner-Popa basis $\{m_i\}$ for $(P\ootimes Q)^\omega\subseteq (M\ootimes Q)^\omega$ with $m_i=\{b_i^{(n)}\}$, where each $b_i^{(n)}\in (P'\cap M)\ootimes Q$.

Then for each $x\in M$, $x\otimes 1\in M\ootimes Q\subseteq (M\ootimes Q)^\omega$, since $m_i=\{b_i^{(n)}\}$ is a Pimsner-Popa basis, we have
$$x\otimes 1=\sum_i m_i E_{(P\ootimes Q)^\omega}(m_i^*(x\otimes 1)),$$
In other words,
$$\left\|x\otimes 1-\sum_i b_i^{(n)} E_{P\ootimes Q}\left((b_i^{(n)})^*(x\ootimes 1)\right)\right\|_2\to 0$$
Note that $E_{P\ootimes Q}=E_P\otimes \id_Q$, because the inclusion $P\ootimes Q\subseteq M\ootimes Q$ is given by $(P\subseteq M)\ootimes Q$.
Apply the contractive map $\id_P\otimes \tr_Q:L^2(M\ootimes Q)\to L^2(M)$, we have
$$\left\|x-\sum_i(\id_P\otimes \tr_Q)\left(b_i^{(n)} (E_P\otimes \id_Q)\left((b_i^{(n)})^*(x\ootimes 1)\right)\right) \right\|_2\to 0.$$
Since $P\subseteq M$ is of finite index, $P'\cap M$ is finite-dimensional. Choose a basis $v_1,\dots, v_r$ of $P'\cap M$, we can write 
$$b_i^{(n)}=\sum_{j=1}^r v_j\otimes q_j,$$
for some $q_j\in Q$.
Hence $(b_i^{(n)})^*(x\ootimes 1) = \sum_{j=1}^r v_r^*x\otimes q_j^*$ and 
$$b_i^{(n)} (E_P\otimes \id_Q)\left((b_i^{(n)})^*(x\ootimes 1)\right) = \sum_{k=1}^r v_k\otimes q_k\sum_{j=1}^r E_P(v_j^* x)\otimes q_j^* = \sum_{j,k=1}^r v_k E_P(v_j^* x) \otimes q_kq_j^*  $$
Applying $\id_P\otimes \tr_Q$, we obtain
$$(\id_P\otimes \tr_Q)\left(b_i^{(n)} (E_P\otimes \id_Q)\left((b_i^{(n)})^*(x\ootimes 1)\right)\right) = \sum_{j,k=1}^r \tr_Q(q_kq_j^*) v_kE_P(v_j^* x)$$
Each $v_k\in P'\cap M$ and each $E_P(v_j^* x)\in P$, so this belongs to $P\vee (P'\cap M)$.
Therefore, $x$ is in the $L^2$-closure of $P\vee (P'\cap M)$, which is $P\vee (P'\cap M)$.
This proves that $M=P\vee (P'\cap M)$.

According to \cite[Prop.~4.14]{MR4753059}, this implies that $L^2(M)\cong \bigoplus_i L^2(P)$ as $P$-$P$ bimodules.
On the other hand, $L^2(M)\cong X\boxtimes_P \overline{X}\cong L^2(P)\oplus Y\oplus \overline{Y}\oplus (Y\boxtimes_P \overline{Y})$.
Since $Y$ is indecomposable and appeared in the direct summand of $L^2(M)$, $Y\cong L^2(P)$.
\end{proof}

\begin{prop}\label{prop:ecomponentTrivial}
$\Chi(M)\cong \Hilb$.
\end{prop}
\begin{proof}
Suppose $X\in \Chi(M)$ is indecomposable. 
Since $X$ is centrally trivial, by Proposition \ref{prop:ctBimX=YPn}, $X\cong Y\ootimes L^2(P_n)$, where $Y\cong \Bim(M_n)$. Since $X$ is indecomposable, $Y$ is also indecomposable.
By the definition of $\Chi(M)$, it is the dualizable part of $\aiBim(M)\cap \ctBim(M)$, so $X$ and its dual are both approximately inner. 
Note that $M_n$ is non-Gamma, by Lemma \ref{lem:YQ AI->Y=P}, $Y\cong L^2(M_n)$. 
Therefore, $X\cong L^2(M_n)\ootimes L^2(P_n)\cong L^2(M)$, and hence $\Chi(M)\cong \Hilb$. 
\end{proof}

\begin{rem}
Recall a $\rm II_1$ factor $M$ is s-McDuff if $M\cong N\ootimes R$ for some non-Gamma $\rm II_1$ factor $N$.
Connes asked when $M$ is McDuff, whether it is true that $M$ is s-McDuff if and only if $\chi(M)$ is trivial.
Popa gives a negative answer to Connes' question by constructing the example of countable infinite spatial tensor product of non-Gamma $\rm II_1$ factors $M=\ootimes_i N_i$ which is McDuff but not s-McDuff \cite{MR2661553}. 

In \cite[Ex.~4.16]{MR4753059}, authors prove that the categorical $\Chi$ of s-McDuff $\rm II_1$ factor is always trivial, so they asked whether Connes' conjecture holds for the categorical $\Chi(M)$. 
Here we show that $M$ from Popa's construction also has trivial  $\Chi(M)$.
Therefore, this also gives an negative answer their question.
\end{rem}

\subsection{$\Chi(M\rtimes G)$ for entry-wise action}

\begin{fact}\label{fact:Aiautospect}
Let $M$ be a separable finite factor and $\alpha\in \Aut(M)$.
In the injective case, Connes showed that approximate innerness can be characterized by the presence of almost central vectors in the associated bimodule \cite[Thm.~3.1]{MR454659}.
Mingo later removes the injectivity assumption: he proves that, for arbitrary separable finite von Neumann algebras, an automorphism is approximately inner if and only if it is approximately inner as a completely positive map, and that this is equivalent to weak containment of the identity correspondence in the correspondence associated with $\alpha$ \cite{MR1040957}.
Therefore,
$$ \alpha\not\in\Ai(M) \qquad\Longleftrightarrow\qquad L^2(M)\npreceq {}_{\alpha}L^2(M).$$
Since weak containment of $L^2(M)$ in a correspondence is equivalent to the existence of almost central unit vectors, its failure is equivalent, by a standard contrapositive argument, to the existence of a finite set $F\subseteq M$ and a constant $C>0$ such that 
$$\|x\|_2^2\le C\sum_{a\in F}\|\beta(a)x-xa\|_2^2.\qquad\forall x\in L^2(M).$$
\end{fact}

\begin{prop}\label{prop:aiXY=aiXaiY}
Suppose $P\overline{\otimes} Q$ is tensor product of two $\rm II_1$ factors. 
Suppose $X\in \Bim(P)$ and $\alpha\in \Aut(Q)$. Then $X\ootimes {}_{\alpha}L^2(Q)\in \aiBim(P\overline{\otimes}Q)$ implies $\alpha\in \Ai(Q)$.
\end{prop}
\begin{proof}
We prove the contrapositive. 
Denote $H={}_{\alpha}L^2(Q)$. 
For each $a\in F$, define the bounded linear map $T_a:H\to H$ by 
$$T_a(\xi)=a\rhd \xi-\xi\lhd a=\alpha(a)\xi-\xi a.$$
Suppose $\alpha\in \Aut(Q)$ is not approximately inner, by Fact \ref{fact:Aiautospect}, then there exist $C>0$ and finite subset $F\subseteq Q$ such that 
$$\|\xi\|_2^2\le C\sum_{a\in F}\|\alpha(a)\xi-\xi a\|_2^2 = C\sum_{a\in F}\|T_a(\xi)\|_2^2$$
for all $\xi\in H$.
In terms of $\bbC$-valued inner products, this means
$$\langle \xi\mid\xi\rangle\le C\sum_{a\in F}\langle T_a\xi\mid T_a\xi\rangle=\left\langle\xi\, \middle|\, C\sum_{a\in F}T_a^*T_a \xi\right\rangle,$$
which is equivalent to say
$$\id_{H}\le C\sum_{a\in F}T_a^*T_a$$
as positive operators on $H$.

Now tensor this inequality with $\id_X$. On the Hilbert space $X\ootimes H$, we have
$$\id_X\otimes \id_H\le \sum_{a\in F}(\id_X\otimes T_a)^*(\id_X\otimes T_a),$$
therefore for every $\eta\in X\ootimes H$,
$$\|\eta\|_2^2\le C\sum_{a\in F}\|(\id_X\otimes T_a)\eta\|_2^2 = C\sum_{a\in F}\|(1\otimes\alpha(a))\eta-\eta(1\otimes a)\|_2^2.$$

Suppose $X\ootimes H\in\Bim(P\overline{\otimes}Q)$ is approximately inner, then there exists an approximately inner $(X\ootimes H)_{P\overline{\otimes}Q}$ basis $\{y_i^{(n)}\}$.
For all $1\otimes a\in P\overline{\otimes}Q$, 
$$\|(1\otimes a)\rhd y_i^{(n)}-y_i^{(n)}\lhd (1\otimes a)\|_2\to 0.$$
Then we have 
$$\|y_i^{(n)}\|_2^2\le C\sum_{a\in F}\|(1\otimes\alpha(a))\eta-\eta(1\otimes a)\|_2^2 \to 0.$$
This implies that 
$$\sum_{i}|y_i^{(n)}\rangle \langle y_i^{(n)}| \to 0\qquad\mathrm{strongly},$$
since $\{y_i^{(n)}\}$ is uniformly bounded in the module norm.
However, as an approximate basis,
$$\sum_{i}|y_i^{(n)}\rangle \langle y_i^{(n)}|\to \id_{X}\otimes \id_{H}\qquad\mathrm{strongly},$$
which is a contradiction.
\end{proof}

\begin{construction}\label{const:M=otimes N,psi=otimes alpha}
Suppose $G$ is a finite group.
Popa and Shlyakhtenko
showed that every subfactor standard invariant can be realized as an $LF_\infty$ subfactor \cite{MR2051399}.
Guionnet-Jones-Shlyakhtenko provide an alternative
realization of finite depth standard invariants as inclusions of interpolated free group factors \cite{MR2732052}.
In other words, there exists a non-Gamma $\rm II_1$ factor $N$ and an outer action $\alpha:G\to \Aut(N)$.

Let $M=\bigootimes_{i=1}^\infty N$ to be the countable infinite spatial tensor product of $N$.
We define $\psi:G\to \Aut(M)$ by $\psi_g:=\bigotimes_{i=1}^\infty \alpha_g$.    
\end{construction}

\begin{prop}\label{prop:psigonMnotCT}
The group action $\psi$ in Construction \ref{const:M=otimes N,psi=otimes alpha} is not centrally trivial.
\end{prop}
\begin{proof}
Fix $g\ne e$. Since 
$\alpha_g$ is outer, we choose $a\in N$ such that $\|\alpha_g(a)-a\|_2=\epsilon>0$
For each $n$, let 
$$a_n=1\otimes \cdots \otimes 1\otimes a\otimes 1\otimes\cdots\in P_{n-1}\subseteq M$$
with $a$ at the $n$-th tensor.
Then $(a_n)$ is a central sequence by Lemma \ref{lem:CsinM}. 
Note that
$$\psi_g(a_n)=1\otimes \cdots \otimes 1\otimes \alpha_g(a)\otimes 1\otimes\cdots,$$
so 
$$\|\psi_g(a_n)-a_n\|_2=\|\alpha_g(a)-a\|_2 = \epsilon\ne 0.$$
Therefore, $\psi$ is not centrally trivial.
\end{proof}

\begin{prop}\label{prop:psigonMnotAI}
The group action $\psi$ in Construction \ref{const:M=otimes N,psi=otimes alpha} is not approximately inner.
\end{prop}
\begin{proof}
We know that $M=N\ootimes P_1$ and $\psi_g=\alpha_g\otimes \beta_g$, where $\beta_g\in \Aut(P_1)$.
Suppose $\psi_g$ is approximately inner, then ${}_{\psi_g}L^2(M)={}_{\alpha_g}L^2(N)\otimes {}_{\beta_g}L^2(P_1)$ is approximately inner. 
By Proposition \ref{prop:aiXY=aiXaiY}, this implies $\alpha_g\in \Aut(N)$ is approximately inner. However, since $N$ is non-Gamma so that every approximately inner automorphism must be inner, which contradicts to the construction that $\alpha_g$ is outer when $g\ne e$.
\end{proof}

\begin{prop}\label{prop:gcomponentEmpty}
Suppose $X\in \ctBim(M)$, ${}_{\psi_g}L^2(M)\boxtimes_M X$ is not approximately inner when $g\ne e$.
\end{prop}
\begin{proof}
By Proposition \ref{prop:ctBimX=YPn}, any centrally trivial bimodule $X\cong Y\ootimes L^2(P_n)$ for some $Y\in \Bim(M_n)$. 
Note that $\psi_g=\sigma_g\otimes \theta_g$, where $\sigma_g=\ootimes_{i=1}^n \alpha_g\in \Aut(M_n)$ and $\theta_g=\ootimes_{i>n}\alpha_g\in \Aut(P_n)$, then
$${}_{\psi_g}L^2(M)\boxtimes_M X\cong ({}_{\sigma_g}L^2(M_n)\boxtimes_{M_n} Y)\ootimes ({}_{\theta_g}L^2(P_n))$$
Suppose this is approximately inner, by Proposition \ref{prop:aiXY=aiXaiY}, it implies that $\theta_g\in \Aut(P_n)$ is approximately inner. 
However, $\theta_g$ on $P_n$ is given by the same way as $\psi_g$ on $M$, by Proposition \ref{prop:psigonMnotAI}, $\theta_g$ is not approximately inner for all $g\ne e$, which is a contradiction.
\end{proof}

For the given action $\psi:G\to \Aut(M)$, according to Proposition \ref{prop:psigonMnotCT} and \ref{prop:psigonMnotAI}, $K=\psi^{-1}(\Ct(M))\cap G$ and $L=\psi^{-1}(\Ai(M))\cap G$ are trivial. 
By Theorem \ref{thm:G/KdeequivofChi(MG)},
we have $G$-equivariantization of $\Chi(M\rtimes G)$,
$$\Chi(M\rtimes G)_{G}\cong \bigoplus_{g\in G} ({}_{\psi_g}L^2(M)\boxtimes_{M} \ctBim(M))\cap \aiBim(M)$$
as unitary $G$-crossed braided tensor categories.
By Proposition \ref{prop:ecomponentTrivial}, we know the $g=e$ component is equivalent to $\Hilb$; and by Proposition \ref{prop:gcomponentEmpty}, the $g\ne e$ component is empty. 
Therefore,
$$\Chi(M\rtimes G)_{G}\cong \Hilb$$
as unitary $G$-crossed braided tensor categories.

Applying $G$-equivariantization on $\Chi(M\rtimes G)_G$, we have
$$\Chi(M\rtimes G)\cong (\Chi(M\rtimes G)_G)^G\cong \Hilb^G\cong \Rep(G).$$

We can summarize above results into the following theorem.
\begin{thm}
Suppose $G$ is a finite group, $N$ is a non-Gamma $\rm II_1$ factor and $\alpha:G\to \Aut(N)$ is an outer action. Let $M=\bigootimes_{i=1}^\infty N$ be the countable infinite spatial tensor product of $N$ and $\psi:G\to \Aut(M)$ is given by $\psi_g:=\bigotimes_{i=1}^\infty \alpha_g$. 
Then $\Chi(M\rtimes_{\psi} G)\cong \Rep(G)$ as unitary braided tensor categories.
\end{thm}

\begin{cor}
For any finite group $G$, its unitary representation category $\Rep(G)$ can be realized as $\Chi(M)$ for some McDuff $\rm II_1$ factor $M$.
\end{cor}

Note that $\Rep(G)$ is a symmetric braided category, which is not modular.
In particular, it is not the Drinfeld center of a unitary fusion category.

\begin{cor}
There exist McDuff $\rm II_1$ factors $M$ whose $\Chi(M)$ is not modular.
\end{cor}

\begin{fact}\label{fact:A5A6}
Let $A_n$ be the alternating group on $n$ letters.   
The groups of characters $\widehat{A_5}\cong \widehat{A_6}$ are trivial while the categories of unitary representations $\Rep(A_5)\ncong \Rep(A_6)$. 
\end{fact}

\begin{prop}
For any finite abelian group $A$, there exist finite groups $G$ and $H$ such that $\widehat{G}\cong \widehat{H}\cong A$ but $\Rep(G)\ncong \Rep(H)$.
\end{prop}
\begin{proof}
Let $A$ be an arbitrary finite abelian group. Let $G:=A\times A_5$ and $H:=A\times A_6$.
Note that $\Hom(-\to U(1))$ preserves finite direct product, we have 
$$\widehat{G} = \widehat{A \times A_5} \cong \widehat{A} \times \widehat{A_5}\cong A\times \{e\}\cong A\cong \widehat{H}.$$
It is known that for finite groups $G_1$ and $G_2$,
$$\Rep(G_1\times G_2)\cong \Rep(G_1)\boxtimes \Rep(G_2),$$
as unitary braided tensor categories, where $\boxtimes$ is the Deligne tensor product.
Combine this with Fact \ref{fact:A5A6},
\[\Rep(G)\cong \Rep(A_5)\boxtimes \Rep(A)\ncong \Rep(A_6)\boxtimes \Rep(A)\cong \Rep(H). \qedhere
\]
\end{proof}

\begin{cor}
For any finite abelian group $A$,
there exist McDuff $\rm II_1$ factors $M$ and $N$ such that $\chi(M)\cong \chi(N)\cong A$, but $\Chi(M)\ncong \Chi(N)$.
\end{cor}

\bibliographystyle{alpha}
{\footnotesize{
\bibliography{bibliography}}}

\end{document}